\documentclass[12p]{amsart}
\usepackage{amssymb}
\usepackage{amsmath}
\usepackage{amsfonts}
\usepackage{geometry}
\usepackage{graphicx}
\usepackage{mathrsfs,amssymb}

\usepackage{hyperref}
\usepackage{cleveref}

\theoremstyle{plain}

\newtheorem{corollary}{Corollary}

\newtheorem{lemma}{Lemma}

\newtheorem{proposition}{Proposition}
\newtheorem{remark}{Remark}

\newtheorem{theorem}{Theorem}
\numberwithin{equation}{section}

\begin{document}
\title[]{A determination of the blowup solutions to the focusing NLS with mass equal to the mass of the soliton}

\author{Benjamin Dodson}

\begin{abstract}
In this paper we prove rigidity for blowup solutions to the focusing, mass-critical nonlinear Schr{\"o}dinger equation in dimensions $2 \leq d \leq 15$ with mass equal to the mass of the soliton. We prove that the only such solutions are the solitons and the pseudoconformal transformation of the solitons. We show that this implies a Liouville result for the nonlinear Schr{\"o}dinger equation.

\end{abstract}
\maketitle

\section{Introduction}
In this paper we continue the study begun in \cite{dodson2021determination} of the focusing, mass-critical nonlinear Schr{\"o}dinger equation
\begin{equation}\label{1.1}
i u_{t} + \Delta u + |u|^{\frac{4}{d}} u = 0, \qquad u(0,x) = u_{0}(x) \in L^{2}(\mathbb{R}^{d}).
\end{equation}
In \cite{dodson2021determination} we proved a rigidity result for solutions to $(\ref{1.1})$ in one dimension with mass equal to the mass of the soliton. In this paper we address higher dimensions.

In general, the Hamiltonian equation
\begin{equation}\label{1.1.1}
i u_{t} + \Delta u + |u|^{p - 1} u = 0,
\end{equation}
has the scaling symmetry
\begin{equation}\label{1.1.2}
u(t,x) \mapsto v(t,x) = \lambda^{\frac{2}{p - 1}} u(\lambda^{2} t, \lambda x).
\end{equation}
That is, if $u$ solves $(\ref{1.1.1})$, then $v$ solves $(\ref{1.1.1})$ for any $\lambda > 0$, where $v$ is given by $(\ref{1.1.2})$. In particular, $(\ref{1.1})$ is called $L^{2}$-critical or mass-critical, since if $u$ solves $(\ref{1.1})$, then
\begin{equation}\label{1.5.1}
v(t,x) = \lambda^{d/2} u(\lambda^{2} t, \lambda x)
\end{equation}
is also a solution to $(\ref{1.1})$ with initial data $v_{0} = \lambda^{d/2} u_{0}(\lambda x)$. A change of variables calculation verifies that $\| v_{0} \|_{L^{2}} = \| u_{0} \|_{L^{2}}$.\medskip

As in one dimension, the scaling symmetry $(\ref{1.5.1})$ completely controls the local well-posedness theory of $(\ref{1.1})$. Indeed, \cite{cazenave1990cauchy} proved
\begin{theorem}\label{t1.1}
The initial value problem $(\ref{1.1})$ is locally well-posed for any $u_{0} \in L^{2}$.
\begin{enumerate}
\item For any $u_{0} \in L^{2}$ there exists $T(u_{0}) > 0$ such that $(\ref{1.1})$ is locally well-posed on the interval $(-T, T)$.

\item If $\| u_{0} \|_{L^{2}}$ is small then $(\ref{1.1})$ is globally well-posed, and the solution scatters both forward and backward in time. That is, there exist $u_{+}$, $u_{-} \in L^{2}(\mathbb{R}^{d})$ such that
\begin{equation}\label{1.5.2}
\lim_{t \nearrow +\infty} \| u(t) - e^{it \Delta} u_{+} \|_{L^{2}} = 0, \qquad \lim_{t \searrow -\infty} \| u(t) -  e^{it \Delta} u_{-} \|_{L^{2}} = 0.
\end{equation}

\item If $I$ is the maximal interval of existence for a solution to $(\ref{1.1})$ with initial data $u_{0}$, $u$ is said to blow up forward in time if
\begin{equation}\label{1.6.1}
\lim_{T \nearrow \sup(I)} \| u \|_{L_{t,x}^{\frac{2(d + 2)}{d}}([0, T] \times \mathbb{R}^{d})} = +\infty.
\end{equation}
If $u$ does not blow up forward in time, then $\sup(I) = +\infty$ and $u$ scatters forward in time.

\item If $\sup(I) < \infty$, then for any $s > 0$,
\begin{equation}\label{1.7.1}
\lim_{t \nearrow \sup(I)} \| u(t) \|_{H^{s}} = +\infty.
\end{equation}

\item Time reversal symmetry implies that the results corresponding to $(3)$ and $(4)$ also hold going backward in time.
\end{enumerate}
\end{theorem}
\begin{proof}
The proof in \cite{cazenave1990cauchy} combines Strichartz estimates (see \cite{ginibre1992smoothing}, \cite{keel1998endpoint}, \cite{strichartz1977restrictions},\cite{yajima1987existence}) with Picard iteration. See Section $1.3$ of \cite{dodson2019defocusing} or Section $3.3$ of \cite{tao2006nonlinear} for a detailed proof. See also \cite{ginibre1979class}, \cite{ginibre1979class2}, \cite{ginibre1985global}, and \cite{kato1987nonlinear}.
\end{proof}
\noindent Theorem $\ref{t1.1}$ also holds for the defocusing, nonlinear Schr{\"o}dinger equation,
\begin{equation}\label{1.7.2}
i u_{t} + \Delta u - |u|^{4/d} u = 0, \qquad u(0,x) = u_{0}(x) \in L^{2}(\mathbb{R}^{d}),
\end{equation}
and the proof is identical. 

However, the global theory for $(\ref{1.1})$ with large data differs substantially from the global theory for $(\ref{1.7.2})$ with large data. Observe that the equation
\begin{equation}\label{1.7.3}
i u_{t} + \Delta u - \mu |u|^{p - 1} u = 0,
\end{equation}
is the Hamiltonian equation for the Hamiltonian
\begin{equation}\label{1.7.4}
E(u(t)) = \frac{1}{2} \| \nabla u(t) \|_{L^{2}}^{2} + \frac{\mu}{p + 1} \| u(t) \|_{L^{p + 1}}^{p + 1} = E(u(0)).
\end{equation}
Both $(\ref{1.1})$ and $(\ref{1.7.2})$ conserve the mass
\begin{equation}\label{1.3.0}
M(u(t)) = \int |u(t,x)|^{2} dx = M(u(0)).
\end{equation}
When $\mu = +1$, as in $(\ref{1.7.2})$, the energy $(\ref{1.7.4})$ is positive definite, so for $u_{0} \in H^{1}(\mathbb{R}^{d})$, conservation of energy guarantees a uniform bound on $\| u(t) \|_{H^{1}(\mathbb{R}^{d})}$, which by $(\ref{1.7.1})$ guarantees that the solution $u$ to $(\ref{1.7.2})$ with initial data $u_{0} \in H^{1}$ is global. Global well-posedness and scattering for $(\ref{1.7.2})$ for general $u_{0} \in L^{2}$ was proved in \cite{dodson2012global}, \cite{dodson2016global}, and \cite{dodson2016global2}. When $\mu = -1$, the energy is given by
\begin{equation}\label{1.4}
E(u(t)) = \frac{1}{2} \int |\nabla u(t,x)|^{2} dx - \frac{d}{2(d + 2)} \int |u(t,x)|^{\frac{2(d + 2)}{d}} dx.
\end{equation}
The most that $(\ref{1.3.0})$ and $(\ref{1.4})$ guarantee is a uniform bound on $\| \nabla u(t) \|_{L^{2}}$ for $\| u_{0} \|_{L^{2}}$ below a threshold mass. Indeed, in two dimensions, a straightforward application of the fundamental theorem of calculus and H{\"o}lder's inequality implies
\begin{equation}\label{1.8}
\aligned
\int |u(t,x,y)|^{4} dx dy \leq \int |u(t,x,y)|^{2} (\int |u_{y}(t,x,s_{2})|^{2} ds_{2})^{1/2} (\int |u(t,x,s_{2})|^{2} ds_{2})^{1/2} dx dy \\
\leq \int \int (\int |u_{x}(t, s_{1}, y)|^{2} ds_{1})^{1/2} (\int |u(t,s_{1}, y)|^{2} ds_{1})^{1/2} (\int |u_{y}(t,x,s_{2})|^{2} ds_{2})^{1/2} (\int |u(t,x,s_{2})|^{2} ds_{2})^{1/2} dx dy \\
\leq \| \nabla u \|_{L^{2}}^{2} \| u \|_{L^{2}}^{2}.
\endaligned
\end{equation}
Therefore, there exists a threshhold mass $M_{0}$ for which, if $\| u_{0} \|_{L^{2}} < M_{0}$,
\begin{equation}\label{1.10}
E(u(t)) \gtrsim_{M_{0}} \| u(t) \|_{\dot{H}^{1}(\mathbb{R}^{2})}^{2},
\end{equation}
with implicit constant $\searrow 0$ as $\| u_{0} \|_{L^{2}} \nearrow M_{0}$. In higher dimensions, by the Sobolev embedding theorem,
\begin{equation}\label{1.10.1}
\| u \|_{L_{x}^{\frac{2(d + 2)}{d}}(\mathbb{R}^{d})}^{\frac{2(d + 2)}{d}} \lesssim \| u \|_{L^{2}}^{\frac{4}{d}} \| \nabla u \|_{L^{2}}^{2},
\end{equation}
so then $(\ref{1.10.1})$ holds in dimensions $d \geq 3$ for some $M_{0}$ that may depend on $d$.
\begin{remark}
See the introduction to \cite{dodson2021determination} and references therein for more information on blowup for mass-subcritical and mass-supercritical results.
\end{remark}

From \cite{weinstein1983nonlinear}, the optimal constant in $(\ref{1.8})$ and $(\ref{1.10.1})$ is given by the Gagliardo--Nirenberg inequality,
\begin{equation}\label{1.11}
\| u \|_{L^{\frac{2(d + 2)}{d}}(\mathbb{R}^{d})}^{\frac{2(d + 2)}{d}} \leq \frac{d + 2}{d} (\frac{\| u \|_{L^{2}}^{4/d}}{\| Q \|_{L^{2}}^{4/d}}) \| \nabla u \|_{L^{2}}^{2},
\end{equation}
where $Q$ is the unique positive, radial solution to 
\begin{equation}\label{1.12}
\Delta Q + Q^{1 + \frac{4}{d}} = Q.
\end{equation}
\begin{remark}
The unique positive, radial solution to $(\ref{1.12})$ is called the ground state. See \cite{MR512091}, \cite{berestycki1981ode}, \cite{kwong1989uniqueness}, and \cite{strauss1977existence} for the existence and uniqueness of a ground state solution to $(\ref{1.12})$ in general dimensions.
\end{remark}
\noindent Thus, $(\ref{1.4})$ implies that $(\ref{1.1})$ with initial data $u_{0} \in H^{1}$ and $\| u_{0} \|_{L^{2}} < \| Q \|_{L^{2}}$ is globally well-posed. Global well-posedness and scattering for $(\ref{1.1})$ with a general $\| u_{0} \|_{L^{2}} < \| Q \|_{L^{2}}$ was proved in \cite{dodson2015global}.\medskip

It is straightforward to see that $u(t,x) = e^{it} Q(x)$ solves $(\ref{1.1})$, which gives a solution to $(\ref{1.1})$ with mass $\| u_{0} \|_{L^{2}} = \| Q \|_{L^{2}}$ that blows up both forward and backward in time (according to $(\ref{1.6.1})$). Making a pseudoconformal transformation of $e^{it} Q(x)$,
\begin{equation}\label{1.15}
u(t,x) = \frac{1}{t^{d/2}} e^{\frac{i}{t} + \frac{i |x|^{2}}{4t}} Q(\frac{x}{t}),
\end{equation}
is a solution to $(\ref{1.1})$ that blows up as $t \searrow 0$. Observe that $\| \nabla u(t) \|_{L^{2}} \nearrow \infty$ as $t \searrow 0$ for $(\ref{1.15})$.\medskip

The soliton solutions and the pseudoconformal transformations of the solitons are the only solutions to $(\ref{1.1})$ with mass $\| u_{0} \|_{L^{2}} = \| Q \|_{L^{2}}$. Previously, in \cite{dodson2021determination}, we proved
\begin{theorem}\label{t1.2}
In dimension $d = 1$, the only solutions to $(\ref{1.1})$ with mass $\| u_{0} \|_{L^{2}} = \| Q \|_{L^{2}}$ that blow up forward in time are the family of soliton solutions
\begin{equation}\label{1.2}
e^{-i \theta - it \xi_{0}^{2}} e^{i \lambda^{2} t} e^{ix \xi_{0}} \lambda^{1/2} Q(\lambda (x - 2 t \xi_{0}) + x_{0}), \qquad \lambda > 0, \qquad \theta \in \mathbb{R}, \qquad x_{0} \in \mathbb{R}, \qquad \xi_{0} \in \mathbb{R},
\end{equation}
and the pseudoconformal transformation of the family of solitons,
\begin{equation}\label{1.3}
\aligned
\frac{\lambda^{1/2}}{(T - t)^{1/2}} e^{i \theta} e^{\frac{i (x - \xi_{0})^{2}}{4(t - T)}} e^{i \frac{\lambda^{2}}{t - T}} Q(\frac{\lambda (x - \xi_{0}) - (T - t) x_{0}}{T - t}), \\ \text{where} \qquad \lambda > 0, \qquad \theta \in \mathbb{R}, \qquad x_{0} \in \mathbb{R}, \qquad \xi_{0} \in \mathbb{R}, \qquad T \in \mathbb{R}, \qquad t < T.
\endaligned
\end{equation}
\end{theorem}
\noindent In this paper we prove the same result in dimensions $2 \leq d \leq 15$.
\begin{theorem}\label{t1.3}
In dimensions $2 \leq d \leq 15$, the only solutions to $(\ref{1.1})$ with mass $\| u_{0} \|_{L^{2}} = \| Q \|_{L^{2}}$ that blow up forward in time are the family of soliton solutions
\begin{equation}\label{1.3.1}
e^{-i \theta - it |\xi_{0}|^{2}} e^{i \lambda^{2} t} e^{ix \cdot \xi_{0}} \lambda^{d/2} Q(\lambda (x - 2 t \xi_{0}) + x_{0}), \qquad \lambda > 0, \qquad \theta \in \mathbb{R}, \qquad x_{0} \in \mathbb{R}^{d}, \qquad \xi_{0} \in \mathbb{R}^{d},
\end{equation}
and the pseudoconformal transformation of the family of solitons,
\begin{equation}\label{1.5}
\aligned
\frac{\lambda^{d/2}}{(T - t)^{d/2}} e^{i \theta} e^{\frac{i |x - \xi_{0}|^{2}}{4(t - T)}} e^{i \frac{\lambda^{2}}{t - T}} Q(\frac{\lambda (x - \xi_{0}) - (T - t) x_{0}}{T - t}), \\ \text{where} \qquad \lambda > 0, \qquad \theta \in \mathbb{R}, \qquad x_{0} \in \mathbb{R}^{d}, \qquad \xi_{0} \in \mathbb{R}^{d}, \qquad T \in \mathbb{R}, \qquad t < T.
\endaligned
\end{equation}
Applying time reversal symmetry to $(\ref{1.1})$, this theorem completely settles the question of qualitative behavior of solutions to $(\ref{1.1})$ for initial data satisfying $\| u_{0} \|_{L^{2}} = \| Q \|_{L^{2}}$.
\end{theorem}
\begin{remark}
The obstruction to proving Theorem $\ref{t1.3}$ in dimensions $d \geq 16$ appears to be a purely technical obstruction. The issue arises only in section ten, and will be discussed in more detail there.
\end{remark}

The proof of Theorem $\ref{t1.3}$ relies heavily on the virial identity
\begin{equation}
\frac{d}{dt} \int x \cdot Im[\bar{u} \nabla u](t,x) dx = 4 E(u(t)).
\end{equation}
Using the Pohozaev identity,
\begin{equation}\label{1.16}
E(Q) = \frac{1}{2} \int (Q - \Delta Q - Q^{1 + 4/d})(\frac{d}{2} Q + x \cdot \nabla Q) dx = 0.
\end{equation} 
Thus, by $(\ref{1.11})$, $Q$ is a minimizer of the energy when $\| u \|_{L^{2}} = \| Q \|_{L^{2}}$. In fact, up to scaling, $Q$ is the unique minimizer of the energy (see  \cite{weinstein1983nonlinear}). So when $\| u \|_{L^{2}} = \| Q \|_{L^{2}}$, the energy $E(u)$ in $(\ref{1.4})$ gives a good measurement for the distance from $u$ to the set
\begin{equation}
\{ e^{i \theta} \lambda^{d/2} Q(\lambda x + x_{0}) : \lambda > 0, \qquad x_{0} \in \mathbb{R}^{d}, \qquad \theta \in \mathbb{R} \}.
\end{equation}
However, $E(u)$ is not invariant under the scaling symmetry $(\ref{1.5.1})$, so this notion will not be made precise until later. It will also be necessary to account for the Galilean transformation, which does change the energy.\medskip

The proof of Theorem $\ref{t1.3}$ will occupy most of the paper, and will follow the argument in \cite{dodson2021determination}. There are many places where the argument is exactly the same, and in those places the argument will often be abbreviated, and the reader will be referred to \cite{dodson2021determination} for more details. There are other places where the argument is much more technically difficult, especially in dimensions $d \geq 3$. This is due to the fact that $F(x) = |x|^{4/d} x$ is not a smooth function of $x$ in dimensions $d \geq 3$. Circumventing this difficulty will rely on the tools in \cite{taylor2007tools} and \cite{visan2007defocusing}, along with bounds on $\frac{\nabla Q(x)}{Q(x)^{1 - \alpha}}$ for some $\alpha > 0$, when $Q$ is the ground state solution to $(\ref{1.12})$. Theorem $\ref{t14.0}$ is the only obstacle to the proof of Theorem $\ref{t1.3}$ in dimensions $d \geq 16$. \medskip

However, before proving Theorem $\ref{t1.3}$, it will be useful to cite two previous results making partial progress in the direction of Theorem $\ref{t1.3}$.
\begin{theorem}
If $u_{0} \in H^{1}$, $\| u_{0} \|_{L^{2}} = \| Q \|_{L^{2}}$, and the solution $u(t)$ to $(\ref{1.1})$ blows up in finite time $T > 0$, then $u(t,x)$ is in the form of $(\ref{1.5})$.
\end{theorem}
\begin{proof} 
This result was proved in \cite{merle1992uniqueness} and \cite{merle1993determination}, and was proved for the focusing, mass-critical nonlinear Schr{\"o}dinger equation in every dimension.
\end{proof}

For the mass-critical nonlinear Schr{\"o}dinger equation in higher dimensions with radially symmetric initial data, \cite{killip2009characterization} proved
\begin{theorem}
If $\| u_{0} \|_{L^{2}} = \| Q \|_{L^{2}}$ is radially symmetric, and $u$ is the solution to the focusing, mass-critical nonlinear Schr{\"o}dinger equation with initial data $u_{0}$, and $u$ blows up both forward and backward in time, then $u$ is equal to $(\ref{1.3.1})$ with $x_{0} = \xi_{0} = 0$.
\end{theorem}

After proving Theorem $\ref{t1.3}$, we will show how Theorem $\ref{t1.3}$ implies a Liouville result for blowup solutions to the mass-critical problem. This result is very similar in nature to the Liouville result in \cite{martel2000liouville} for the generalized KdV equation. This result holds in any dimension for which Theorems $\ref{t1.2}$ or $\ref{t1.3}$ hold.

\section{Reduction of a blowup solution}
As in \cite{dodson2021determination}, the first step is to reduce Theorem $\ref{t1.3}$ to a result for solutions to $(\ref{1.1})$ that blow up forward in time and are close to the family of solitons for every positive time.

\begin{theorem}\label{t2.2}
Let $0 < \eta_{\ast} \ll 1$ be a small, fixed constant to be defined later. If $u$ is a solution to $(\ref{1.1})$ on the maximal interval of existence $I \subset \mathbb{R}$, $\| u_{0} \|_{L^{2}} = \| Q \|_{L^{2}}$, $u$ blows up forward in time, and
\begin{equation}\label{2.6}
\sup_{t \in [0, \sup(I))} \inf_{\gamma \in \mathbb{R}, \lambda > 0, \xi \in \mathbb{R}^{d}, \tilde{x} \in \mathbb{R}^{d}} \| e^{i \gamma} e^{ix \cdot \xi} \lambda^{d/2} u(t, \lambda x + \tilde{x}) - Q \|_{L^{2}} \leq \eta_{\ast},
\end{equation}
then $u$ is a soliton solution of the form $(\ref{1.12})$ or a pseudoconformal transformation of the soliton of the form $(\ref{1.15})$.
\end{theorem}
Theorem $\ref{t1.3}$ follows from Theorem $\ref{t2.2}$.
\begin{proof}[Proof that Theorem $\ref{t2.2}$ implies Theorem $\ref{t1.3}$]
Let $u$ be a blowup solution to $(\ref{1.1})$ with mass $\| u_{0} \|_{L^{2}} = \| Q \|_{L^{2}}$. By $(\ref{1.12})$,
\begin{equation}\label{2.1}
\aligned
\inf_{\lambda > 0, \gamma \in \mathbb{R}, \tilde{x} \in \mathbb{R}^{d}, \xi \in \mathbb{R}^{d}} \| e^{i \gamma} e^{ix \cdot \xi} \lambda^{d/2} u_{0}(\lambda x + \tilde{x}) - Q(x) \|_{L^{2}(\mathbb{R}^{d})} \\ = \inf_{\lambda > 0, \gamma \in \mathbb{R}, \tilde{x} \in \mathbb{R}^{d}, \xi \in \mathbb{R}^{d}} \| u_{0}(x) - e^{-i \gamma} e^{-ix \cdot \frac{\xi}{\lambda}} \lambda^{-d/2} Q(\frac{x - \tilde{x}}{\lambda}) \|_{L^{2}(\mathbb{R}^{d})}.
\endaligned
\end{equation}
As in the $d = 1$ case, there exist $x_{0} \in \mathbb{R}^{d}$, $\xi_{0} \in \mathbb{R}^{d}$, $\lambda_{0} > 0$, $\gamma_{0} \in \mathbb{R}$ where this infimum is attained.

\begin{lemma}\label{l2.1}
There exists $\lambda_{0} > 0$, $\gamma_{0} \in \mathbb{R}$, $x_{0} \in \mathbb{R}^{d}$, and $\xi_{0} \in \mathbb{R}^{d}$, such that
\begin{equation}\label{2.2}
\| u_{0}(x) - e^{-i \gamma_{0}} e^{-i x \cdot \frac{\xi_{0}}{\lambda_{0}}} \lambda_{0}^{-d/2} Q(\frac{x - x_{0}}{\lambda_{0}}) \|_{L^{2}} = \inf_{\lambda > 0, \gamma \in \mathbb{R}, \tilde{x} \in \mathbb{R}^{d}, \xi \in \mathbb{R}^{d}} \| u_{0}(x) - e^{-i \gamma} e^{-ix \cdot \frac{\xi}{\lambda}} \lambda^{-d/2} Q(\frac{x - \tilde{x}}{\lambda}) \|_{L^{2}}.
\end{equation}
\end{lemma}
\begin{proof}
As in one dimension, the ground state solution $Q$ is smooth and rapidly decreasing.
\begin{theorem}
There exists a unique positive radially symmetric solution $Q$ to
\begin{equation}
\Delta Q + Q^{1 + 4/d} = Q
\end{equation}
in $H^{1}$ which is called the ground state. In addition, $Q \in C^{\infty}(\mathbb{R}^{d})$, $\partial_{r} Q(r) < 0$ for all $r > 0$, and there exists $\delta > 0$ such that
\begin{equation}
|\partial^{\alpha} Q(x)| \lesssim_{\alpha} e^{-\delta |x|}, \qquad \forall x \in \mathbb{R}^{d}, \qquad \forall \alpha \in \mathbb{Z}_{+}^{d}.
\end{equation}
\end{theorem}
\begin{proof}
See page $641$ of \cite{cote2016asymptotic}.
\end{proof}

Throughout the paper, $(f, g)_{L^{2}}$ will denote the inner product
\begin{equation}
(f, g)_{L^{2}} = Re \int \overline{f(x)} g(x) dx.
\end{equation}
Since $Q$ is smooth and rapidly decreasing,
\begin{equation}\label{2.3}
\aligned
(u_{0}(x) - e^{-i \gamma} e^{-ix \cdot \frac{\xi}{\lambda}} \lambda^{-d/2} Q(\frac{x - \tilde{x}}{\lambda}), u_{0}(x) - e^{-i \gamma} e^{-ix \cdot \frac{\xi}{\lambda}} \lambda^{-d/2} Q(\frac{x - \tilde{x}}{\lambda}))_{L^{2}} \\
= 2 \| Q \|_{L^{2}}^{2} - 2(e^{-i \gamma} e^{-ix \cdot \frac{\xi}{\lambda}} \lambda^{-d/2} Q(\frac{x - \tilde{x}}{\lambda}), u_{0}(x))_{L^{2}},
\endaligned
\end{equation}
is differentiable in $\gamma \in \mathbb{R}$, $\xi \in \mathbb{R}^{d}$, $\tilde{x} \in \mathbb{R}^{d}$, and $\lambda > 0$. Since
\begin{equation}\label{2.4}
\int_{0}^{2 \pi} (e^{-i \gamma} e^{-ix \cdot \frac{\xi}{\lambda}} \lambda^{-d/2} Q(\frac{x - \tilde{x}}{\lambda}), u_{0}(x))_{L^{2}} d\gamma = 0,
\end{equation}
if $(\ref{2.2}) = 2 \| Q \|_{L^{2}}^{2}$, then $(\ref{2.3}) = 2 \| Q \|_{L^{2}}^{2}$ for any $\lambda$, $\gamma$, $\tilde{x}$, and $\xi$. In this case it is convenient to take $\lambda_{0} = 1$, $\gamma_{0} = 0$, $\xi_{0} = 0$, and $\tilde{x} = 0$. If $(\ref{2.2}) < 2 \| Q \|_{L^{2}}^{2}$, then since
\begin{equation}\label{2.5}
(e^{-i \gamma} e^{-ix \cdot \frac{\xi}{\lambda}} \lambda^{-d/2} Q(\frac{x - \tilde{x}}{\lambda}), u_{0}(x))_{L^{2}} \rightarrow 0,
\end{equation}
as $\lambda \nearrow \infty$, $\lambda \searrow 0$, $|\tilde{x}| \rightarrow \infty$, or $|\xi| \rightarrow \infty$, $(\ref{2.2})$ is equal to the infimum over a compact set $[\lambda_{1}, \lambda_{2}] \times \{ \xi : |\xi| \leq R \} \times \{ \tilde{x} : |\tilde{x}| \leq R \} \times [0, 2 \pi]$. Indeed, $(\ref{2.5})$ holds as $\lambda \searrow 0$ or $\lambda \nearrow \infty$, uniformly over $\tilde{x}$ and $\xi$, so restrict $\lambda \in [\lambda_{1}, \lambda_{2}]$. Then $(\ref{2.5})$ holds as $|\tilde{x}| \rightarrow \infty$, uniformly over $\lambda \in [\lambda_{1}, \lambda_{2}]$ and $\xi \in \mathbb{R}^{d}$, so restrict $\tilde{x} \in \{ x : |x| \leq R \}$ for some large $R$. Finally, $(\ref{2.5})$ holds as $|\xi| \rightarrow \infty$ for $|\tilde{x}| \leq R$ and $\lambda \in [\lambda_{1}, \lambda_{2}]$. Since $(\ref{2.3})$ is continuous in $\lambda$, $\tilde{x}$, $\xi$, and $\gamma$, the infimum is attained on this compact set.
\end{proof}

Now recall the sequential convergence results of \cite{fan20182} and \cite{dodson20212}.
\begin{theorem}\label{t2.3}
Assume $u$ is a symmetric solution to $(\ref{1.1})$ with $\| u \|_{L^{2}} = \| Q \|_{L^{2}}$ that does not scatter forward in time. Let $(T_{-}(u), T_{+}(u))$ be its lifespan. Then there exists a sequence $t_{n} \rightarrow T_{+}(u)$ and a family of parameters $\lambda_{n}$, $\gamma_{n}$ such that
\begin{equation}\label{2.7}
e^{i \gamma_{n}} \lambda_{n}^{1/2} u(t_{n}, \lambda_{n} x) \rightarrow Q, \qquad \text{in} \qquad L^{2}.
\end{equation}
\end{theorem}
\begin{proof}
This is proved in Theorem $1.3$ of \cite{fan20182}.
\end{proof}

\begin{theorem}\label{t2.4}
Assume $u$ is a solution to $(\ref{1.1})$ with $\| u_{0} \|_{L^{2}} = \| Q \|_{L^{2}}$ which blows up forward in time. Let $(T_{-}(u), T_{+}(u))$ be its lifespan. Then there exists a sequence $t_{n} \nearrow T_{+}(u)$ and a family of parameters $\lambda_{\ast, n}$, $\gamma_{\ast, n}$, $\xi_{\ast, n}$, $x_{\ast, n}$ such that
\begin{equation}\label{2.8}
 e^{i \gamma_{\ast, n}} e^{ix \cdot \xi_{\ast, n}} \lambda_{\ast, n}^{d/2} u(t_{n}, \lambda_{\ast, n} x + x_{\ast, n}) \rightarrow Q, \qquad \text{in} \qquad L^{2}.
 \end{equation}
 \end{theorem}
 \begin{proof}
 This is proved in Theorem $2$ of \cite{dodson20212}.
\end{proof}
\noindent Then make the same argument as in \cite{dodson2021determination}, where we proved that Theorem $6$ implies Theorem $4$ and that Theorem $20$ implies Theorem $5$.
\end{proof}

\section{Decomposition of the solution near $Q$}
Under the assumptions of Theorem $\ref{t2.2}$, when $\eta_{\ast} \ll 1$ is sufficiently small, it is possible to decompose $u$ into the sum of a soliton and a remainder for which the linearization of $(\ref{1.12})$ has good spectral properties. Recall that in one dimension, the linearization of $(\ref{1.12})$,
\begin{equation}
\mathcal L = -\partial_{xx} + 1 - 5 Q^{4},
\end{equation}
has one negative eigenvalue and one zero eigenvalue. Choose $\lambda$, $\xi$, $\tilde{x}$, and $\gamma$ so that the difference $\epsilon$ between the soliton and the solution $u$ acted on by the representation of the group element $(\lambda, \xi, \tilde{x}, \gamma)$ is orthogonal to these two eigenvectors. The fact that $\mathcal L$ is positive definite on a subspace containing $\epsilon$ gives lower bounds on various virial and energy identities as a function of the size of $\epsilon$.

The same can be done in higher dimensions. In two dimensions, the spectral theory for
\begin{equation}\label{3.1}
\mathcal L = -\Delta + 1 - 3 Q^{2}, \qquad \mathcal L_{-} = -\Delta + 1 - Q^{2},
\end{equation}
is found in \cite{chang2008spectra}, \cite{kwong1989uniqueness}, and \cite{marics2002existence}. From the product rule, for any $j = 1, 2$,
\begin{equation}\label{3.2}
\mathcal L (Q_{x_{j}}) = -\Delta Q_{x_{j}} + Q_{x_{j}} - 3 Q^{2} Q_{x_{j}} = \partial_{x_{j}}(-\Delta Q + Q - Q^{3}) = 0.
\end{equation}
Thus $Q_{x_{j}}$ are eigenvectors of $\mathcal L$ with eigenvalue zero. These are the only two, and there is only one eigenvector of $\mathcal L$ with negative eigenvalue.
\begin{theorem}\label{t3.1}
The following holds for an operator $\mathcal L$ defined in $(\ref{3.1})$.

\begin{enumerate}
\item $\mathcal L$ is a self-adjoint operator and $\sigma_{ess}(\mathcal L) = [1, +\infty)$.

\item $Ker(\mathcal L) = span \{ Q_{x_{1}}, Q_{x_{2}} \}$.

\item $\mathcal L$ has a unique single negative eigenvalue $-\lambda_{0}$ associated to a positive, radially symmetric eigenfunction $\chi_{0}$. Without loss of generality, $\chi_{0}$ can be chosen such that $\| \chi_{0} \|_{L^{2}} = 1$. Moreover, there exists $\delta > 0$ such that $|\chi_{0}(x)| \lesssim e^{-\delta |x|}$ for all $x \in \mathbb{R}^{2}$.
\end{enumerate}
\end{theorem}
\begin{proof}
This is Theorem $3.3$ of \cite{farah2018blow}.
\end{proof}

In higher dimensions,  \cite{chang2008spectra}, \cite{kwong1989uniqueness}, and \cite{weinstein1985modulational} proved a similar result for the spectral theory of $\mathcal L$ and $\mathcal L_{-}$. 
\begin{equation}\label{3.2.1}
\mathcal L = -\Delta + 1 - \frac{d + 4}{d} Q^{\frac{4}{d}}, \qquad \mathcal L_{-} = -\Delta + 1 - Q^{\frac{4}{d}},
\end{equation}
\begin{theorem}\label{t3.1.1}
The following holds for an operator $\mathcal L$ defined in $(\ref{3.2.1})$.

\begin{enumerate}
\item $\mathcal L$ is a self-adjoint operator and $\sigma_{ess}(\mathcal L) = [1, +\infty)$.

\item $Ker(\mathcal L) = span \{ Q_{x_{1}}, \cdots, Q_{x_{d}} \}$.

\item $\mathcal L$ has a unique single negative eigenvalue $-\lambda_{d}$ associated to a positive, radially symmetric eigenfunction $\chi_{0}$. Without loss of generality, $\chi_{0}$ can be chosen such that $\| \chi_{0} \|_{L^{2}} = 1$. Moreover, there exists $\delta > 0$ such that $|\chi_{0}(x)| \lesssim e^{-\delta |x|}$ for all $x \in \mathbb{R}^{d}$.
\end{enumerate}
\end{theorem}
\begin{proof}
This theorem is copied from page $642$ of \cite{cote2016asymptotic}.
\end{proof}

For $u_{0}$ sufficiently close to $Q$, it is possible to choose $\lambda$, $\gamma$, $\tilde{x}$, and $\xi$ so that the remainder is orthogonal to the kernel of $\mathcal L$ and to the negative eigenvector.

\begin{theorem}\label{t3.2}
Take $u \in L^{2}$. There exists $\alpha > 0$ sufficiently small such that if there exist $\lambda_{0} > 0$, $\gamma_{0} \in \mathbb{R}$, $x_{0} \in \mathbb{R}^{d}$, and $\xi_{0} \in \mathbb{R}^{d}$ that satisfy
\begin{equation}\label{3.3}
\| e^{i \gamma_{0}} e^{ix \cdot \xi_{0}} \lambda_{0}^{d/2} u(\lambda_{0} x + x_{0}) - Q(x) \|_{L^{2}} \leq \alpha,
\end{equation}
then there exist unique $\lambda > 0$, $\gamma \in \mathbb{R}$, $\tilde{x} \in \mathbb{R}^{d}$, $\xi \in \mathbb{R}^{d}$ such that if
\begin{equation}\label{3.4}
\epsilon(x) = e^{i \gamma} e^{ix \cdot \xi} \lambda^{d/2} u(\lambda x + \tilde{x}) - Q(x),
\end{equation}
then for $j = 1, ..., d$,
\begin{equation}\label{3.5}
( \epsilon, \chi_{0} )_{L^{2}} = ( \epsilon, i \chi_{0})_{L^{2}} = (\epsilon, Q_{x_{j}})_{L^{2}} = (\epsilon, i Q_{x_{j}})_{L^{2}} = 0,
\end{equation}
and
\begin{equation}\label{3.6}
\| \epsilon \|_{L^{2}} + |\frac{\lambda}{\lambda_{0}} - 1| + |\gamma - \gamma_{0} - \xi_{0} \cdot (\tilde{x} - x_{0})| + |\xi - \frac{\lambda}{\lambda_{0}} \xi_{0}| + |\frac{\tilde{x} - x_{0}}{\lambda_{0}}| \lesssim \| e^{i \gamma_{0}} e^{ix \cdot \xi_{0}} \lambda_{0} u(\lambda_{0} x + x_{0}) - Q \|_{L^{2}}.
\end{equation}
\end{theorem}
\begin{remark}
As usual, $\gamma$ is unique up to translation by $2 \pi n$, where $n$ is an integer.
\end{remark}
\begin{proof}
Let $f$ denote an element of the set
\begin{equation}\label{3.7}
f \in \{ \chi_{0}, i \chi_{0}, Q_{x_{1}}, ..., Q_{x_{d}}, i Q_{x_{1}}, ..., i Q_{x_{d}} \}.
\end{equation}
By H{\"o}lder's inequality,
\begin{equation}\label{3.8}
|( e^{i \gamma_{0}} e^{ix \cdot \xi_{0}} \lambda_{0}^{d/2} u(t, \lambda_{0} x + \tilde{x}) - Q(x), f )_{L^{2}}| \lesssim \| e^{i \gamma_{0}} e^{ix \cdot \xi_{0}} \lambda_{0}^{d/2} u(\lambda_{0} x + \tilde{x}) - Q \|_{L^{2}}.
\end{equation}
The inner product in $(\ref{3.8})$ is $C^{1}$ as a function of $\gamma$, $\lambda$, $\tilde{x}$, and $\xi$. Indeed,
\begin{equation}\label{3.9}
\frac{\partial}{\partial \gamma} ( e^{i \gamma} e^{ix \cdot \xi} \lambda^{d/2} u(\lambda x + \tilde{x}) - Q(x), f )_{L^{2}} = ( i e^{i \gamma} e^{ix \cdot \xi} \lambda^{d/2} u(\lambda x + \tilde{x}), f )_{L^{2}} \lesssim \| u \|_{L^{2}} \| f \|_{L^{2}}.
\end{equation}
Next, since all $f$ are smooth with rapidly decreasing derivatives,
\begin{equation}\label{3.10}
\frac{\partial}{\partial \xi} ( e^{i \gamma} e^{ix \cdot \xi} \lambda^{d/2} u(\lambda x + \tilde{x}) - Q(x), f )_{L^{2}} = ( ix e^{i \gamma} e^{ix \cdot \xi} \lambda^{d/2} u(\lambda x + \tilde{x}), f )_{L^{2}} \lesssim \| u \|_{L^{2}} \| x f \|_{L^{2}}.
\end{equation}
Integrating by parts, 
\begin{equation}\label{3.11}
\aligned
\frac{\partial}{\partial \tilde{x}} ( e^{i \gamma} e^{ix \cdot \xi} \lambda^{d/2} u(\lambda x + \tilde{x}) - Q(x), f )_{L^{2}} = ( e^{i \gamma} e^{ix \cdot \xi} \lambda^{d/2} \nabla u(\lambda x + \tilde{x}), f )_{L^{2}} \\ = -\frac{1}{\lambda} (e^{i \gamma} e^{ix \cdot \xi} \lambda^{d/2} u(\lambda x + \tilde{x}), \nabla f)_{L^{2}} - \frac{\xi}{\lambda} (i e^{i \gamma} e^{ix \cdot \xi} \lambda^{d/2} u(\lambda x + \tilde{x}), f)_{L^{2}} \lesssim \frac{1}{\lambda} \| u \|_{L^{2}} \| \nabla f \|_{L^{2}} + \frac{|\xi|}{\lambda} \| u \|_{L^{2}} \| f \|_{L^{2}}.
\endaligned
\end{equation}
Finally,
\begin{equation}\label{3.12}
\aligned
\frac{\partial}{\partial \lambda} ( e^{i \gamma} e^{ix \cdot \xi} \lambda^{d/2} u(\lambda x + \tilde{x}) - Q(x), f )_{L^{2}} = ( \frac{d}{2} \lambda^{d/2 - 1} e^{i \gamma} e^{ix \cdot \xi} u(\lambda x + \tilde{x}) + e^{i \gamma} e^{ix \cdot \xi} \lambda^{d/2} x \cdot \nabla u(\lambda x + \tilde{x}), f)_{L^{2}} \\
\lesssim \frac{1}{\lambda} \| u \|_{L^{2}} \| f \|_{L^{2}} + \frac{|\xi|}{\lambda} \| u \|_{L^{2}} \| x f \|_{L^{2}} + \frac{1}{\lambda} \| u \|_{L^{2}} \| x \nabla f \|_{L^{2}}.
\endaligned
\end{equation}
Therefore, the inner product is a $C^{1}$ function of $\gamma$, $\xi$, $\lambda$, and $\tilde{x}$. Repeating the above calculations would also show that the inner product is $C^{2}$.

Computing $(\ref{3.9})$--$(\ref{3.12})$ at $u = Q$, $\xi = \tilde{x} = \gamma = 0$, $\lambda = 1$,
\begin{equation}\label{3.13}
\aligned
\frac{\partial}{\partial \gamma} ( e^{i \gamma} e^{ix \cdot \xi} \lambda^{d/2} u(\lambda x + \tilde{x}) - Q(x), f )_{L^{2}}|_{u = Q, \lambda = 1, \gamma = \tilde{x} = \xi = 0} = ( i Q, f )_{L^{2}} = 0 \qquad \text{if} \qquad f \in \{ \chi_{0}, Q_{x_{j}}, i Q_{x_{j}} \}, \\
= (i Q, i \chi_{0})_{L^{2}} = (Q, \chi_{0})_{L^{2}} > 0, \qquad \text{if} \qquad f = i \chi_{0}.
\endaligned
\end{equation}
The fact that $(Q, \chi_{0})_{L^{2}} > 0$ follows from the fact that $\chi_{0} > 0$ and $Q > 0$. Next,
\begin{equation}\label{3.14}
\aligned
\frac{\partial}{\partial \xi_{k}} ( e^{i \gamma} e^{ix \cdot \xi} \lambda^{d/2} u(\lambda x + \tilde{x}) - Q(x), f )_{L^{2}}|_{u = Q, \lambda = 1, \gamma = \tilde{x} = \xi = 0} = ( ix_{k} Q, f )_{L^{2}} = 0, \qquad \text{if} \qquad f \in \{ \chi_{0}, Q_{x_{j}}, i \chi_{0} \}, \\
= (i x_{k} Q, i Q_{x_{j}})_{L^{2}} = (x_{k} Q, Q_{x_{j}})_{L^{2}} = -\frac{\delta_{jk}}{2} \| Q \|_{L^{2}}^{2} \qquad \text{if} \qquad f = i Q_{x_{j}}, \qquad j = 1, ..., d.
\endaligned
\end{equation}
Next,
\begin{equation}\label{3.15}
\aligned
\frac{\partial}{\partial \tilde{x}_{k}} ( e^{i \gamma} e^{ix \cdot \xi} \lambda^{d/2} u(\lambda x + \tilde{x}) - Q(x), f )_{L^{2}}|_{u = Q, \lambda = 1, \gamma = \tilde{x} = \xi = 0} = ( Q_{x_{k}}, f )_{L^{2}} = 0, \\ \qquad \text{if} \qquad f \in \{ i \chi_{0}, i Q_{x_{j}}, \chi_{0} \} \\ = (Q_{x_{k}}, Q_{x_{j}})_{L^{2}} = \frac{\delta_{jk}}{d} \| \nabla Q \|_{L^{2}}^{2}, \qquad \text{if} \qquad f = Q_{x_{j}}, \qquad j = 1, ..., d.
\endaligned
\end{equation}
Finally,
\begin{equation}\label{3.16}
\aligned
\frac{\partial}{\partial \lambda} ( e^{i \gamma} e^{ix \cdot \xi} \lambda^{d/2} u(\lambda x + \tilde{x}) - Q(x), f )_{L^{2}}|_{u = Q, \lambda = 1, \gamma = \tilde{x} = \xi = 0} = (\frac{d}{2} Q + x \cdot \nabla Q, f)_{L^{2}} = 0 \\ \qquad \text{if} \qquad f \in \{ i \chi_{0}, i Q_{x_{j}}, Q_{x_{j}} \}, \\
 = (\frac{d}{2} Q + x \cdot \nabla Q, \chi_{0})_{L^{2}} = \frac{-1}{\lambda_{d}} (\frac{d}{2} Q + x \cdot \nabla Q, \mathcal L \chi_{0})_{L^{2}} = -\frac{1}{\lambda_{d}} (\mathcal L (\frac{d}{2} Q + x \cdot \nabla Q), \chi_{0})_{L^{2}} = \frac{2}{\lambda_{d}} (Q, \chi_{0}) > 0, \\ \qquad \text{if} \qquad f = \chi_{0}.
\endaligned
\end{equation}
Therefore, by the inverse function theorem, there exists a unique $\lambda > 0$, $\tilde{x} \in \mathbb{R}^{2}$, $\xi \in \mathbb{R}^{2}$ and $\gamma \in \mathbb{R} / 2 \pi \mathbb{Z}$ in a neighborhood of $\lambda_{0}$, $\xi_{0}$, $x_{0}$, and $\gamma_{0}$ such that $(\ref{3.5})$ holds.

The proof of $(\ref{3.6})$ follows from acting on $u$ by symmetries to map to $\lambda_{0} = 1$, $\tilde{x} = \xi = \gamma = 0$, applying the inverse function theorem, and then mapping back to the original $u$. The proof of uniqueness in $\mathbb{R}_{> 0} \times \mathbb{R}^{n} \times \mathbb{R}^{n} \times \mathbb{R} / 2 \pi \mathbb{Z}$ is identical to the proof of uniqueness in \cite{dodson2021determination}.
\end{proof}

Therefore, in Theorem $\ref{t2.2}$, there exist functions
\begin{equation}\label{3.17}
\aligned
\lambda(t) : [0, \sup(I)) \rightarrow (0, \infty), \qquad \xi(t) : [0, \sup(I)) \rightarrow \mathbb{R}^{d}, \\ \qquad x(t) : [0, \sup(I)) \rightarrow \mathbb{R}^{d}, \qquad \gamma(t) : [0, \sup(I)) \rightarrow \mathbb{R},
\endaligned
\end{equation}
such that $(\ref{3.5})$ holds for all $t \in [0, \sup(I))$.

Furthermore, define a monotone function,
\begin{equation}\label{3.18}
s(t) : [0, \sup(I)) \rightarrow [0, \infty), \qquad s(t) = \int_{0}^{t} \lambda(\tau)^{-2} d\tau.
\end{equation}
As in Theorem $10$ of \cite{dodson2021determination}, the functions in $(\ref{3.17})$ are differentiable in time for almost every $t \in [0, \sup(I))$, and furthermore, taking $\epsilon = \epsilon_{1} + i \epsilon_{2}$,
\begin{equation}\label{3.19}
\aligned
\epsilon_{s} = i (\gamma_{s} + 1) (Q + \epsilon) + i \xi_{s} \cdot x (Q + \epsilon) + \frac{\lambda_{s}}{\lambda} (\frac{d}{2} (Q + \epsilon) + x \cdot \nabla (Q + \epsilon)) - i \frac{\lambda_{s}}{\lambda} \xi(s) \cdot x (Q + \epsilon) \\
+ \frac{x_{s}}{\lambda} \cdot \nabla (Q + \epsilon) - i \frac{x_{s}}{\lambda} \cdot \xi(s) (Q + \epsilon) + i \mathcal L \epsilon_{1} - \mathcal L_{-} \epsilon_{2} + 2 \xi(s) \cdot \nabla (Q + \epsilon) - i |\xi(s)|^{2} (Q + \epsilon) \\ + i O(|Q|^{4/d - 1} |\epsilon|^{2} + |\epsilon|^{1 + 4/d}), \qquad \text{when} \qquad 2 \leq d \leq 4, \qquad + i O(|\epsilon|^{1 + 4/d}) \qquad \text{when} \qquad d > 5.
\endaligned
\end{equation}

For $f$ in $(\ref{3.7})$,
\begin{equation}\label{3.20}
\frac{d}{ds} (\epsilon, f )_{L^{2}} = (\epsilon_{s}, f)_{L^{2}} = 0.
\end{equation}
Plugging $(\ref{3.19})$ into $(\ref{3.20})$ and following the analysis in Section $10.2$ of \cite{dodson2021determination}, for any $a \in \mathbb{Z}_{\geq 0}$,
\begin{equation}\label{3.21}
\int_{a}^{a + 1} |\gamma_{s} + 1 - \frac{x_{s}}{\lambda} \cdot \xi(s) - |\xi(s)|^{2}| ds \lesssim \int_{a}^{a + 1} \| \epsilon(s) \|_{L^{2}}^{2} ds,
\end{equation}
\begin{equation}\label{3.22}
\int_{a}^{a + 1} |\xi_{s} - \frac{\lambda_{s}}{\lambda} \xi(s)| ds \lesssim \int_{a}^{a + 1} \| \epsilon(s) \|_{L^{2}}^{2} ds,
\end{equation}
\begin{equation}\label{3.23}
\int_{a}^{a + 1} |\frac{\lambda_{s}}{\lambda}| ds \lesssim \int_{a}^{a + 1} \| \epsilon(s) \|_{L^{2}} ds,
\end{equation}
and
\begin{equation}\label{3.24}
\int_{a}^{a + 1} |\frac{x_{s}}{\lambda} + 2 \xi| ds \lesssim \int_{a}^{a + 1} \| \epsilon(s) \|_{L^{2}} ds.
\end{equation}

As in \cite{dodson2021determination}, it is also true that under the conditions of Theorem $\ref{t2.2}$, for any $a \geq 0$,
\begin{equation}\label{3.25}
\sup_{s \in [a, a + 1]} \| \epsilon(s) \|_{L^{2}} \sim \inf_{s \in [a, a + 1]} \| \epsilon(s) \|_{L^{2}}.
\end{equation}

\section{A long time Strichartz estimate when $d = 2$}
As in the one dimensional case, the next step is to obtain a long time Strichartz estimate. Roughly speaking, if $I$ is an interval of length $|I| = T$, and $\lambda(t) = 1$ on $I$, the goal is to obtain good Strichartz estimates at frequencies greater than or equal to $T^{\alpha}$ for some $\alpha < \frac{1}{2}$. In two dimensions, as in one dimension, $\alpha = \frac{1}{3}$ will do. In higher dimensions, we will take $T^{\alpha_{d}}$, where $\alpha_{d} \rightarrow \frac{1}{2}$ as $d$ goes to infinity.

The main new technical difficulty in two dimensions is utilizing the interaction Morawetz bilinear estimate of \cite{planchon2009bilinear}. Setting
\begin{equation}
M(t) = \int |u(t,y)|^{2} \frac{(x - y)_{j}}{|(x - y)_{j}|} Im[\bar{u} \partial_{j} u](t,x) dx dy,
\end{equation}
if $u$ solves $(\ref{1.1})$,
\begin{equation}
\aligned
\frac{d}{dt} M(t) = -2 \int \partial_{k} Im[\bar{u} \partial_{k} u](t, y) \frac{(x - y)_{j}}{|(x - y)_{j}|} Im[\bar{u} \partial_{j} u](t, x) dx dy + \frac{1}{2} \int |u(t, x)|^{2} \frac{(x - y)_{j}}{|(x - y)_{j}} \partial_{j} \Delta(|u|^{2}) dx dy \\
- 2 \int |u(t,y)|^{2} \frac{(x - y)_{j}}{|(x - y)_{j}|} \partial_{k} Re[\partial_{j} \bar{u} \partial_{k}](t, x) dx dy + \frac{2}{d + 2} \int |u(t,y)|^{2} \frac{(x - y)_{j}}{|(x - y)_{j}|} \partial_{j}(|u(t,x)|^{\frac{2(d + 2)}{d}}) dx dy.
\endaligned
\end{equation}
In one dimension, integrating by parts,
\begin{equation}\label{6.0}
\frac{d}{dt} M(t) = \int |\partial_{x}(|u|^{2})(t,x)|^{2} dx - \frac{2}{3} \int |u(t,x)|^{8} dx.
\end{equation}
If the $u$'s in the interaction Morawetz estimate are replaced by Fourier truncations of $u$, this gives good bilinear estimates, as was used in \cite{dodson2021determination}. In two dimensions, fixing $j$, say $j = 1$,
\begin{equation}\label{6.0.1}
\frac{d}{dt} M(t) = \int |\partial_{x_{1}}(u(t, x_{1}, x_{2}) \overline{u(t, x_{1}, y_{2})})|^{2} dx_{1} dx_{2} dy_{2} - \frac{1}{2} \int |u(t,x_{1}, x_{2})|^{2} |u(t, x_{1}, y_{2})|^{4} dx_{1} dx_{2} dy_{2}.
\end{equation}
In order to turn $(\ref{6.0.1})$ into a bilinear estimate of the form $(\ref{6.0})$, we utilize the Fourier support of $u$. Suppose again that the $u$'s are replaced by a $u$ at frequencies $\geq k$ and a $u$ at frequencies $\leq k$. By H{\"o}lder's inequality and the product rule, if $\tilde{\psi}$ is a rapidly decreasing function obtained from the Littlewood--Paley kernel,
\begin{equation}\label{6.0.2}
\aligned
\int |\partial_{x_{1}}(P_{\leq k} u(t, x_{1}, x_{2}) \overline{P_{\geq k} u(t, x_{1}, x_{2})})|^{2} dx_{1} dx_{2} \\ \lesssim \int |\partial_{x_{1}}(P_{\geq k} u(t, x_{1}, x_{2})|^{2} |P_{\leq k} u(t, x_{1}, x_{2})|^{2} dx_{1} dx_{2} + \| \partial_{x_{1}} P_{\leq k} u \|_{L^{p}}^{2} \| P_{\geq k} u \|_{L^{q}}^{2} \\ \lesssim 2^{k} \int \tilde{\psi}(2^{k}(x_{2} - y_{2})) |\partial_{x_{1}}(P_{\leq k} u(t, x_{1}, x_{2}) \overline{P_{\geq k} u(t, x_{1}, y_{2})})|^{2} dx_{1} dx_{2} dy_{2} \\ + 2^{k} \int \tilde{\psi}(2^{k}(x_{2} - y_{2})) |\partial_{x_{1}} P_{\leq k} u(t, x_{1}, x_{2})|^{2} |P_{\geq k} u(t, x_{1}, y_{2})|^{2}  dx_{1} dx_{2} dy_{2} + \| \partial_{x_{1}} P_{\leq k} u \|_{L^{p}}^{2} \| P_{\geq k} u \|_{L^{q}}^{2} \\
\lesssim 2^{k}  \int |\partial_{x_{1}}(P_{\leq k} u(t, x_{1}, x_{2}) \overline{P_{\geq k} u(t, x_{1}, y_{2})})|^{2} dx_{1} dx_{2} dy_{2} + \| \partial_{x_{1}} P_{\leq k} u \|_{L^{p}}^{2} \| P_{\geq k} u \|_{L^{q}}^{2}.
\endaligned
\end{equation}
This calculation appeared previously in \cite{dodson2016global2} and \cite{dodson2019defocusing}. Here $\frac{1}{p} + \frac{1}{q} = \frac{1}{2}$.\medskip

Let $J = [a, b]$ be an interval. As in the one dimensional case, choose
\begin{equation}\label{6.1}
0 < \eta_{1} \ll \eta_{0} \ll 1,
\end{equation}
such that
\begin{equation}\label{6.2}
\sup_{t \in J} \| \epsilon(t,x) \|_{L^{2}}^{2} \leq \eta_{0}^{2},
\end{equation}
and that
\begin{equation}\label{6.2.1}
\int_{|\xi| \geq \eta_{1}^{-1/2}} |\hat{Q}(\xi)|^{2} \leq \eta_{0}^{2},
\end{equation}
where $Q$ is the soliton solution $(\ref{1.12})$. Furthermore, suppose that
\begin{equation}\label{6.2.2}
\frac{1}{\eta_{1}} \leq \lambda(t) \leq \frac{1}{\eta_{1}} T^{1/100}, \qquad \text{and} \qquad \frac{|\xi(t)|}{\lambda(t)} \leq \eta_{0}, \qquad \text{for all} \qquad t \in J,
\end{equation}
and that there exists $k \in \mathbb{Z}_{\geq 0}$ such that
\begin{equation}\label{6.2.3}
\int_{J} \lambda(t)^{-2} dt = T, \qquad \text{and} \qquad \eta_{1}^{-2} T = 2^{3k}.
\end{equation}

When $i \in \mathbb{Z}$, $i > 0$, let $P_{i}$ denote the standard Littlewood-Paley projection operator. When $i = 0$, let $P_{i}$ denote the projection operator $P_{\leq 0}$, and when $i < 0$, let $P_{i}$ denote the zero operator. It is convenient to calculate for $\lambda(t) = \frac{1}{\eta_{1}}$ first and then generalize to $(\ref{6.2.2})$. It is also convenient to assume without loss of generality that $a = 0$.
\begin{theorem}\label{t6.2}
Suppose $J = [0, \eta_{1}^{-2} T]$ is an interval for which
\begin{equation}\label{6.2.4}
\lambda(t) = \frac{1}{\eta_{1}}, \qquad \text{and} \qquad \frac{|\xi(t)|}{\lambda(t)} \leq \eta_{0}, \qquad  \text{for all} \qquad t \in J, \qquad \int_{J} \lambda(t)^{-2} dt = T, \qquad \text{and} \qquad \eta_{1}^{-2} T = 2^{3k}.
\end{equation}
Then for some $p > 2$, $p = 5/2$ will do, define the norm
\begin{equation}\label{6.3}
\aligned
\| u \|_{X([0, \eta_{1}^{-2} T] \times \mathbb{R}^{2})}^{2} =  \sup_{0 \leq i \leq k} \sup_{1 \leq a \leq 2^{3k - 3i}} \| P_{\geq i} u \|_{U_{\Delta}^{p}([(a - 1) 2^{3i}, a 2^{3i}] \times \mathbb{R}^{2})}^{2} + \frac{1}{\eta_{0}^{1/10}}  \| (P_{\geq i} u)(P_{\leq i - 3} u) \|_{L_{t,x}^{2}([(a - 1) 2^{3i}, a 2^{3i}] \times \mathbb{R}^{2})}^{2}.
\endaligned
\end{equation}
Also, for any $0 \leq j \leq k$, let
\begin{equation}\label{6.4}
\aligned
\| u \|_{X_{j}([0, \eta_{1}^{-2} T] \times \mathbb{R}^{2})}^{2} = \sup_{0 \leq i \leq j} \sup_{1 \leq a \leq 2^{3k - 3i}} \| P_{\geq i} u \|_{U_{\Delta}^{p}([(a - 1) 2^{3i}, a 2^{3i}] \times \mathbb{R}^{2})}^{2} \\ + \sup_{0 \leq i \leq j} \sup_{1 \leq a \leq 2^{3k - 3i}}  \frac{1}{\eta_{0}^{1/10}}  \| (P_{\geq i} u)(P_{\leq i - 3} u) \|_{L_{t,x}^{2}([(a - 1) 2^{3i}, a 2^{3i}] \times \mathbb{R}^{2})}^{2}.
\endaligned
\end{equation}
Then the long time Strichartz estimate,
\begin{equation}\label{6.5}
\| u \|_{X([0, \eta_{1}^{-2} T] \times \mathbb{R}^{2})} \lesssim 1,
\end{equation}
holds with implicit constant independent of $T$.
\end{theorem}
\begin{proof}
This estimate is proved by induction on $j$. Local well-posedness arguments combined with the fact that $\lambda(t) = \frac{1}{\eta_{1}}$ for any $t \in [0, T]$ imply that
\begin{equation}\label{6.6}
\| u \|_{U_{\Delta}^{2}([a, a + 1] \times \mathbb{R}^{2})} \lesssim 1,
\end{equation}
and when $i = 0$,
\begin{equation}\label{6.7}
(P_{\geq i} u)(P_{\leq i - 3} u) = 0.
\end{equation}
Therefore,
\begin{equation}\label{6.8}
\| u \|_{X_{j}([0, \eta_{1}^{-2} T] \times \mathbb{R}^{2})} \lesssim 1,
\end{equation}
when $j = 0$. This is the base case.\medskip

To prove the inductive step, recall that by Duhamel's principle, if $J_{a}^{(k)} = [(a - 1) 2^{3k - 3i}, a 2^{3k - 3i}]$, then for any $t_{0} \in J_{a}^{(k)}$,
\begin{equation}\label{6.9}
u(t) = e^{i(t - t_{0}) \Delta} u(t_{0}) + i \int_{t_{0}}^{t} e^{i(t - \tau) \Delta} (|u|^{2} u) d\tau,
\end{equation}
and
\begin{equation}\label{6.10}
\| P_{\geq i} u \|_{U_{\Delta}^{p}(J_{a}^{(k)} \times \mathbb{R}^{2})} \lesssim \| P_{\geq i}(u(t_{0})) \|_{L^{2}} + \| \int_{t_{0}}^{t} e^{i(t - \tau) \Delta} P_{\geq i}(|u|^{2} u) d\tau \|_{U_{\Delta}^{p}(J_{a}^{(k)} \times \mathbb{R}^{2})}.
\end{equation}

Since $\| u(t_{0}) \|_{L^{2}} \lesssim 1$, turn to the Duhamel term. Choose $v \in V_{\Delta}^{p'}(J \times \mathbb{R})$ such that $\| v \|_{V_{\Delta}^{p'}(J \times \mathbb{R})} = 1$ and $\hat{v}(t, \xi)$ is supported on the Fourier support of $P_{i}$. It is a well-known fact that
\begin{equation}\label{6.12}
 \| \int_{t_{0}}^{t} e^{i(t - \tau) \Delta} P_{\geq i}(|u|^{2} u) d\tau \|_{U_{\Delta}^{p}(J \times \mathbb{R})} \lesssim \sup_{v} \| v P_{\geq i}(|u|^{2} u) \|_{L_{t,x}^{1}},
 \end{equation}
 where $\sup_{v}$ is the supremum over all such $v$ supported on $P_{i}$ satisfying $\| v \|_{V_{\Delta}^{p'}(J \times \mathbb{R})} = 1$. See \cite{hadac2009well} for a proof.
 
Throughout this section it is not so important to distinguish between $u$ and $\bar{u}$. Since
\begin{equation}\label{6.23}
P_{\geq i}(|u_{\leq i - 3}|^{2} u_{\leq i - 3}) = 0,
\end{equation}
decompose
\begin{equation}\label{6.11}
P_{\geq i}(|u|^{2} u) = P_{\geq i} O((P_{\geq i - 3} u)^{3}) + P_{\geq i} O((P_{\geq i - 3} u)^{2} (P_{\leq i - 3} u)) + P_{\geq i} O((P_{\geq i - 3} u)(P_{\leq i - 3} u)^{2}).
\end{equation}

By H{\"o}lder's inequality,
\begin{equation}\label{6.13}
\aligned
\| v(u_{\geq i - 3})^{2}(u_{\leq i - 3}) \|_{L_{t,x}^{1}} + \|  v(u_{\geq i - 3})^{3} \|_{L_{t,x}^{1}}
\lesssim \| v \|_{L_{t}^{\infty} L_{x}^{2}} \| u_{\geq i - 3} \|_{L_{t}^{3} L_{x}^{6}}^{3} + \| v \|_{L_{t}^{3} L_{x}^{6}} \| u_{\geq i - 3} \|_{L_{t}^{3} L_{x}^{6}}^{2} \| u_{\leq i - 3} \|_{L_{t}^{\infty} L_{x}^{2}}.
\endaligned
\end{equation}
Since $V_{\Delta}^{p'} \subset U_{\Delta}^{2}$ for any $p > 2$, again see \cite{hadac2009well},
\begin{equation}\label{6.14}
\| v \|_{L_{t}^{\infty} L_{x}^{2}} + \| v \|_{L_{t}^{3} L_{x}^{6}} \lesssim \| v \|_{U_{\Delta}^{2}} \lesssim \| v \|_{V_{\Delta}}^{p'} = 1.
\end{equation}
Therefore, when $i > 4$, by $(\ref{6.2})$, $(\ref{6.2.1})$, $(\ref{6.2.4})$, and interpolation,
\begin{equation}\label{6.15}
\| u_{\geq i - 3} \|_{L_{t}^{3} L_{x}^{6}}^{3} \lesssim \| u_{\geq i - 3} \|_{L_{t}^{5/2} L_{x}^{10}}^{5/2} \| u_{\geq i - 3} \|_{L_{t}^{\infty} L_{x}^{2}}^{1/2} \lesssim \eta_{0}^{1/2} \| u_{\geq i - 3} \|_{L_{t}^{5/2} L_{x}^{10}}^{5/2} \lesssim \eta_{0}^{1/2} \| u \|_{X_{i - 3}([0, T] \times \mathbb{R})}^{5/2}.
\end{equation}

When $i \leq 4$, simply use $(\ref{3.4})$, which implies
\begin{equation}\label{6.16}
u(t,x) = \frac{1}{\lambda(t)} e^{-i \gamma(t)} e^{-ix \cdot \frac{\xi(t)}{\lambda(t)}} Q(\frac{x - x(t)}{\lambda(t)}) + \frac{1}{\lambda(t)} e^{-i \gamma(t)} e^{-i x \cdot \frac{\xi(t)}{\lambda(t)}} \epsilon(t, \frac{x - x(t)}{\lambda(t)}).
\end{equation}
By local well-posedness arguments and $\lambda(t) = \frac{1}{\eta_{1}}$, for any $a \in \mathbb{Z}_{\geq 0}$,
\begin{equation}\label{6.16.1}
\| u \|_{L_{t}^{5/2} L_{x}^{10}([a, a + 1] \times \mathbb{R}^{2})} \lesssim 1, \qquad \| \epsilon \|_{L_{t}^{5/2} L_{x}^{10}([a, a + 1] \times \mathbb{R}^{2})} \lesssim \eta_{0},
\end{equation}
and therefore by $(\ref{6.2})$ and $(\ref{6.2.1})$,
\begin{equation}\label{6.16.2}
\| u_{\geq i - 3} \|_{L_{t}^{3} L_{x}^{6}([a, a + 1] \times \mathbb{R}^{2})}^{3} \lesssim \eta_{0}.
\end{equation}
Therefore,
\begin{equation}\label{6.18}
\aligned
\| u_{\geq i - 3} \|_{L_{t}^{3} L_{x}^{6}}^{3} + \| u_{\geq i - 3} \|_{L_{t}^{3} L_{x}^{6}}^{2} \| u_{\leq i - 3} \|_{L_{t}^{\infty} L_{x}^{2}}  \lesssim \eta_{0}^{1/2} \| u \|_{X_{i - 3}([0, T] \times \mathbb{R})}^{3} + \eta_{0}^{1/2} \| u \|_{X_{i - 3}([0, T] \times \mathbb{R})}^{5/2} + \eta_{0}^{1/2}.
\endaligned
\end{equation}

Now compute
\begin{equation}\label{6.24}
\aligned
\| v ((P_{\geq i - 3} u)(P_{\leq i - 3} u)^{2}) \|_{L_{t,x}^{1}} \lesssim \| (P_{\geq i - 3} u)(P_{\leq i - 3} u) \|_{L_{t,x}^{2}} \| v (P_{\leq i - 3} u) \|_{L_{t,x}^{2}}.
\endaligned
\end{equation}

\noindent By $(\ref{6.4})$,
\begin{equation}\label{6.25}
\| (P_{\geq i - 3} u)(P_{\leq i - 3} u) \|_{L_{t,x}^{2}([(a - 1) 2^{3i}, a 2^{3i}] \times \mathbb{R}^{2})} \lesssim \eta_{0}^{1/20} \| u \|_{X_{i - 3}}.
\end{equation}
Next, suppose that it is true that for any $\| v_{0} \|_{L^{2}} = 1$, where $\hat{v}_{0}$ is supported on $|\xi| \geq 2^{i}$,
\begin{equation}\label{6.28}
\sup_{v_{0}} \| (e^{it \Delta} v_{0})(u_{\leq i - 3}) \|_{L_{t,x}^{2}} \lesssim 1 + \eta_{0} \| u \|_{X_{i - 3}}^{3}.
\end{equation}
Then $(\ref{6.28})$ combined with $V_{\Delta}^{p'} \subset U_{\Delta}^{2}$ implies
\begin{equation}\label{6.29}
\| v(P_{\leq i - 3} u) \|_{L_{t,x}^{2}} \lesssim 1 + \eta_{0} \| u \|_{X_{i - 3}}^{3},
\end{equation}
and therefore,
\begin{equation}\label{6.30}
\sup_{1 \leq a \leq 2^{3(k - i)}} \| P_{\geq i} u \|_{U_{\Delta}^{p}([(a - 1) 2^{3i}, a 2^{3i}] \times \mathbb{R}^{2})} \lesssim 1 + \eta_{0}^{1/20} \| u \|_{X_{i - 3}} (1 + \eta_{0} \| u \|_{X_{i - 3}}^{3}).
\end{equation}
This bound is fine for the induction on frequency argument.\medskip

The bilinear estimate $(\ref{6.28})$ is proved using the interaction Morawetz estimate described at the beginning of this section. To simplify notation, let
\begin{equation}\label{6.32}
v(t,x) = e^{it \Delta} v_{0},
\end{equation}
where $\| v_{0} \|_{L^{2}} = 1$ and $\hat{v}_{0}$ is supported on $P_{j}$ for some $j \geq i$. The function $v$ may be split into a piece supported in frequency near the $\hat{e}_{1}$ axis and a piece supported in frequency near the $\hat{e}_{2}$ axis. Suppose without loss of generality that $\hat{v}$ is supported near the $\hat{e}_{1}$ axis. Then set
\begin{equation}\label{6.33}
M(t) = \int |v(t, y)|^{2} \frac{(x - y)_{1}}{|(x - y)_{1}|} Im[\bar{u}_{\leq i - 3} \partial_{x_{1}} u_{\leq i - 3}] dx dy + \int |u_{\leq i - 3}|^{2} \frac{(x - y)_{1}}{|(x - y)_{1}|} Im[\bar{v} \partial_{x_{1}} v] dx dy.
\end{equation}
Let $F(u) = |u|^{2} u$. Then $u_{\leq i - 3}$ solves the equation
\begin{equation}\label{6.34}
i \partial_{t} u_{\leq i - 3} + \Delta u_{\leq i - 3} + F(u_{\leq i - 3}) = F(u_{\leq i - 3}) - P_{\leq i - 3} F(u) = -\mathcal N_{i - 3}.
\end{equation}
Making a direct computation,
\begin{equation}\label{6.35}
\aligned
\frac{d}{dt} M(t) = 2 \int \int_{x_{1} = y_{1}} |\partial_{x_{1}}(\overline{v(t,y)} u_{\leq i - 3}(t,x))|^{2} dx dy - \int \int_{x_{1} = y_{1}} |v(t,y)|^{2} |u_{\leq i - 3}(t,x)|^{4} dx dy \\
+ \int |v(t,y)|^{2} \frac{(x - y)_{1}}{|(x - y)_{1}|} \cdot Re[\overline{u_{\leq i - 3}} \partial_{x_{1}} \mathcal N_{i - 3}](t,x) dx dy \\ -  \int |v(t,y)|^{2} \frac{(x - y)_{1}}{|(x - y)_{1}|} \cdot Re[\overline{\mathcal N_{i - 3}} \partial_{x_{1}} u_{\leq i - 3}](t,x) dx dy \\
+ 2 \int Im[\overline{u_{\leq i - 3}} \mathcal N_{i - 3}](t, y) \frac{(x - y)_{1}}{|(x - y)_{1}|} \cdot Im[\bar{v} \partial_{x_{1}} v](t,x) dx dy.
\endaligned
\end{equation}
Then by the fundamental theorem of calculus, Bernstein's inequality, the Fourier support of $\bar{v} u_{\leq i - 3}$, $\| v_{0} \|_{L^{2}} = 1$, the fact that $\| u \|_{L^{2}} = \| Q \|_{L^{2}}$, and $(\ref{6.0.2})$,
\begin{equation}\label{6.36}
\aligned
2^{2j} \| \bar{v} u_{\leq i - 3} \|_{L_{t,x}^{2}(J \times \mathbb{R}^{2})}^{2} \lesssim 2^{i} \int_{x_{1} = y_{1}} \int_{x_{1} = y_{1}} |\partial_{x_{1}}(\overline{v(t,y)} u_{\leq i - 3}(t,x))|^{2} dx dy dt + \int_{J} \| \nabla u_{\leq i - 3} \|_{L^{p}}^{2} \| v \|_{L^{q}}^{2} dt
\\  \lesssim 2^{j + i} + 2^{i} \int \int_{x_{1} = y_{1}} |v(t,y)|^{2} |u_{\leq i - 3}(t,x)|^{4} dx dy dt \\ -  2^{i} \int \int \int |v(t,y)|^{2} \frac{(x - y)_{1}}{|(x - y)_{1}|} \cdot Re[\overline{\mathcal N_{i - 3}} \partial_{x_{1}} u_{\leq i - 3}](t,x) dx dy dt \\ + 2^{i} \int \int \int |v(t,y)|^{2} \frac{(x - y)_{1}}{|(x - y)_{1}|} Re[\overline{u_{\leq i - 3}} \partial_{x_{1}} \mathcal N_{i - 3}](t,x) dx dy dt \\
+ 2^{i + 1} \int \int \int Im[\overline{u_{\leq i - 3}} \mathcal N_{i - 3}](t, y) \frac{(x - y)_{1}}{|(x - y)_{1}|} Im[\bar{v} \partial_{x_{1}} v](t,x) dx dy dt + \int \| \nabla u_{\leq i - 3} \|_{L^{p}}^{2} \| v \|_{L^{q}}^{2} dt.
\endaligned
\end{equation}
Also note that
\begin{equation}\label{6.37}
\| \bar{v} u_{\leq i - 3} \|_{L_{t,x}^{2}}^{2} = \| \bar{v} v \bar{u}_{\leq i - 3} u_{\leq i - 3} \|_{L_{t,x}^{1}} = \| v u_{\leq i - 3} \|_{L_{t,x}^{2}}^{2},
\end{equation}
so it is not too important to pay attention to complex conjugates in the proceeding calculations.

First, by $(\ref{6.2})$, Bernstein's inequality, and the fact that $\lambda(t) = \frac{1}{\eta_{1}}$,
\begin{equation}\label{6.38}
\aligned
\int_{x_{1} = y_{1}} |v(t,y)|^{2} |u_{\leq i - 3}(t,x)|^{4} dx dy \lesssim 2^{-2j} \| u_{\leq i - 3} \|_{L^{\infty}}^{2} \int_{x_{1} = y_{1}} |\partial_{x_{1}}(\overline{v(t,y)} u_{\leq i - 3}(t,x))|^{2} dx dy \\ \lesssim \eta_{0}^{2} 2^{2i - 2j} \int_{x_{1} = y_{1}} |\partial_{x_{1}}(\overline{v(t,y)} u_{\leq i - 3}(t,x))|^{2} dx dy.
\endaligned
\end{equation}
Since $j \geq i$ and $\eta_{0} \ll 1$ is small, 
\begin{equation}
\sum_{j \geq i} (\ref{6.38}) \ll \int_{x_{1} = y_{1}} |\partial_{x_{1}}(\bar{v}(t,y) u_{\leq i - 3}(t,x))|^{2} dx dy,
\end{equation} 
which is more than sufficient for our purposes. Next, Strichartz estimates, $(\ref{6.2})$, $(\ref{6.2.1})$, and $(\ref{6.2.4})$ imply
\begin{equation}\label{6.38.1}
\int \| v \|_{L^{20}}^{2} \| \nabla P_{\leq i - 3} u \|_{L^{5/2}} \| \nabla P_{\leq i - 3} u \|_{L^{2}} dt \lesssim \| v \|_{L_{t}^{20/9} L_{x}^{20}}^{2} \| \nabla P_{\leq i - 3} u \|_{L_{t}^{10} L_{x}^{5/2}} \| \nabla P_{\leq i - 3} u \|_{L_{t}^{\infty} L_{x}^{2}} \lesssim \eta_{0} 2^{2i} \| u \|_{X_{i - 3}}.
\end{equation}
This calculation is also sufficient since $\sum_{j \geq i} 2^{2i - 2j} \eta_{0} \| u \|_{X_{i - 3}}$ is bounded by the right hand side of $(\ref{6.29})$.

Now consider the term,
\begin{equation}\label{6.39}
\mathcal N_{i - 3} = P_{\leq i - 3} F(u) - F(u_{\leq i - 3}).
\end{equation}
Since by Fourier support arguments
\begin{equation}\label{6.40}
P_{\leq i - 3} F(u_{\leq i - 6}) - F(u_{\leq i - 6}) = 0,
\end{equation}
\begin{equation}\label{6.41}
\aligned
\mathcal N_{i} = P_{\leq i - 3}(2 |u_{\leq i - 6}|^{2} u_{\geq i - 6} + (u_{\leq i - 6})^{2} \overline{u_{\geq i - 6}}) - (2 |u_{\leq i - 6}|^{2} u_{i - 6 \leq \cdot \leq i - 3} + (u_{\leq i - 6})^{2} \overline{u_{i - 6 \leq \cdot \leq i - 3}}) \\
+ P_{\leq i - 3} O((u_{\geq i - 6})^{2} u) + O((u_{i - 6 \leq \cdot \leq i - 3})^{2} u) = \mathcal N_{i - 3}^{(1)} + \mathcal N_{i - 3}^{(2)}.
\endaligned
\end{equation}
Following $(\ref{6.13})$--$(\ref{6.25})$,
\begin{equation}\label{6.42}
\aligned
\| |\mathcal N_{i - 3}^{(2)}| |u_{\leq i - 3}| \|_{L_{t,x}^{1}} \lesssim \| |u_{\geq i - 6}|^{2} |u_{\leq i - 9}|^{2} \|_{L_{t,x}^{1}} + \| |u_{\geq i - 6}|^{2} |u_{\geq i - 9}|^{2} \|_{L_{t,x}^{1}} \\ \lesssim \| (u_{\geq i - 6})(u_{\leq i - 9}) \|_{L_{t,x}^{2}}^{2} + \| u_{\geq i - 6} \|_{L_{t,x}^{4}}^{2} \| u_{\geq i - 9} \|_{L_{t,x}^{4}}^{2} \lesssim \eta_{0}^{1/10} \| u \|_{X_{i - 3}(J \times \mathbb{R}^{2})}^{2}(1 + \| u \|_{X_{i - 3}(J \times \mathbb{R}^{2})}^{6}).
\endaligned
\end{equation}
Therefore, since $\| v_{0} \|_{L^{2}} = 1$,
\begin{equation}\label{6.43}
\aligned
-  \int \int \int |v(t,y)|^{2} \frac{(x - y)_{1}}{|(x - y)_{1}|} Re[\bar{\mathcal N}_{i - 3}^{(2)} \partial_{x_{1}} u_{\leq i - 3}](t,x) dx dy dt \\ + \int \int \int |v(t,y)|^{2} \frac{(x - y)_{1}}{|(x - y)_{1}|} Re[\bar{u}_{\leq i - 3} \partial_{x_{1}} \mathcal N_{i - 3}^{(2)}](t,x) dx dy dt \\
+ 2 \int \int \int Im[\bar{u}_{\leq i - 3} \mathcal N_{i - 3}^{(2)}](t, y) \frac{(x - y)_{1}}{|(x - y)_{1}|} Im[\bar{v} \partial_{x_{1}} v](t,x) dx dy dt \\ \lesssim \eta_{0}^{1/10}  2^{j} \| u \|_{X_{i - 3}([0, T] \times \mathbb{R}^{2})}^{2} (1 + \| u \|_{X_{i - 3}([0, T] \times \mathbb{R}^{2})}^{6}).
\endaligned
\end{equation}

Next, observe that
\begin{equation}\label{6.44}
2 P_{\leq i - 3}(|u_{\leq i - 6}|^{2} u_{\geq i - 6}) - 2(|u_{\leq i - 6}|^{2} u_{i - 6 \leq \cdot \leq i - 3}) = 2 P_{> i - 3}(|u_{\leq i - 6}|^{2} u_{i - 6 \leq \cdot \leq i - 3}) + 2 P_{\leq i - 3}(|u_{\leq i - 6}|^{2} u_{> i - 3}).
\end{equation}
The computation for
\begin{equation}
P_{\leq i - 3}((u_{\leq i - 6})^{2} \overline{u_{\geq i - 6}}) - (u_{\leq i - 6})^{2} (\overline{u_{i - 6 \leq \cdot \leq i - 3}})
\end{equation}
is similar.

Again following $(\ref{6.13})$--$(\ref{6.28})$,
\begin{equation}\label{6.45}
\aligned
\| P_{\leq i - 3}(|u_{\leq i - 6}|^{2} u_{> i - 3}) (u_{i - 6 \leq \cdot \leq i - 3}) \|_{L_{t,x}^{1}} \\ + \| P_{\leq i - 3}(|u_{\leq i - 6}|^{2} u_{> i - 3}) (u_{i - 6 \leq \cdot \leq i - 3}) \|_{L_{t,x}^{1}} \lesssim \eta_{0}^{1/10} \| u \|_{X_{i - 3}(J \times \mathbb{R}^{2})}^{2}(1 + \| u \|_{X_{i - 3}(J \times \mathbb{R}^{2})}^{6}).
\endaligned
\end{equation}
Finally, observe that the Fourier support of
\begin{equation}\label{6.46}
2 P_{> i - 3}(|u_{\leq i - 6}|^{2} u_{i - 6 \leq \cdot \leq i - 3})(u_{\leq i - 6}) + 2 P_{\leq i - 3}(|u_{\leq i - 6}|^{2} u_{> i - 3})(u_{\leq i - 6})
\end{equation}
is on frequencies $|\xi| \geq 2^{i - 6}$. Therefore, integrating by parts,
\begin{equation}\label{6.47}
\aligned
\int \int \int Im[\bar{u}_{\leq i - 6} P_{> i - 3}(|u_{\leq i - 6}|^{2} u_{i - 6 \leq \cdot \leq i - 3})](t, y) \frac{(x - y)_{1}}{|(x - y)_{1}|} Im[\bar{v} \partial_{x_{1}} v](t,x) dx dy dt \\
= \int \int \int_{x_{1} = y_{1}} Im[\bar{v} \partial_{x_{1}} v](t,y) \cdot \frac{\partial_{x_{1}}}{\partial_{x_{1}}^{2}} Im[\bar{u}_{\leq i - 6} P_{> i - 3}(|u_{\leq i - 6}|^{2} u_{i - 6 \leq \cdot \leq i - 3})](t, x)  dx dy dt \\
\lesssim 2^{-i} (\int \int \int_{x_{1} = y_{1}} |\bar{v} (u_{\leq i - 6})|^{2} dx dy dt)^{1/2} (\int \int \int_{x_{1} = y_{1}} |\partial_{x_{1}} \bar{v} (u_{\leq i - 6})|^{2} dx dy dt)^{1/2} \| u_{\leq i - 3} \|_{L_{t,x}^{\infty}}^{2} \\
\lesssim 2^{j + i} \eta_{0}^{2} (\int \int \int_{x_{1} = y_{1}} |\bar{v} (u_{\leq i - 6})|^{2} dx dy dt).
\endaligned
\end{equation}
A similar calculation gives the estimate
\begin{equation}\label{6.47.1}
\aligned
\int \int \int Im[\bar{u}_{\leq i - 6} P_{\leq i - 3}(|u_{\leq i - 6}|^{2} u_{> i - 3})](t, y) \frac{(x - y)_{1}}{|(x - y)_{1}|} Im[\bar{v} \partial_{x_{1}} v](t,x) dx dy dt \\ \lesssim 2^{j + i} \eta_{0}^{2} (\int \int \int_{x_{1} = y_{1}} |\bar{v} (u_{\leq i - 6})|^{2} dx dy dt).
\endaligned
\end{equation}
We can analyze the terms
\begin{equation}\label{6.48}
\int \int \int |v(t,y)|^{2} \frac{(x - y)_{1}}{|(x - y)_{1}|} Re[\bar{\mathcal N}_{i - 3}^{(1)} \partial_{x_{1}} u_{\leq i - 3}](t,x) dx dy dt,
 \end{equation}
 and
 \begin{equation}\label{6.49}
\int \int  \int |v(t,y)|^{2} \frac{(x - y)_{1}}{|(x - y)_{1}|} Re[\bar{u}_{\leq i - 3} \partial_{x} \mathcal N_{i - 3}^{(1)}](t,x) dx dy dt,
\end{equation}
in a similar manner.\medskip

Plugging $(\ref{6.37})$--$(\ref{6.49})$ into $(\ref{6.36})$ gives
\begin{equation}\label{6.50}
2^{2j} \| \bar{v} u_{\leq i - 3} \|_{L_{t,x}^{2}}^{2} + 2^{2j} \| v u_{\leq i - 3} \|_{L_{t,x}^{2}}^{2} \lesssim 2^{i + j} + \eta_{0}^{1/10} 2^{i + j} (1 + \| u \|_{X_{i - 3}}^{8}).
\end{equation}
Summing up over $j \geq i$ implies $(\ref{6.28})$.

The estimate of
\begin{equation}\label{6.50.1}
\| (P_{\geq i} u)(P_{\leq i - 3} u) \|_{L_{t,x}^{2}},
\end{equation}
also uses an interaction Morawetz estimate. This time, for a fixed $j \geq i$, define the Morawetz potential,
\begin{equation}\label{6.50.2}
M_{j}(t) = \int |P_{j} u(t, y)|^{2} \frac{(x - y)_{1}}{|(x - y)_{1}|} Im[\overline{u_{\leq i - 3}} \partial_{x_{1}} u_{\leq i - 3}] dx dy + \int |u_{\leq i - 3}|^{2} \frac{(x - y)_{1}}{|(x - y)_{1}|} Im[\overline{P_{j} u} \partial_{x_{1}} P_{j} u] dx dy.
\end{equation}
Making a direct computation, since $P_{j} u$ is not a solution to the linear Schr{\"o}dinger equation, we have three additional terms,
\begin{equation}\label{6.50.3}
\aligned
\frac{d}{dt} M_{j}(t) = 2 \int \int_{x_{1} = y_{1}} |\partial_{x_{1}}(\bar{P_{j} u}(t,y) u_{\leq i - 3}(t,x))|^{2} dx dy - \int \int_{x_{1} = y_{1}} |P_{j} u(t,y)|^{2} |u_{\leq i - 3}(t,x)|^{4} dx dy \\
+ \int \int |P_{j} u(t,y)|^{2} \frac{(x - y)_{1}}{|(x - y)_{1}|} \cdot Re[\overline{u_{\leq i - 3}} \partial_{x_{1}} \mathcal N_{i - 3}](t,x) dx dy \\ -  \int \int |P_{j} u(t,y)|^{2} \frac{(x - y)_{1}}{|(x - y)_{1}|} \cdot Re[\overline{\mathcal N_{i - 3}} \partial_{x_{1}} u_{\leq i - 3}](t,x) dx dy \\
+ 2 \int \int Im[\overline{u_{\leq i - 3}} \mathcal N_{i - 3}](t, y) \frac{(x - y)_{1}}{|(x - y)_{1}|} \cdot Im[\bar{P_{j} u} \partial_{x_{1}} P_{j} u](t,x) dx dy  \\
+ \int \int |P_{\leq i - 3} u|^{2} \frac{(x - y)_{1}}{|(x - y)_{1}|} Re[\overline{P_{j} u} \partial_{x_{1}} P_{j}(|u|^{2} u)] dx dy \\
- \int \int |P_{\leq i - 3} u|^{2} \frac{(x - y)_{1}}{|(x - y)_{1}} Re[\overline{P_{j}(|u|^{2} u)} \partial_{x_{1}} P_{j} u] dx dy \\
+ 2 \int \int Im[\overline{P_{j} u} P_{j}(|u|^{2} u)] \frac{(x - y)_{1}}{|(x - y)_{1}} Im[\overline{u_{\leq i - 3}} \partial_{x_{1}} u_{\leq i - 3}] dx dy.
\endaligned
\end{equation}
Now then,
\begin{equation}\label{6.50.4}
\sum_{j \geq i} 2^{i-2j} \sup_{t \in J} |M_{j}(t)| \lesssim \| P_{\geq i} u \|_{L_{t}^{\infty} L_{x}^{2}}^{2} \lesssim \eta_{0}^{2}.
\end{equation}
Following $(\ref{6.37})$--$(\ref{6.49})$, $(\ref{6.1})$ and $(\ref{6.2})$ imply
\begin{equation}
\aligned
- \sum_{j \geq i} 2^{i - 2j} \int \int \int_{x_{1} = y_{1}} |P_{j} u(t,y)|^{2} |u_{\leq i - 3}(t,x)|^{4} dx dy dt \\
+ \sum_{j \geq i} 2^{i - 2j} \int \int \int |P_{j} u(t,y)|^{2} \frac{(x - y)_{1}}{|(x - y)_{1}|} \cdot Re[\overline{u_{\leq i - 3}} \partial_{x_{1}} \mathcal N_{i - 3}](t,x) dx dy dt \\ - \sum_{j \geq i} 2^{i - 2j} \int \int \int |P_{j} u(t,y)|^{2} \frac{(x - y)_{1}}{|(x - y)_{1}|} \cdot Re[\overline{\mathcal N_{i - 3}} \partial_{x_{1}} u_{\leq i - 3}](t,x) dx dy dt \\
+ 2 \sum_{j \geq i} 2^{i - 2j} \int \int \int Im[\overline{u_{\leq i - 3}} \mathcal N_{i - 3}](t, y) \frac{(x - y)_{1}}{|(x - y)_{1}|} \cdot Im[\overline{P_{j} u} \partial_{x_{1}} P_{j} u](t,x) dx dy dt \lesssim \eta_{0}^{2} (1 + \| u \|_{X_{i - 3}}^{8}).
\endaligned
\end{equation}
To analyze the new terms,
\begin{equation}\label{6.50.5.1}
\aligned
2 \sum_{j \geq i} 2^{i - 2j} \int \int \int Im[\overline{P_{j} u} P_{j}(|u|^{2} u)] \frac{(x - y)_{1}}{|(x - y)_{1}} Im[\overline{u_{\leq i - 3}} \partial_{x_{1}} u_{\leq i - 3}] dx dy dt \\
+ \sum_{j \geq i} 2^{i - 2j} \int \int \int |P_{\leq i - 3} u|^{2} \frac{(x - y)_{1}}{|(x - y)_{1}|} Re[\overline{P_{j} u} \partial_{x_{1}} P_{j}(|u|^{2} u)] dx dy dt \\
- \sum_{j \geq i} 2^{i - 2j}  \int \int \int |P_{\leq i - 3} u|^{2} \frac{(x - y)_{1}}{|(x - y)_{1}} Re[\overline{P_{j}(|u|^{2} u)} \partial_{x_{1}} P_{j} u] dx dy dt,
\endaligned
\end{equation}
first observe that by $(\ref{6.2})$, $(\ref{6.2.1})$, $(\ref{6.2.4})$, $(\ref{6.16})$, and $(\ref{6.42})$,
\begin{equation}
\aligned
2 \sum_{j \geq i} 2^{i - 2j} \int \int \int Im[\overline{P_{j} u} P_{j}(|u|^{2} u)] \frac{(x - y)_{1}}{|(x - y)_{1}} Im[\overline{u_{\leq i - 3}} \partial_{x_{1}} u_{\leq i - 3}] dx dy dt \\ \lesssim 2^{-i} \| \nabla P_{\leq i - 3} u \|_{L^{2}} \| P_{\leq i - 3} u \|_{L^{2}} (\| P_{\geq i - 6} u \|_{L_{t,x}^{4}}^{4} + \| (P_{\geq i - 3} u)(P_{\leq i - 6} u) \|_{L_{t,x}^{2}}^{2}) \\ \lesssim \eta_{0} (\eta_{0}^{1/10} \| u \|_{X_{i - 3}}^{2}(1 +  \| u \|_{X_{i - 3}}^{6})).
\endaligned
\end{equation}
For the other two terms in $(\ref{6.50.5.1})$, decompose 
\begin{equation}
\aligned
P_{j}(|u|^{2} u) = P_{j} (O((P_{\geq i - 3} u)^{3})) + P_{j}(O((P_{\geq i - 3} u)^{2}(P_{\leq i - 3} u))) \\ + P_{j}(2 |P_{\leq i - 3} u|^{2} (P_{> i - 3} u) + (P_{\leq i - 3} u)^{2} (P_{> i - 3} \bar{u})) = \mathcal N_{1} + \mathcal N_{2}.
\endaligned
\end{equation}
By $(\ref{6.15})$,
\begin{equation}
\aligned
\sum_{j \geq i} 2^{i - 2j} \int \int \int |P_{\leq i - 3} u|^{2} \frac{(x - y)_{1}}{|(x - y)_{1}|} Re[\overline{P_{j} u} \partial_{x_{1}} P_{j}\mathcal N_{1}] dx dy dt \\
- \sum_{j \geq i} 2^{i - 2j}  \int \int \int |P_{\leq i - 3} u|^{2} \frac{(x - y)_{1}}{|(x - y)_{1}} Re[\mathcal N_{1} \partial_{x_{1}} P_{j} u] dx dy dt \lesssim \| P_{\geq i - 3} u \|_{L_{t}^{3} L_{x}^{6}}^{3} \lesssim \eta_{0}^{1/2} \| u \|_{X_{i - 3}}^{5/2}.
\endaligned
\end{equation}
Next, if $m$ is a smooth Fourier multiplier satisfying $\nabla m(\xi) \lesssim \frac{1}{|\xi|}$, $|m(\xi + \xi_{1}) - m(\xi)| \lesssim \frac{|\xi_{1}}{|\xi|}$ for $|\xi_{1}| \ll |\xi|$, and therefore, as in $(\ref{6.15})$,
\begin{equation}\label{6.50.5}
\aligned
\| (P_{j} u) \cdot [P_{j}(|P_{\leq i - 3} u|^{2} (P_{> i - 3} u)) - |P_{\leq i - 3} u|^{2} (P_{j} u)] \|_{L_{t,x}^{1}} \\ + \| (P_{j} u) \cdot [P_{j}((P_{\leq i - 3} u)^{2}(\overline{P_{> i - 3} u)}) - (P_{\leq i - 3} u)^{2}(\overline{P_{j} u})] \|_{L_{t,x}^{1}} \\
\lesssim \| \nabla P_{\leq i - 3} u \|_{L_{t}^{5} L_{x}^{10/3}} \| u_{\leq i - 3} \|_{L_{t}^{\infty} L_{x}^{2}} \| P_{\geq i - 3} u \|_{L_{t}^{5/2} L_{x}^{10}}^{2} \lesssim \eta_{0}^{1/2} \| u \|_{X_{i - 3}}^{5/2}.
\endaligned
\end{equation}
Also by the product rule,
\begin{equation}\label{6.50.6}
\aligned
2 Re[\overline{P_{j} u} \partial_{x_{1}} (|P_{\leq i - 3} u|^{2} (P_{j} u))] - 2 Re[\mathcal |P_{\leq i - 3} u|^{2} (\overline{P_{j} u}) \partial_{x_{1}} P_{j} u] \\ + Re[\overline{P_{j} u} \partial_{x_{1}} ((P_{\leq i - 3} u)^{2} (P_{j} \bar{u}))] - Re[(\overline{P_{\leq i - 3} u})^{2} (P_{j} u) \partial_{x_{1}} (P_{j} u)] = O((\nabla P_{\leq i - 3} u)(P_{\leq i - 3} u)(P_{j} u)^{2}).
\endaligned
\end{equation}
Using the analysis in $(\ref{6.50.5})$, $(\ref{6.50.6}) \lesssim \eta_{0}^{1/2} \| u \|_{X_{i - 3}}^{5/2}$. Finally, integrating by parts and using $(\ref{6.38})$,
\begin{equation}\label{6.50.7}
\aligned
\sum_{j \geq i} 2^{i - 2j} \int \int \int |P_{\leq i - 3} u|^{2} \frac{(x - y)_{1}}{|(x - y)_{1}|} Re[\partial_{x_{1}} ((P_{j} u)^{2}(P_{\leq i - 3} u)^{2}] dx dy dt \\ \ll \int_{x_{1} = y_{1}} |\partial_{x_{1}}(\bar{v}(t,y) u_{\leq i - 3}(t,x))|^{2} dx dy.
\endaligned
\end{equation}

Therefore, by the fundamental theorem of calculus,
\begin{equation}
\aligned
\sum_{j \geq i} 2^{-2j} \| \partial_{x_{1}}(P_{j} u \bar{u}_{\leq i - 3}) \|_{L_{t,x}^{2}}^{2} \\ \lesssim \sum_{j \geq i} 2^{i - 2j} \int \int_{x_{1} = y_{1}} |\partial_{x_{1}}(\overline{P_{j} u}(t,y) u_{\leq i - 3}(t,x))|^{2} dx dy dt \lesssim \eta_{0}^{2} + \eta_{0}^{1/2}(1 + \| u \|_{X_{i - 3}}^{8}).
\endaligned
\end{equation}
It is possible to perform the same analysis with $\partial_{x_{1}}$ replaced by $\partial_{x_{2}}$, so
\begin{equation}
\sum_{j \geq i} 2^{-2j} \| \nabla(P_{j} u \bar{u}_{\leq i - 3}) \|_{L_{t,x}^{2}}^{2} \lesssim \eta_{0}^{2} + \eta_{0}^{1/2} \| u \|_{X_{i - 3}}^{4}.
\end{equation}
Using the Fourier support properties of $P_{j} u (\overline{u_{\leq i - 3}})$ implies
\begin{equation}\label{6.50.7}
\| (P_{\geq i} u)(P_{\leq i - 3} u) \|_{L_{t,x}^{2}}^{2} \lesssim \eta_{0}^{2} + \eta_{0}^{1/2} \| u \|_{X_{i - 3}}^{4}
\end{equation}
Combining $(\ref{6.30})$ with $(\ref{6.50.7})$ and $(\ref{6.8})$, and arguing by induction on $i$ implies Theorem $\ref{t6.2}$.
\end{proof}

Theorem $\ref{t6.2}$ may be upgraded when $u$ is close to a soliton, as in Theorem $\ref{t2.2}$.
\begin{theorem}\label{t6.3}
When $J = [0, \eta_{1}^{-2} T]$ is an interval that satisfies $(\ref{6.2.4})$, and $u$ satisfies Theorem $\ref{t2.2}$,
\begin{equation}\label{6.51}
\| P_{\geq k} u \|_{U_{\Delta}^{p}([0, \eta_{1}^{-2} T] \times \mathbb{R}^{2})} 
+ \| (P_{\geq k} u)(P_{\leq k - 3} u) \|_{L_{t,x}^{2}([0, \eta_{1}^{-2} T] \times \mathbb{R}^{2})} 
\lesssim (\frac{\eta_{1}^{2}}{T} \int_{0}^{\eta_{1}^{-2} T}  \| \epsilon(t) \|_{L^{2}}^{2} dt)^{1/2} + \frac{1}{T^{10}}.
\end{equation}
\end{theorem}
\begin{proof}
Make another induction on frequency argument starting at level $\frac{k}{2}$. First, observe that since Theorem $\ref{t6.2}$ is invariant under translation in time, for any $a \in \mathbb{Z}$,
\begin{equation}\label{6.54}
\| P_{\geq \frac{k}{2}} u \|_{U_{\Delta}^{2}([a \eta_{1}^{-1} T^{1/2}, (a + 1) \eta_{1}^{-1} T^{1/2}] \times \mathbb{R}^{2})} \lesssim 1.
\end{equation}
Next, following Theorem $\ref{t6.2}$,
\begin{equation}\label{6.55}
\aligned
\| P_{\geq \frac{k}{2} + 3} u \|_{U_{\Delta}^{2}([512 a \eta_{1}^{-1} T^{1/2}, 512 (a + 1) \eta_{1}^{-1} T^{1/2}] \times \mathbb{R}^{2})} \lesssim \inf_{t \in [512 a \eta_{1}^{-1} T^{1/2}, 512 (a + 1) \eta_{1}^{-1} T^{1/2}]} \| P_{\geq \frac{k}{2} + 3} u(t) \|_{L^{2}} \\ + \eta_{0} \| P_{\geq \frac{k}{2}} u \|_{U_{\Delta}^{2}([512 a \eta_{1}^{-1} T^{1/2}, 512 (a + 1) \eta_{1}^{-1} T^{1/2}] \times \mathbb{R}^{2})}.
\endaligned
\end{equation}
Now then, by $(\ref{6.16})$, $(\ref{6.2.4})$, and the fact that $Q$ is smooth and all its derivatives are rapidly decreasing,
\begin{equation}\label{6.56}
\| P_{\geq \frac{k}{2} + 3} u(t) \|_{L^{2}} \leq \| \epsilon(t) \|_{L^{2}} + \| P_{\geq \frac{k}{2} + 3} (\lambda(t)^{-1} e^{-i \gamma(t)} e^{-i x \cdot \frac{\xi(t)}{\lambda(t)}} Q(\frac{x - x(t)}{\lambda(t)}) \|_{L^{2}} \lesssim \| \epsilon(t) \|_{L^{2}} + T^{-10}.
\end{equation}

Plugging $(\ref{6.56})$ back into $(\ref{6.55})$,
\begin{equation}\label{6.59}
\aligned
\| P_{\geq \frac{k}{2} + 3} u \|_{U_{\Delta}^{2}([512 a \eta_{1}^{-1} T^{1/2}, 512 (a + 1) \eta_{1}^{-1} T^{1/2}] \times \mathbb{R}^{2})} \lesssim (\frac{\eta_{1}}{512 T^{1/2}} \int_{512a \eta_{1}^{-1} T^{1/2}}^{512 (a + 1) \eta_{1}^{-1} T^{1/2}}  \| \epsilon(t, x) \|_{L^{2}}^{2} dt)^{1/2} \\ +  T^{-10} + \eta_{0} (\sum_{j = 1}^{512} \| P_{\geq \frac{k}{2}} u \|_{U_{\Delta}^{2}([(512 a + (j - 1)) \eta_{1}^{-1} T^{1/2} , (512 a + j) \eta_{1}^{-1} T^{1/2}] \times \mathbb{R}^{2})}^{2})^{1/2}. 
\endaligned
\end{equation}
Arguing by induction in $k$, taking $\lfloor \frac{k}{6} \rfloor$ steps in all, for $\eta_{0}$ sufficiently small,
\begin{equation}\label{6.60}
\aligned
\| P_{\geq k} u \|_{U_{\Delta}^{2}([0, T] \times \mathbb{R}^{2})} \lesssim T^{-10} + 2^{k/2} \eta_{0}^{-\frac{k}{6}} + (\frac{1}{T} \int_{0}^{T} \inf_{\lambda, \gamma} \| e^{i \gamma} \lambda^{1/2} u(t, \lambda x) - Q(x) \|_{L^{2}}^{2} dt)^{1/2} \\
\lesssim T^{-10} + (\frac{1}{T} \int_{0}^{T} \inf_{\lambda, \gamma} \| e^{i \gamma} \lambda^{1/2} u(t, \lambda x) - Q(x) \|_{L^{2}}^{2} dt)^{1/2}.
\endaligned
\end{equation}
Indeed, if $C$ is the implicit constant in $(\ref{6.59})$ then for $\eta_{0} \ll 1$ sufficiently small,
\begin{equation}\label{6.61}
(C \eta_{0})^{\lfloor \frac{k}{6} \rfloor} \leq T^{-10}.
\end{equation}
\end{proof}

Throughout the proof, we assumed that $\lambda(t) = \frac{1}{\eta_{1}}$, although most of the time we only needed $\lambda(t) \geq \frac{1}{\eta_{1}}$. For the general case when $\lambda(t) \geq \frac{1}{\eta_{1}}$, replace the intervals in $(\ref{6.6})$ with intervals on which
\begin{equation}
\int_{J_{a}} \lambda(t)^{-2} dt = \eta_{1}^{2},
\end{equation}
and then argue by induction in exactly the same manner.

\begin{corollary}\label{c6.4}
When $J$ is an interval that satisfies
\begin{equation}
\frac{1}{\eta_{1}} \leq \lambda(t), \qquad \text{and} \qquad \frac{|\xi(t)|}{\lambda(t)} \leq \eta_{0}, \qquad \text{for all} \qquad t \in J,
\end{equation}
 and $(\ref{6.2.3})$, and $u$ satisfies the conditions of Theorem $\ref{t2.2}$,
\begin{equation}\label{6.62}
\| P_{\geq k} u \|_{U_{\Delta}^{p}(J \times \mathbb{R}^{2})} 
+ \| (P_{\geq k} u)(P_{\leq k - 3} u) \|_{L_{t,x}^{2}(J \times \mathbb{R}^{2})} 
\lesssim (\frac{1}{T} \int_{J}  \| \epsilon(t) \|_{L^{2}}^{2} \lambda(t)^{-2} dt)^{1/2} + \frac{1}{T^{10}}.
\end{equation}
\end{corollary}

\section{A long time Strichartz estimate when $d \geq 3$}
In dimensions $d \geq 3$, the proof of long time Strichartz estimates is complicated by the fact that $F(x) = |x|^{4/d} x$ is not a smooth function. To circumvent this difficulty, we will utilize a bound on $\frac{\nabla Q}{Q^{1 - \alpha}}$ for $\alpha > 0$.

\begin{theorem}\label{t5.0}
If $Q$ is the unique positive solution to $\Delta Q + Q^{1 + 4/d} = Q$, $\frac{\nabla Q}{Q^{1 - \alpha}}$ is uniformly bounded for any $\alpha > 0$.
\end{theorem}
\begin{proof}
To see this, first observe that since $Q$ is smooth and positive, $\frac{\nabla Q}{Q}$ is uniformly bounded on the set $\{ x : |x| \leq 1 \}$.

Next, since $Q$ is radially symmetric,
\begin{equation}
Q_{rr} + \frac{d - 1}{r} Q_{r} = r^{-(d - 1)} \partial_{r} (r^{d - 1} Q_{r}) = Q - Q^{1 + 4/d}.
\end{equation}
By the fundamental theorem of calculus, since $Q$ and all its derivatives are rapidly decreasing, and $Q$ is strictly decreasing, for any $\alpha > 0$,
\begin{equation}
r^{d - 1} Q_{r} \leq \int_{r}^{\infty} s^{d - 1} Q(s) ds \lesssim_{\alpha} Q(r)^{1 - \alpha}.
\end{equation}
This gives a bound on $\frac{\nabla Q}{Q^{1 - \alpha}}$ when $r > 1$.
\end{proof}

As before, let $J = [a, b]$ be an interval, and choose
\begin{equation}\label{5.1}
0 < \eta_{1} \ll \eta_{0} \ll 1,
\end{equation}
such that
\begin{equation}\label{5.2}
\sup_{t \in J} \| \epsilon(t,x) \|_{L^{2}}^{2} \leq \eta_{0}^{2},
\end{equation}
and that
\begin{equation}\label{5.3}
\int_{|\xi| \geq \eta_{1}^{-1/2}} |\hat{Q}(\xi)|^{2} \leq \eta_{0}^{2},
\end{equation}
where $Q$ is the soliton solution $(\ref{1.12})$. Suppose
\begin{equation}\label{5.4}
\frac{1}{\eta_{1}} \leq \lambda(t) \leq \frac{1}{\eta_{1}} T^{1/50d}, \qquad \text{and} \qquad \frac{|\xi(t)|}{\lambda(t)} \leq \eta_{0}, \qquad \text{for all} \qquad t \in J,
\end{equation}
and that there exists $k \in \mathbb{Z}_{\geq 0}$ such that
\begin{equation}\label{5.5}
\aligned
\int_{J} \lambda(t)^{-2} dt = T, \qquad \text{and} \qquad \eta_{1}^{-2} T = 2^{\alpha_{d} k}, \qquad \text{where} \qquad \alpha_{d} &= 3 - \frac{1}{5d} \qquad \text{when} \qquad 3 \leq d \leq 8, \\
\alpha_{d} &= 2(1 + \frac{4}{d}) - \frac{1}{5d}, \qquad \text{when} \qquad d \geq 9.
\endaligned
\end{equation}

Once again, when $i \in \mathbb{Z}$, $i > 0$, let $P_{i}$ denote the standard Littlewood-Paley projection operator. When $i = 0$, let $P_{i}$ denote the projection operator $P_{\leq 0}$, and when $i < 0$, let $P_{i}$ denote the zero operator.\medskip

\begin{theorem}\label{t5.1}
If $J$ is an interval on which $|x(t)| \leq T^{\frac{1}{2000d^{2}}}$, and $(\ref{5.4})$ and $(\ref{5.5})$ hold, letting $p$ satisfy $\frac{1}{p} = \frac{1}{2} - \frac{1}{4000d^{2}}$, for $\frac{k}{10d} \leq i \leq k(1 + \frac{1}{10d})$,
\begin{equation}\label{5.6}
\| P_{\geq i} u \|_{U_{\Delta}^{p}(J \times \mathbb{R}^{d})} + \| P_{\geq i} u \|_{L_{t}^{2} L_{x}^{\frac{2d}{d - 2}}(J \times \mathbb{R}^{d})} \lesssim 2^{\frac{\alpha_{d}}{2}(k(1 + \frac{1}{10d}) - i)}(\frac{1}{T} \int_{J} \| \epsilon(t) \|_{L^{2}}^{2} \lambda(t)^{-2} dt)^{1/2} + T^{-10}.
\end{equation}
\end{theorem}
\begin{proof}
This theorem is also proved using induction on frequency. Make the decomposition
\begin{equation}\label{5.6.1}
u(t,x) = e^{-i \gamma(t)} e^{-i x \cdot \frac{\xi(t)}{\lambda(t)}} \lambda(t)^{-d/2} Q(\frac{x - x(t)}{\lambda(t)}) + e^{-i \gamma(t)} e^{-i x \cdot \frac{\xi(t)}{\lambda(t)}} \lambda(t)^{-d/2} Q(\frac{x - x(t)}{\lambda(t)}) = \tilde{\epsilon}(t,x) + \tilde{Q}.
\end{equation}
By $(\ref{5.4})$, $(\ref{5.5})$, the fact that $Q$ is smooth and all its derivatives are rapidly decreasing, and local well-posedness theory,
\begin{equation}\label{5.7}
\| P_{\geq \frac{k}{10d}} u \|_{L_{t}^{2} L_{x}^{\frac{2d}{d - 2}}(J \times \mathbb{R}^{d})} \lesssim (\int_{J} \| \epsilon(t) \|_{L^{2}}^{2} \lambda(t)^{-2} dt)^{1/2} + T^{-10} = 2^{\frac{\alpha_{d} k}{2}} (\frac{1}{T} \int_{J} \| \epsilon(t) \|_{L^{2}}^{2} \lambda(t)^{-2} dt)^{1/2} + T^{-10}.
\end{equation}

For $i \geq \frac{k}{10d}$, by Duhamel's principle, for $t_{0}, t \in J$, 
\begin{equation}\label{5.8}
u(t) = e^{i(t - t_{0}) \Delta} u(t_{0}) + i \int_{t_{0}}^{t} e^{i(t - \tau) \Delta} F(u) d\tau.
\end{equation}
Also, by $(\ref{5.4})$, $(\ref{5.5})$, and the fact that $Q$ is smooth and all its derivatives are rapidly decreasing, for $i \geq \frac{k}{10d}$, choosing $t_{0} \in J$ such that $\| \tilde{\epsilon}(t_{0}) \|_{L^{2}} = \inf_{t \in J} \| \tilde{\epsilon}(t) \|_{L^{2}}$,
\begin{equation}\label{5.8.1}
\| P_{\geq i} e^{it \Delta} u(t_{0}) \|_{L_{t}^{2} L_{x}^{\frac{2d}{d - 2}}(J \times \mathbb{R}^{d})} + \| P_{\geq i} e^{i t \Delta} u(t_{0}) \|_{U_{\Delta}^{p}(J \times \mathbb{R}^{d})} \lesssim (\frac{1}{T} \int_{J} \| \epsilon(t) \|_{L^{2}}^{2} \lambda(t)^{-2} dt)^{1/2} + T^{-10}.
\end{equation}

For the Duhamel term, observe that by the endpoint Strichartz estimates of \cite{keel1998endpoint} and \cite{hadac2009well} (compare to $(\ref{6.12})$),
\begin{equation}
\| \int_{t_{0}}^{t} e^{i(t - \tau) \Delta} F(\tau) d\tau \|_{L_{t}^{2} L_{x}^{\frac{2d}{d - 2}} \cap U_{\Delta}^{p}(J \times \mathbb{R}^{d})} \lesssim \| F \|_{L_{t}^{2} L_{x}^{\frac{2d}{d + 2}}(J \times \mathbb{R}^{d})}.
\end{equation}
We will split $F = F_{1} + F_{2}$, where $F_{1} \in L_{t}^{2} L_{x}^{\frac{2d}{d + 2}}$ and $F_{2}$ will be estimated in a different function space. Use Taylor's formula to expand the nonlinearity,
\begin{equation}\label{5.10}
F(u) = F(u_{\leq i}) + F(u) - F(u_{\leq i}) = F(u_{\leq i}) + \int_{0}^{1} F'(u_{\leq i} + s u_{> i}) \cdot (u_{> i}) ds.
\end{equation}
\begin{remark}
When proving Theorem $\ref{t5.1}$, it is not so important to distinguish between $u$ and $\bar{u}$.
\end{remark}

Split
\begin{equation}\label{5.12}
|F'(u_{\leq i} + s u_{> i})| \lesssim |e^{-i \gamma(t)} e^{-ix \cdot \frac{\xi(t)}{\lambda(t)}} \frac{1}{\lambda(t)^{d/2}} Q(\frac{x - x(t)}{\lambda(t)})|^{4/d} +  |e^{-i \gamma(t)} e^{-ix \cdot \xi(t)} \frac{1}{\lambda(t)^{d/2}} \epsilon(t, \frac{x - x(t)}{\lambda(t)})|^{4/d}.
\end{equation}
By $(\ref{5.2})$,
\begin{equation}\label{5.13}
\aligned
\| |P_{> i} u| ||e^{-i \gamma(t)} e^{-ix \cdot \frac{\xi(t)}{\lambda(t)}} \frac{1}{\lambda(t)^{d/2}} \epsilon(t, \frac{x - x(t)}{\lambda(t)})|^{4/d} \|_{L_{t}^{2} L_{x}^{\frac{2d}{d + 2}}} \lesssim \eta_{0}^{4/d} \| P_{> i} u \|_{L_{t}^{2} L_{x}^{\frac{2d}{d - 2}}}.
\endaligned
\end{equation}

Next, using local smoothing,
\begin{lemma}\label{lind}
Under the assumptions of Theorem $\ref{t5.1}$,
\begin{equation}\label{5.13.0}
\| |\frac{1}{\lambda(t)^{d/2}} Q(\frac{x - x(t)}{\lambda(t)})|^{4/d} |P_{\geq i} u| \|_{L_{t}^{2} L_{x}^{\frac{2d}{d + 2}}(J \times \mathbb{R}^{d})} \lesssim \eta_{1} T^{\frac{1}{2000d^{2}}} 2^{-\frac{i}{p}} \| P_{\geq i} u \|_{U_{\Delta}^{p}(J \times \mathbb{R}^{d})}.
\end{equation}
\end{lemma}

\begin{proof}[Proof of Lemma $\ref{lind}$]
If $v$ is a solution to $(i \partial_{t} + \Delta) v = 0$ and $\hat{v}_{0}$ is supported on $|\xi| \geq 2^{i}$, then for any $R > 0$,
\begin{equation}\label{5.13.1}
\| v \|_{L_{t,x}^{2}(\mathbb{R} \times \{ x : |x| \leq R \})} \lesssim R 2^{-\frac{i}{2}} \| v_{0} \|_{L^{2}}.
\end{equation}
Then, by $(\ref{5.13.1})$, $|x(t)| \leq T^{\frac{1}{2000d^{2}}}$, $(\ref{5.4})$, and the fact that $Q$ is rapidly decreasing,
\begin{equation}\label{5.14}
\| v |\frac{1}{\lambda(t)^{d/2}} Q(\frac{x - x(t)}{\lambda(t)})|^{4/d} \|_{L_{t}^{2} L_{x}^{\frac{2d}{d + 2}}} \lesssim \| v |\frac{1}{\lambda(t)^{d/2}} Q(\frac{x - x(t)}{\lambda(t)})|^{2/d} \|_{L_{t,x}^{2}} \| Q \|_{L^{2}}^{2/d} \lesssim \eta_{1}^{1/2} 2^{-\frac{i}{2}} T^{\frac{1}{4000d^{2}}} \| v_{0} \|_{L^{2}}.
\end{equation}
Also,
\begin{equation}\label{5.14.1}
\aligned
\| v |\frac{1}{\lambda(t)^{d/2}} Q(\frac{x - x(t)}{\lambda(t)})|^{4/d} \|_{L_{t}^{2} L_{x}^{\frac{2d}{d + 2}}(J \times \mathbb{R}^{d})} \\ \lesssim \| v |\frac{1}{\lambda(t)^{d/2}} Q(\frac{x - x(t)}{\lambda(t)})|^{2/d} \|_{L_{t,x}^{2}}^{2/p} \| v \|_{L_{t}^{\infty}  L_{x}^{2}}^{\frac{p - 2}{p}} \| Q \|_{L^{2}}^{2/d} \cdot (\int_{J} \lambda(t)^{-2} dt)^{\frac{1}{4000d^{2}}} \lesssim \eta_{1}^{\frac{1}{p}} 2^{-\frac{i}{p}} T^{\frac{1}{2000d^{2}}} \| v_{0} \|_{L^{2}}.
\endaligned
\end{equation}
Lemma $\ref{lind}$ follows by replacing $v$ with a $U_{\Delta}^{p}$ atom and summing up.
\end{proof}

Since $\eta_{1}^{\frac{1}{p}} 2^{-\frac{i}{p}} T^{\frac{1}{2000d^{2}}} \ll 1$ when $i \geq \frac{k}{10d}$, the contributions of $(\ref{5.13})$ and $(\ref{5.13.0})$ may be absorbed into the left hand side of $(\ref{5.6})$. Therefore, by $(\ref{5.8})$ and $(\ref{5.10})$, for any $\frac{k}{10d} \leq i \leq k(1 + \frac{1}{10d})$,
\begin{equation}\label{5.14.2}
\aligned
\| P_{\geq i} u \|_{U_{\Delta}^{p}(J \times \mathbb{R}^{d})} + \| P_{\geq i} u \|_{L_{t}^{2} L_{x}^{\frac{2d}{d - 2}}(J \times \mathbb{R}^{d})} \\ \lesssim \| \int_{t_{0}}^{t} e^{i(t - \tau) \Delta} P_{\geq i} F(u_{\leq i}) d\tau \|_{U_{\Delta}^{p} \cap L_{t}^{2} L_{x}^{\frac{2d}{d - 2}}} +  (\int_{J} \| \epsilon(t) \|_{L^{2}}^{2} \lambda(t)^{-2} dt)^{1/2} + T^{-10}.
\endaligned
\end{equation}
By $(\ref{1.12})$, $(\ref{5.4})$, and $(\ref{5.5})$, for $i \geq \frac{k}{10d}$,
\begin{equation}\label{5.9}
\| P_{\geq i} F(\tilde{Q}) \|_{L_{t}^{1} L_{x}^{2}(J \times \mathbb{R}^{d})} \lesssim T^{-10},
\end{equation}
so it only remains to compute
\begin{equation}\label{5.11}
\| \int_{t_{0}}^{t} e^{i(t - \tau) \Delta} P_{\geq i} [F(u_{\leq i}) - F(\tilde{Q})]d\tau \|_{U_{\Delta}^{p} \cap L_{t}^{2} L_{x}^{\frac{2d}{d - 2}}(J \times \mathbb{R}^{d})}.
\end{equation}
Once again from \cite{hadac2009well}
\begin{equation}\label{5.17.0}
\| \int_{t_{0}}^{t} e^{i(t - \tau) \Delta} F d\tau \|_{L_{t}^{2} L_{x}^{\frac{2d}{d - 2}} \cap U_{\Delta}^{p}} \lesssim \sup_{v} \int_{J} (v, F)_{L^{2}} dt,
\end{equation}
where $\| v \|_{U_{\Delta}^{p}(J \times \mathbb{R}^{d})} \lesssim \| v \|_{V_{\Delta}^{2}(J \times \mathbb{R}^{d})} = 1$. Therefore, suppose $\| v \|_{V_{\Delta}^{2}(J \times \mathbb{R}^{d})} = 1$ and $\hat{v}(t, \xi)$ is supported on $|\xi| \geq 2^{i}$. By Bernstein's inequality,
\begin{equation}\label{5.17.1}
\aligned
\int_{J} (v, F(u_{\leq i}) - F(\tilde{Q}))_{L^{2}} dt \lesssim 2^{-i} \int_{J} (v, (\nabla u_{\leq i}) F'(u_{\leq i}) - (\nabla \tilde{Q}) F'(\tilde{Q}))_{L^{2}} dt \\ = 2^{-i} \int_{J} (v, \nabla (u_{\leq i} - \tilde{Q}) F'(u_{\leq i}) + \nabla \tilde{Q} (F'(u_{\leq i}) - F'(\tilde{Q})))_{L^{2}} dt.
\endaligned
\end{equation}

To esimate $(\ref{5.17.1})$, it is useful to split into regions where $u_{\leq i} - \tilde{Q} \ll \tilde{Q}$ and where $\tilde{Q} \lesssim u_{\leq i} - \tilde{Q}$. Let $\psi$ be a smooth, cutoff function, $\psi(x) = 1$ for $|x| \geq \frac{1}{2}$, $\psi(x) = 0$ for $|x| \leq \frac{1}{4}$. Now abuse notation and let $\psi(x)$ denote $\psi(x) = \psi(\frac{P_{\leq i} u - \tilde{Q}}{\tilde{Q}})$. There exists a sequence of constants $c_{j}$ that are uniformly bounded such that
\begin{equation}\label{5.25}
\aligned
(1 - \psi^{2}(x)) [\nabla (u_{\leq i} - \tilde{Q}) F'(u_{\leq i}) + \nabla \tilde{Q} (F'(u_{\leq i}) - F'(\tilde{Q}))] \\ = (1 - \psi^{2}(x)) [\nabla (u_{\leq i} - \tilde{Q}) \sum_{j \geq 0} c_{j} \frac{(u_{\leq i} - \tilde{Q})^{j}}{\tilde{Q}^{j - 4/d}} + \nabla \tilde{Q} \sum_{j \geq 1} \frac{(u_{\leq i} - \tilde{Q})^{j}}{\tilde{Q}^{j - 4/d}}.
\endaligned
\end{equation}
Following the computations in the proof of Lemma $\ref{lind}$, since $\| v \|_{U_{\Delta}^{p}} \lesssim 1$,
\begin{equation}\label{5.25.1}
\aligned
2^{-i} \int_{J} (v,  (1 - \psi^{2}(x)) \nabla (u_{\leq i} - \tilde{Q}) \sum_{j \geq 0} c_{j} \frac{(u_{\leq i} - \tilde{Q})^{j}}{\tilde{Q}^{j - 4/d}})_{L^{2}} dt \\ \lesssim 2^{-i - \frac{i}{p}}  \int_{J} (v, \nabla (u_{\leq i} - \tilde{Q}) |\tilde{Q}|^{4/d})_{L^{2}} dt \lesssim \eta_{1}^{1/p} 2^{-\frac{i}{p} - i} T^{\frac{1}{2000d^{2}}} \| \nabla (u_{\leq i} - \tilde{Q}) \|_{L_{t}^{2} L_{x}^{\frac{2d}{d - 2}}}.
\endaligned
\end{equation}
Next, on the support of $1 - \psi^{2}(x)$, by the definition of $\tilde{Q}$, $(\ref{5.4})$, $(\ref{5.5})$, and Theorem $\ref{t5.0}$,
\begin{equation}\label{5.25.2}
\aligned
\nabla \tilde{Q} \sum_{j \geq 1} \frac{(u_{\leq i} - \tilde{Q})^{j}}{\tilde{Q}^{j - 4/d}} \lesssim \frac{|\xi(t)|}{\lambda(t)} \tilde{Q} \sum_{j \geq 1} \frac{(u_{\leq i} - \tilde{Q})^{j}}{\tilde{Q}^{j - 4/d}} + \lambda(t)^{2} \frac{\nabla Q(\frac{x - x(t)}{\lambda(t)})}{Q(\frac{x - x(t)}{\lambda(t)})^{1 - \frac{2}{d}}}  \tilde{Q}^{1 - \frac{2}{d}} \sum_{j \geq 1} \frac{(u_{\leq i} - \tilde{Q})^{j}}{\tilde{Q}^{j - 4/d}} \\
\lesssim \frac{|\xi(t)|}{\lambda(t)} \tilde{Q}^{4/d} |u_{\leq i} - \tilde{Q}| + \lambda(t)^{2} \frac{\nabla Q(\frac{x - x(t)}{\lambda(t)})}{Q(\frac{x - x(t)}{\lambda(t)})^{1 - \frac{2}{d}}} \tilde{Q}^{2/d} |u_{\leq i} - \tilde{Q}|.
\endaligned
\end{equation}
Again following the computations in the proof of Lemma $\ref{lind}$, since $\| v \|_{U_{\Delta}^{p}} \lesssim 1$, by $(\ref{5.4})$ and $(\ref{5.5})$,
\begin{equation}\label{5.25.3}
2^{-i} \int_{J} (v,  (1 - \psi^{2}(x)) \nabla \tilde{Q} \sum_{j \geq 1} \frac{(u_{\leq i} - \tilde{Q})^{j}}{\tilde{Q}^{j - 4/d}})_{L^{2}} dt \lesssim \eta_{1}^{1/p} \eta_{0} 2^{-\frac{i}{p} - i} T^{\frac{1}{2000d^{2}}} \| (u_{\leq i} - \tilde{Q}) \|_{L_{t}^{2} L_{x}^{\frac{2d}{d - 2}}}.
\end{equation}
Similarly, since $|\tilde{Q}| \lesssim |u_{\leq i} - \tilde{Q}|$ on the support of $\psi(x)$, by the definition of $\tilde{Q}$, $(\ref{5.4})$, $(\ref{5.5})$, and Theorem $\ref{t5.0}$,
\begin{equation}\label{5.25.4}
\aligned
\psi^{2}(x) (\nabla \tilde{Q} (F'(u_{\leq i}) - F'(\tilde{Q})) \lesssim \psi^{2}(x) \nabla \tilde{Q} \cdot |u_{\leq i} - \tilde{Q}|^{4/d} \\
\lesssim \psi^{2}(x) \frac{|\xi(t)|}{\lambda(t)} \tilde{Q} |u_{\leq i} - \tilde{Q}|^{4/d} + \psi(x)^{2} \lambda(t)^{-2} \frac{\nabla Q(\frac{x - x(t)}{\lambda(t)})}{Q(\frac{x - x(t)}{\lambda(t)})^{1 - 2/d}} |u_{\leq i} - \tilde{Q}|^{4/d} \tilde{Q}^{1 - 2/d} \\ 
\lesssim \eta_{0} |u_{\leq i} - \tilde{Q}| |\tilde{Q}|^{2/d} |\tilde{\epsilon}|^{2/d} + \eta_{1} \lambda(t)^{-1} \frac{\nabla Q(\frac{x - x(t)}{\lambda(t)})}{Q(\frac{x - x(t)}{\lambda(t)})^{1 - 2/d}} |u_{\leq i} - \tilde{Q}| |\tilde{\epsilon}|^{2/d}.
\endaligned
\end{equation}
Therefore,
\begin{equation}\label{5.25.5}
2^{-i} \int_{J} (v,  \psi^{2}(x) \nabla \tilde{Q} [F'(u_{\leq i}) - F'(\tilde{Q})])_{L^{2}} dt \lesssim \eta_{1}^{1/p} \eta_{0} 2^{-\frac{i}{p} - i} T^{\frac{1}{2000d^{2}}} \eta_{0} \| (u_{\leq i} - \tilde{Q}) \|_{L_{t}^{2} L_{x}^{\frac{2d}{d - 2}}}.
\end{equation}
Now decompose,
\begin{equation}\label{5.25.6}
\nabla (u_{\leq i} - \tilde{Q}) F'(u_{\leq i}) = \nabla (u_{\leq i} - \tilde{Q}) F'(u_{\leq i} - \tilde{Q}) + \nabla(u_{\leq i} - \tilde{Q}) [F'(u_{\leq i}) - F'(u_{\leq i} - \tilde{Q})].
\end{equation}
Since
\begin{equation}\label{5.25.7}
\nabla(u_{\leq i} - \tilde{Q}) [F'(u_{\leq i}) - F'(u_{\leq i} - \tilde{Q})] \lesssim |\nabla (u_{\leq i} - \tilde{Q})| |\tilde{Q}|^{4/d},
\end{equation}
then as in $(\ref{5.25.1})$,
\begin{equation}\label{5.25.8}
\aligned
2^{-i} \int_{J} (v,  \psi^{2}(x) |\nabla(u_{\leq i} - \tilde{Q}| [F'(u_{\leq i}) - F'(u_{\leq i} - \tilde{Q})])_{L^{2}} dt \lesssim \eta_{1}^{1/p} 2^{-\frac{i}{p} - i} T^{\frac{1}{2000d^{2}}} \| \nabla (u_{\leq i} - \tilde{Q}) \|_{L_{t}^{2} L_{x}^{\frac{2d}{d - 2}}}.
\endaligned
\end{equation}
Also, since
\begin{equation}\label{5.25.9}
(1 - \psi^{2}(x)) |\nabla(u_{\leq i} - \tilde{Q})| |F'(u_{\leq i}) - F'(\tilde{Q})| \lesssim |\nabla (u_{\leq i} - \tilde{Q})| |\tilde{Q}|^{4/d},
\end{equation}
by $(\ref{5.25.1})$,
\begin{equation}\label{5.25.10}
2^{-i} \int_{J} (v,  (1 - \psi^{2}(x)) |\nabla(u_{\leq i} - \tilde{Q})| [F'(u_{\leq i} - \tilde{Q})])_{L^{2}} dt \lesssim \eta_{1}^{1/p} 2^{-\frac{i}{p} - i} T^{\frac{1}{2000d^{2}}} \| \nabla (u_{\leq i} - \tilde{Q}) \|_{L_{t}^{2} L_{x}^{\frac{2d}{d - 2}}}.
\end{equation}

Therefore, it only remains to estimate
\begin{equation}\label{5.25.11}
\aligned
2^{-i} \| \int_{t_{0}}^{t} e^{i(t - \tau) \Delta} P_{\geq i}(\nabla(u_{\leq i} - \tilde{Q}) F'(u_{\leq i} - \tilde{Q})) d\tau \|_{L_{t}^{2} L_{x}^{\frac{2d}{d - 2}} \cap U_{\Delta}^{p}(J \times \mathbb{R}^{d})} \\ \lesssim 2^{-i} \| P_{\geq i} (\nabla(u_{\leq i} - \tilde{Q}) F'(u_{\leq i} - \tilde{Q})) \|_{L_{t}^{2} L_{x}^{\frac{2d}{d + 2}}(J \times \mathbb{R}^{d})}.
\endaligned
\end{equation}
For this term, use fractional derivative chain rule (Proposition $A.1$ of \cite{visan2007defocusing}),
\begin{proposition}\label{pvisan}
Let $F$ be a H{\"o}lder continuous function of order $0 < \alpha < 1$. Then for every $0 < \sigma < \alpha$, $1 < p < \infty$, and $\frac{\sigma}{\alpha} < s < 1$, we have
\begin{equation}\label{5.25.12}
\| |\nabla|^{\sigma} F(u) \|_{L^{p}} \lesssim \| |u|^{\alpha - \frac{\sigma}{s}} \|_{L^{p_{1}}} \| |\nabla|^{s} u \|_{L^{\frac{\sigma}{s} p_{2}}}^{\frac{\sigma}{s}},
\end{equation}
provided $\frac{1}{p} = \frac{1}{p_{1}} + \frac{1}{p_{2}}$ and $(1 - \frac{\sigma}{\alpha s}) p_{1} > 1$.
\end{proposition}
Then for any $0 \leq \sigma < 1 + \frac{4}{d}$,
\begin{equation}\label{5.25.13}
\aligned
2^{-i} \int_{J} (v, \nabla(u_{\leq i} - \tilde{Q}) F'(u_{\leq i} - \tilde{Q}))_{L^{2}} dt \\ \lesssim_{\sigma} 2^{-\sigma i} \| |\nabla|^{\sigma} (u_{\leq i} - \tilde{Q}) \|_{L_{t}^{2} L_{x}^{\frac{2d}{d - 2}}} \| u_{\leq i} - \tilde{Q} \|_{L_{t}^{\infty} L_{x}^{2}}^{4/d} \lesssim \eta_{0}^{4/d} 2^{-\sigma i} \| |\nabla|^{\sigma} (u_{\leq i} - \tilde{Q}) \|_{L_{t}^{2} L_{x}^{\frac{2d}{d - 2}}}.
\endaligned
\end{equation}

Therefore, arguing by induction on frequency, if $(\ref{5.6})$ holds for all $\frac{k}{10} \leq i \leq j_{0}$ for some $\frac{k}{10} \leq j_{0} \leq k$, then for $i = j_{0} + 1$, $(\ref{5.6})$ holds, which proves the Theorem $\ref{t5.1}$ by induction on frequency. Indeed, choosing $\sigma$ such that $\sigma > \frac{\alpha_{d}}{2}$, the contributions of $(\ref{5.11})$--$(\ref{5.25.13})$ may be bounded by
\begin{equation}\label{5.25.14}
\aligned
\eta_{1}^{1/p} 2^{-\frac{i}{p} - i} T^{\frac{1}{2000d^{2}}} \sum_{0 \leq j \leq i} 2^{j} \| P_{j}(u - \tilde{Q}) \|_{L_{t}^{2} L_{x}^{\frac{2d}{d - 2}}} + \eta_{0}^{4/d} \sum_{0 \leq j \leq i} 2^{\sigma j} \| P_{j}(u - \tilde{Q}) \|_{L_{t}^{2} L_{x}^{\frac{2d}{d - 2}}(J \times \mathbb{R}^{d})} \\
\lesssim (\eta_{1}^{1/p} 2^{-\frac{i}{11d}} T^{\frac{1}{2000d^{2}}}  + \eta_{0}^{4/d}) (2^{\frac{\alpha_{d}}{2}((1 + \frac{1}{10d})k - i)} (\frac{1}{T} \int_{J} \| \epsilon(t) \|_{L^{2}}^{2})^{1/2} + T^{-10}).
\endaligned
\end{equation}
When $i \geq \frac{k}{10d}$,
\begin{equation}
\eta_{1}^{1/p} 2^{-\frac{i}{11d}} T^{\frac{1}{2000d^{2}}}  + \eta_{0}^{4/d} \ll 1,
\end{equation}
so the induction on frequency step is complete.
\end{proof}

\section{Almost conservation of energy}
In two dimensions, the almost conservation of energy computation is exactly the same as in one dimension in \cite{dodson2021determination}. The computations in higher dimensions are more difficult since $F(x) = |x|^{4/d} x$ is not smooth. The good news is that most of the difficult computations have already been done in the previous section. 

\subsection{Almost conservation of energy in dimension $d = 2$}
\begin{theorem}\label{t4.1}
Let $J = [a, b]$ be an interval for which $(\ref{6.2.4})$ holds. Then,
\begin{equation}\label{4.2}
\sup_{t \in J} E(P_{\leq k + 9} u(t)) \lesssim \frac{2^{2k}}{T} \int_{J} \| \epsilon(t) \|_{L^{2}}^{2} \lambda(t)^{-2} dt + \sup_{t \in J} \frac{|\xi(t)|^{2}}{\lambda(t)^{2}} + 2^{2k} T^{-10}.
\end{equation}
\end{theorem}
\begin{proof}
Since $\| \epsilon(t) \|_{L^{2}}$ is continuous as a function of time, the mean value theorem implies that under the conditions of Theorem $\ref{t6.2}$, there exists $t_{0} \in [a, b] = J$ such that
\begin{equation}\label{4.1}
\| \epsilon(t_{0}) \|_{L^{2}}^{2} = \frac{1}{T} \int_{a}^{b} \| \epsilon(t) \|_{L^{2}}^{2} \lambda(t)^{-2} dt.
\end{equation}
Then by $(\ref{6.16})$, the fact that $Q$ is a smooth real valued function and all its derivatives are rapidly decreasing, and the Sobolev embedding theorem, taking $\epsilon = \epsilon_{1} + i \epsilon_{2}$,
\begin{equation}\label{4.4}
\aligned
E(P_{\leq k + 9} u) = \frac{1}{2} \| \nabla P_{\leq k + 9} u \|_{L^{2}}^{2} - \frac{1}{4} \| P_{\leq k + 9} u \|_{L^{4}}^{4} = \frac{1}{2 \lambda(t)^{2}} \| \nabla Q(x) \|_{L^{2}}^{2} + \frac{1}{2} \frac{|\xi(t)|^{2}}{\lambda(t)^{2}} \| Q \|_{L^{2}}^{2} - \frac{1}{4 \lambda(t)^{2}} \| Q \|_{L^{4}}^{4} \\
+ \frac{1}{\lambda(t)^{2}} (\nabla Q, \nabla \epsilon_{1})_{L^{2}} - \frac{\xi(t)}{\lambda(t)^{2}} \cdot (Q, \nabla \epsilon_{2}) + \frac{\xi(t)}{\lambda(t)^{2}} \cdot (\nabla Q, \epsilon_{2})_{L^{2}} + \frac{|\xi(t)|^{2}}{\lambda(t)^{2}} (Q, \epsilon_{1})_{L^{2}} - \frac{1}{\lambda(t)^{2}} (Q^{3}, \epsilon_{1})_{L^{2}} \\
+ O(2^{2k} \| \epsilon \|_{L^{2}}^{2} ) + O(2^{2k} T^{-10}).
\endaligned
\end{equation}
Since $(\epsilon_{2}, \nabla Q)_{L^{2}} = (\nabla \epsilon_{2}, Q)_{L^{2}} = 0$, $(\nabla Q, \nabla \epsilon_{1})_{L^{2}} - (Q^{3}, \epsilon_{1})_{L^{2}} = -(Q, \epsilon_{1})_{L^{2}} = \frac{1}{2} \| \epsilon \|_{L^{2}}^{2}$, and $E(Q) = 0$,
\begin{equation}\label{4.5}
\aligned
E(P_{\leq k + 9} u) = \frac{1}{2} \frac{|\xi(t)|^{2}}{\lambda(t)^{2}} \| Q \|_{L^{2}}^{2} + \frac{1}{2 \lambda(t)^{2}} \| \epsilon \|_{L^{2}}^{2} - \frac{|\xi(t)|^{2}}{2 \lambda(t)^{2}} \| \epsilon \|_{L^{2}}^{2} + O(2^{2k} \| \epsilon \|_{L^{2}}^{2} ) + O(2^{2k} T^{-10}),
\endaligned
\end{equation}
so $(\ref{4.2})$ holds at $t_{0}$.

Next compute the change of energy. This computation utilizes the computations in the Fourier truncation method of \cite{bourgain1998refinements}. See also the I-method in \cite{colliander2002almost}.
\begin{equation}\label{4.6}
\aligned
\frac{d}{dt} E(P_{\leq k + 9} u) = -( P_{\leq k + 9} u_{t}, \Delta P_{\leq k + 9} u )_{L^{2}} - ( P_{\leq k + 9} u_{t}, |P_{\leq k + 9} u|^{2} P_{\leq k + 9} u )_{L^{2}} \\ = -( P_{\leq k + 9} u_{t}, P_{\leq k + 9} F(u) - F(P_{\leq k + 9} u) )_{L^{2}} = ( i \Delta P_{\leq k + 9} u + i P_{\leq k + 9} (|u|^{2} u), P_{\leq k + 9} F(u) - F(P_{\leq k + 9} u) )_{L^{2}}.
\endaligned
\end{equation}

First compute
\begin{equation}\label{4.7}
\int_{t_{0}}^{t'} ( i \Delta P_{\leq k + 9} u, P_{\leq k + 9} F(u) - F(P_{\leq k + 9} u) )_{L^{2}} dt,
\end{equation}
for some $t' \in J$. Making a Littlewood--Paley decomposition,
\begin{equation}\label{4.8}
\aligned
\int_{t_{0}}^{t'} ( i \Delta P_{\leq k + 9} u, P_{\leq k + 9} F(u) - F(P_{\leq k + 9} u) )_{L^{2}} dt \\ \sim \sum_{0 \leq k_{3} \leq k_{2} \leq k_{1}} \sum_{0 \leq k_{4} \leq k + 9} \int_{t_{0}}^{t'} ( i \Delta P_{k_{4}} u, P_{\leq k + 9} (P_{k_{1}} u \cdots P_{k_{3}} u) - (P_{\leq k + 9} P_{k_{1}} u) \cdots (P_{\leq k + 9} P_{k_{3}} u) )_{L^{2}} dt.
\endaligned
\end{equation}
\begin{remark}
For these computations, it is not so important to distinguish between $u$ and $\bar{u}$.
\end{remark}

\noindent \textbf{Case $1$, $k_{1} \leq k + 6$:} In this case $P_{\leq k + 9} P_{k_{1}} = P_{k_{1}}$ and $P_{\leq k + 9}(P_{k_{1}} u \cdots P_{k_{3}} u) = P_{k_{1}} u \cdots P_{k_{3}} u$, so the contribution of these terms is zero. That is, for $k_{1}, ..., k_{3} \leq k + 6$,
\begin{equation}\label{4.9}
\aligned
 \int_{t_{0}}^{t'} ( i \Delta P_{k_{6}} u, P_{\leq k + 9} (P_{k_{1}} u \cdots P_{k_{3}} u) - (P_{\leq k + 9} P_{k_{1}} u) \cdots (P_{\leq k + 9} P_{k_{3}} u) )_{L^{2}} dt = 0.
\endaligned
\end{equation}

\noindent \textbf{Case $2$, $k_{1} \geq k + 6$ and $k_{2} \leq k$:} In this case, Fourier support properties imply that $k_{4} \geq k + 3$. Then by Theorem $\ref{t6.3}$,
\begin{equation}\label{4.10}
 \int_{t_{0}}^{t'} ( i \Delta P_{k + 3 \leq \cdot \leq k + 9} u, P_{\leq k + 9} ((P_{\leq k} u)^{2} (P_{\geq k + 6} u)) - (P_{\leq k} u)^{2} (P_{k + 6 \leq \cdot \leq k + 9} u) )_{L^{2}} dt
\end{equation}
\begin{equation}\label{4.11}
\aligned
\lesssim 2^{2k} \| (P_{\geq k + 3} u)(P_{\leq k} u) \|_{L_{t,x}^{2}} \| (P_{\geq k + 6} u)(P_{\leq k} u) \|_{L_{t,x}^{2}} \lesssim \frac{2^{2k}}{T} \int_{J} \| \epsilon(t) \|_{L^{2}}^{2} \lambda(t)^{-2} dt + 2^{2k} T^{-10}.
\endaligned
\end{equation}

\noindent \textbf{Case 3, $k_{1} \geq k + 6$, $k_{2} \geq k$ $k_{3} \leq k$:} If $k_{4} \leq k$, then by Fourier support properties, $k_{2} \geq k + 3$. In that case,
\begin{equation}\label{4.12}
\aligned
\int_{t_{0}}^{t'} ( i \Delta P_{\leq k} u, P_{\leq k + 9}((P_{\geq k + 6} u)(P_{\geq k + 3} u) (P_{\leq k} u)) - (P_{k + 6 \leq \cdot \leq k + 9} u)(P_{k + 3 \leq \cdot \leq k + 9} u)(P_{\leq k} u))_{L^{2}} dt \\ \lesssim 2^{2k} \| (P_{\geq k + 6} u)(P_{\leq k} u) \|_{L_{t,x}^{2}} \| (P_{\geq k + 3} u)(P_{\leq k} u) \|_{L_{t,x}^{2}} \lesssim \frac{2^{2k}}{T} \int_{J} \| \epsilon(t) \|_{L^{2}}^{2} \lambda(t)^{-2} dt + 2^{2k} T^{-10}.
\endaligned
\end{equation}
In the case when $k_{4} \geq k$,
\begin{equation}\label{4.13}
\aligned
\int_{t_{0}}^{t'} ( i \Delta P_{k \leq \cdot \leq k + 9} u, P_{\leq k + 9}((P_{\geq k + 6} u)(P_{\geq k} u) (P_{\leq k} u)) - (P_{k + 6 \leq \cdot \leq k + 9} u)(P_{k \leq \cdot \leq k + 9} u)(P_{\leq k} u))_{L^{2}} dt \\ \lesssim 2^{2k} \| (P_{\geq k + 6} u)(P_{\leq k} u) \|_{L_{t,x}^{2}} \| P_{\geq k} u \|_{L_{t,x}^{4}}^{2} \lesssim \frac{2^{2k}}{T} \int_{J} \| \epsilon(t) \|_{L^{2}}^{2} \lambda(t)^{-2} dt + 2^{2k} T^{-10}.
\endaligned
\end{equation}

\noindent \textbf{Case 4, $k_{1} \geq k + 6$ and $k_{2}, k_{3} \geq k$:} In this case,
\begin{equation}\label{4.14}
\aligned
\int_{t_{0}}^{t'} ( i \Delta P_{\leq k + 9} u, P_{\leq k + 9}((P_{\geq k + 6} u)(P_{\geq k} u)^{2}) - (P_{k + 6 \leq \cdot \leq k + 9} u)(P_{k \leq \cdot \leq k + 9} u)^{2} )_{L^{2}} dt \\ \lesssim 2^{2k} \| (P_{\geq k + 6} u)(P_{\leq k} u) \|_{L_{t,x}^{2}} \| P_{\geq k} u \|_{L_{t,x}^{4}}^{2}
+ 2^{2k} \| P_{\geq k} u \|_{L_{t, x}^{4}}^{4} \lesssim \frac{2^{2k}}{T} \int_{J} \| \epsilon(t) \|_{L^{2}}^{2} \lambda(t)^{-2} dt + 2^{2k} T^{-10}.
\endaligned
\end{equation}

The contribution of the nonlinear terms is similar, using the fact that
\begin{equation}\label{4.15}
( i P_{\leq k + 9} F(u), P_{\leq k + 9} F(u) - F(P_{\leq k + 9} u) )_{L^{2}} = ( i P_{\leq k + 9} F(u), F(P_{\leq k + 9} u) )_{L^{2}}.
\end{equation}
Then make a Littlewood--Paley decomposition,
\begin{equation}\label{4.16}
\aligned
(i P_{\leq k + 9} F(u), F(P_{\leq k + 9} u))_{L^{2}} \\ = \sum_{0 \leq k_{3} \leq k_{2} \leq k_{1}} \sum_{0 \leq k_{3}' \leq k_{2}' \leq k_{1}'} (i P_{\leq k + 9} (u_{k_{1}} \cdots u_{k_{3}}), (P_{\leq k + 9} P_{k_{1}} u) \cdots (P_{\leq k + 9} P_{k_{3}} u))_{L^{2}}.
\endaligned
\end{equation}
\noindent \textbf{Case 1: $k_{1}, k_{1}' \leq k + 6$:} Once again, if $k_{1}, k_{1}' \leq k + 6$, then the right hand side of $(\ref{4.16})$ is zero, since $P_{\leq k + 9}(u_{k_{1}} u_{k_{2}} u_{k_{3}}) = u_{k_{1}} u_{k_{2}} u_{k_{3}}$ and $P_{\leq k + 9} u_{k_{j}} = u_{k_{j}}$ for $j = 1, 2, 3$.\medskip

\noindent \textbf{Case 2: $k_{1}$ or $k_{1}' \geq k + 6$, four terms are $\leq k$:} In the case that $k_{1}$ or $k_{1}' \geq k + 6$, and four of the terms in $(\ref{4.16})$ are at frequency $\leq k$, then by Fourier support properties the final term should be at frequency $\geq k + 3$. The contribution in this case is bounded by
\begin{equation}\label{4.17}
\| (P_{\geq k + 6} u)(P_{\leq k} u) \|_{L_{t,x}^{2}} \| (P_{\geq k + 3} u)(P_{\leq k} u) \|_{L_{t,x}^{2}} \| P_{\leq k} u \|_{L_{t,x}^{\infty}}^{2} \lesssim 2^{2k} \frac{1}{T} \int_{J} \| \epsilon(t) \|_{L^{2}}^{2} dt + 2^{2k} T^{-10}.
\end{equation}

\noindent \textbf{Case 3: $k_{1}$ or $k_{1}' \geq k + 6$, two additional terms are $\geq k$:} The contribution of the case that $k_{1}$ or $k_{1}' \geq k + 6$, two additional terms in $(\ref{4.16})$ are at frequency $\geq k$, and the other three terms are at frequency $\leq k$ is bounded by
\begin{equation}\label{4.18}
\| (P_{\geq k + 6} u)(P_{\leq k} u) \|_{L_{t,x}^{2}} \| P_{\geq k} u \|_{L_{t,x}^{4}}^{2} \| P_{\leq k} u \|_{L_{t,x}^{\infty}}^{2}   \lesssim 2^{2k} \frac{1}{T} \int_{J} \| \epsilon(t) \|_{L^{2}}^{2} dt + 2^{2k} T^{-10}.
\end{equation}

\noindent \textbf{Case 4: $k_{1}$ or $k_{1}' \geq k + 6$ and at least three additional terms in $(\ref{4.16})$ are at frequencies $\geq k$.} 

This case may always be reduced to the estimate
\begin{equation}\label{4.19}
2^{2k} \| P_{\geq k} u \|_{L_{t,x}^{4}}^{4} \| u \|_{L_{t}^{\infty} L_{x}^{2}}^{2} \lesssim \frac{2^{2k}}{T} \int_{J} \| \epsilon(t) \|_{L^{2}}^{2} \lambda(t)^{-2} dt + 2^{2k} T^{-10}.
\end{equation}
Indeed, if both $P_{\leq k + 9}(u_{k_{1}} u_{k_{2}} u_{k_{3}})$ and $P_{\leq k + 9} u_{k_{1}} \cdot P_{\leq k + 9} u_{k_{2}} \cdot P_{\leq k+ 9} u_{k_{3}}$ have two terms at frequency $\leq k$, then place each term in $L_{t}^{2} L_{x}^{1}$ and then make a Sobolev embedding at frequencies $\leq k + 9$ to place each term in $L_{t,x}^{2}$.

If $P_{\leq k + 9}(u_{k_{1}} u_{k_{2}} u_{k_{3}})$ has three terms at frequency $\geq k$, then estimate this triple product in $L_{t,x}^{4/3}$, and place the term in $P_{\leq k + 9} u_{k_{1}} \cdot P_{\leq k + 9} u_{k_{2}} \cdot P_{\leq k+ 9} u_{k_{3}}$ at frequency $\geq k$ in $L_{t,x}^{4}$ and using the Sobolev embedding theorem, place the other two in $L_{t,x}^{\infty}$.

If $P_{\leq k + 9}(u_{k_{1}} u_{k_{2}} u_{k_{3}})$ has only one term at frequency $\geq k$, then estimate this triple product in $L_{t,x}^{4}$ and place the term $P_{k \leq \cdot \leq k + 9} u_{k_{1}} \cdot P_{k \leq \cdot \leq k + 9} u_{k_{2}} \cdot P_{k \leq \cdot \leq k+ 9} u_{k_{3}}$ in $L_{t,x}^{4/3}$.

This completes the proof of Theorem $\ref{t4.1}$.
\end{proof}

This bound on $E(P_{\leq k + 9} u)$ gives good bounds on the $L^{2}$ norm and $\dot{H}^{1}$ norms of $\epsilon$.

\begin{theorem}\label{t4.2}
If
\begin{equation}\label{4.20}
\frac{1}{\eta_{1}} \leq \lambda(t) \leq \frac{1}{\eta_{1}} T^{1/100}, \qquad \text{and} \qquad  \frac{|\xi(t)|}{\lambda(t)} \leq \eta_{0}, \qquad \text{for all} \qquad t \in J,
\end{equation}
and
\begin{equation}\label{4.21}
\int_{J} \lambda(t)^{-2} dt = T, \qquad \text{and} \qquad \eta_{1}^{-2} T = 2^{3k},
\end{equation}
then
\begin{equation}\label{4.22}
\sup_{t \in J} \| P_{\leq k + 9} (\frac{e^{-i \gamma(t)} e^{-ix \cdot \frac{\xi(t)}{\lambda(t)}}}{\lambda(t)^{1/2}} \epsilon(t, \frac{x - x(t)}{\lambda(t)})) \|_{\dot{H}^{1}}^{2} \lesssim \frac{2^{2k}}{T} \int_{J} \| \epsilon(t) \|_{L^{2}}^{2} \lambda(t)^{-2} dt + \sup_{t \in J} \frac{|\xi(t)|^{2}}{\lambda(t)^{2}} + 2^{2k} T^{-10},
\end{equation}
and
\begin{equation}\label{4.23}
\sup_{t \in J} \| \epsilon(t) \|_{L^{2}}^{2} \lesssim \frac{2^{2k} T^{1/50}}{\eta_{1}^{2} T} \int_{J} \| \epsilon(t) \|_{L^{2}}^{2} \lambda(t)^{-2} dt + \frac{T^{1/50}}{\eta_{1}^{2}} \sup_{t \in J} \frac{|\xi(t)|^{2}}{\lambda(t)^{2}} + 2^{2k} \frac{T^{1/50}}{\eta_{1}^{2}} T^{-10}.
\end{equation}
\end{theorem}
\begin{proof}
This theorem is a direct consequence of Theorem $\ref{t4.1}$ and an expansion of the energy. Recall from $(\ref{4.5})$ that
\begin{equation}\label{4.24}
E(P_{\leq k + 9} u) = \frac{1}{2} \frac{|\xi(t)|^{2}}{\lambda(t)^{2}} \| Q \|_{L^{2}}^{2} + \frac{1}{2 \lambda(t)^{2}} \| \epsilon \|_{L^{2}}^{2} - \frac{|\xi(t)|^{2}}{2 \lambda(t)^{2}} \| \epsilon \|_{L^{2}}^{2} + \mathcal P_{2} + \mathcal P_{3} + \mathcal P_{4} + O(2^{2k} T^{-10}),
\end{equation}
where $\mathcal P_{j}$, $j = 2, 3, 4$ refers to terms in the expansion of $E(P_{\leq k + 9} u)$ with $j$ $\epsilon$'s in the product. Split
\begin{equation}\label{4.25}
\aligned
P_{\leq k + 9} \tilde{\epsilon}(t,x) = P_{\leq k + 9} (e^{-i \gamma(t)} e^{-ix \cdot \frac{\xi(t)}{\lambda(t)}} \lambda(t)^{-1} \epsilon(t, \frac{x - x(t)}{\lambda(t)})) \\ = e^{-i \gamma(t)} e^{-ix \cdot \frac{\xi(t)}{\lambda(t)}} P_{\leq k + 9}(\lambda(t)^{-1} \epsilon(t, \frac{x - x(t)}{\lambda(t)})) + \mathcal R = \tilde{\tilde{\epsilon}}(t,x) + \mathcal R.
\endaligned
\end{equation}
Using the fact that $|\frac{|\xi(t)|}{\lambda(t)}| \leq \eta_{0}$ the Fourier support properties of $P_{\leq k + 9}$, and the discussion of Fourier multipliers prior to $(\ref{6.50.5})$,
\begin{equation}\label{4.25.1}
\| \mathcal R \|_{L^{2}} \lesssim 2^{-k} \frac{|\xi(t)|}{\lambda(t)} \| \epsilon \|_{L^{2}} \lesssim 2^{-k} \eta_{0} \| \epsilon \|_{L^{2}}, \qquad \text{and} \qquad \| \nabla \mathcal R \|_{L^{2}} \lesssim \frac{|\xi(t)|}{\lambda(t)} \| \epsilon \|_{L^{2}} \lesssim \eta_{0} \| \epsilon \|_{L^{2}}.
\end{equation}
Since $Q$ is real valued, smooth, and has derivatives that are rapidly decreasing, using $(\ref{4.25})$ and $(\ref{4.25.1})$,
\begin{equation}\label{4.26}
\aligned
\mathcal P_{2} = \frac{1}{2} \| \nabla \tilde{\tilde{\epsilon}} \|_{L^{2}}^{2} - \frac{1}{\lambda(t)^{4}} \int Q(\frac{x - x(t)}{\lambda(t)})^{2} |P_{\leq k + 9} \epsilon(t,\frac{x - x(t)}{\lambda(t)})|^{2} dx \\ - \frac{1}{2 \lambda(t)^{4}} Re \int Q(\frac{x - x(t)}{\lambda(t)})^{2} (P_{\leq k + 9} \epsilon (t,\frac{x - x(t)}{\lambda(t)}))^{2} dx + O(2^{2k} T^{-10}) + O(\eta_{0}^{2} \| \epsilon \|_{L^{2}}^{2}) + O(\eta_{0} \| \epsilon \|_{L^{2}} \| \nabla \tilde{\tilde{\epsilon}} \|_{L^{2}}).
\endaligned
\end{equation}
By the product rule,
\begin{equation}\label{4.27}
\aligned
\frac{1}{2} \| \nabla \tilde{\tilde{\epsilon}} \|_{L^{2}}^{2} = \frac{|\xi(t)|^{2}}{2 \lambda(t)^{4}} \| P_{\leq k + 9} \epsilon(t, \frac{x - x(t)}{\lambda(t)}) \|_{L^{2}}^{2} + \frac{\xi(t)}{\lambda(t)^{3}} \cdot (P_{\leq k + 9} \epsilon(t, \frac{x - x(t)}{\lambda(t)}), i P_{\leq k + 9} \nabla \epsilon(t, \frac{x - x(t)}{\lambda(t)}))_{L^{2}} \\ + \frac{1}{2\lambda(t)^{2}} \| P_{\leq k + 9} \epsilon(t, \frac{x - x(t)}{\lambda(t)}) \|_{\dot{H}^{1}}^{2}.
\endaligned
\end{equation}
Rescaling, if $2^{n(t)} = \lambda(t)$,
\begin{equation}\label{4.27.0}
\aligned
\frac{1}{2 \lambda(t)^{2}} \| P_{\leq k + 9} \epsilon(t, \frac{x - x(t)}{\lambda(t)}) \|_{\dot{H}^{1}}^{2} - \frac{1}{\lambda(t)^{4}} \int Q(\frac{x - x(t)}{\lambda(t)})^{2} |P_{\leq k + 9} \epsilon(t,\frac{x - x(t)}{\lambda(t)})|^{2} dx \\ - \frac{1}{2 \lambda(t)^{4}} Re \int Q(\frac{x - x(t)}{\lambda(t)})^{2} (P_{\leq k + 9} \epsilon (t,\frac{x - x(t)}{\lambda(t)}))^{2} dx \\
= \frac{1}{2 \lambda(t)^{2}} \| P_{\leq k + 9 + n(t)} \epsilon(t, x) \|_{\dot{H}^{1}}^{2} - \frac{1}{\lambda(t)^{2}} \int Q(x)^{2} |P_{\leq k + 9 + n(t)} \epsilon(t,x)|^{2} dx \\ - \frac{1}{2 \lambda(t)^{2}} Re \int Q(x)^{2} (P_{\leq k + 9 + n(t)} \epsilon (t,x)^{2} dx.
\endaligned
\end{equation}

Then using the spectral theory of $\mathcal L$ in Theorem $\ref{t3.1}$, provided
\begin{equation}\label{4.27.1}
 g \perp span \{ \chi_{0}, i \chi_{0}, Q_{x_{j}}, i Q_{x_{j}} \},
\end{equation}
there exists a fixed constant $\lambda_{1} > 0$ such that
\begin{equation}\label{4.28}
\frac{1}{2} \| \nabla g \|_{L^{2}}^{2} + \frac{1}{2} \| g \|_{L^{2}}^{2} - \int Q(x)^{2} |g(x)|^{2} - \frac{1}{2} Re \int Q(x)^{2} g(x)^{2} dx \geq \lambda_{1} \| g \|_{\dot{H}^{1}}^{2}.
 \end{equation}
It is not quite true for $P_{\leq k + 9 + n(t)} \epsilon$ that $(\ref{4.27.1})$ holds, however, by $(\ref{3.5})$, the fact that $\chi_{0}$ and $Q_{x_{j}}$ are smooth with rapidly decreasing derivatives, and the bounds on $\lambda(t)$ in $(\ref{4.20})$,
\begin{equation}\label{4.32}
(P_{\leq k + 9 + n(t)} \epsilon, f)_{L^{2}} \lesssim T^{-10},
\end{equation}
for any $f$ in $(\ref{3.7})$. Therefore, for $\eta_{0} \ll 1$ sufficiently small, there exists some fixed $\lambda_{1} > 0$ such that
\begin{equation}\label{4.32.1}
\frac{1}{2 \lambda(t)^{2}} \| \epsilon \|_{L^{2}}^{2} + \mathcal P_{2} \geq \frac{\lambda_{1}}{\lambda(t)^{2}} \| \epsilon \|_{L^{2}}^{2} + \lambda_{1} \| \tilde{\tilde{\epsilon}} \|_{\dot{H}^{1}}^{2} - O(2^{2k} T^{-10}).
\end{equation}

 Next, by the Sobolev embedding theorem and $(\ref{4.25.1})$,
\begin{equation}\label{4.29}
\int |P_{\leq k + 9} \tilde{\epsilon}(t,x)|^{4} dx \lesssim \| \tilde{\tilde{\epsilon}} \|_{L^{4}}^{4} + \| \mathcal R \|_{L^{4}}^{4} \lesssim \| \tilde{\tilde{\epsilon}} \|_{\dot{H}^{1}}^{2} \| \epsilon \|_{L^{2}}^{2} + \| \mathcal R \|_{L^{4}}^{4} \lesssim \eta_{0}^{2} \| \tilde{\tilde{\epsilon}} \|_{\dot{H}^{1}}^{2} + 2^{-2k} \eta_{0}^{2} \frac{|\xi(t)|^{2}}{\lambda(t)^{2}} \| \epsilon \|_{L^{2}}^{4},
\end{equation}
and
\begin{equation}\label{4.30}
\frac{1}{\lambda(t)} \int |Q(\frac{x - x(t)}{\lambda(t)}) |\tilde{\epsilon}(t,x)|^{3} dx \lesssim \| \tilde{\epsilon} \|_{L^{2}} \| \tilde{\epsilon} \|_{L^{4}}^{2} \lesssim \frac{1}{\lambda(t)} \| \tilde{\epsilon} \|_{L^{2}}^{2} \| \tilde{\epsilon} \|_{\dot{H}^{1}} \lesssim \frac{1}{\lambda(t)} \| \epsilon \|_{L^{2}} (\eta_{0} \| \tilde{\tilde{\epsilon}} \|_{\dot{H}^{1}} + 2^{-k} \eta_{0} \frac{|\xi(t)|}{\lambda(t)} \| \epsilon \|_{L^{2}}^{2}).
\end{equation}
For $\eta_{0} \ll 1$ sufficiently small, we have therefore proved
\begin{equation}\label{4.31}
\aligned
E(P_{\leq k + 9} u) &\geq \frac{1}{2} \frac{|\xi(t)|^{2}}{\lambda(t)^{2}} \| Q \|_{L^{2}}^{2} + \frac{\lambda_{1}}{2 \lambda(t)^{2}} \| \epsilon \|_{L^{2}}^{2} + \frac{\lambda_{1}}{2} \| \tilde{\tilde{\epsilon}} \|_{\dot{H}^{1}}^{2} - \frac{|\xi(t)|^{2}}{2 \lambda(t)^{2}} \| \epsilon \|_{L^{2}}^{2} - O(2^{2k} T^{-10}) \\
&\geq \frac{1}{4} \frac{|\xi(t)|^{2}}{\lambda(t)^{2}} \| Q \|_{L^{2}}^{2} + \frac{\lambda_{1}}{2 \lambda(t)^{2}} \| \epsilon \|_{L^{2}}^{2} + \frac{\lambda_{1}}{2} \| \tilde{\tilde{\epsilon}} \|_{\dot{H}^{1}}^{2} - O(2^{2k} T^{-10}).
\endaligned
\end{equation}
Plugging $(\ref{4.31})$ into $(\ref{4.2})$ proves Theorem $\ref{t4.2}$.
\end{proof}

\subsection{Almost conservation of energy in dimension $d \geq 3$}

\begin{theorem}\label{t4.3}
Let $J = [a, b]$ be an interval for which $(\ref{5.4})$ and $(\ref{5.5})$ hold, as well as $|x(t)| \leq T^{\frac{1}{2000d^{2}}}$. To simplify notation let $k_{0} = k(1 + \frac{1}{10d})$. Then,
\begin{equation}\label{4.32}
\sup_{t \in J} E(P_{\leq k_{0} + 9} u(t)) \lesssim \frac{2^{2k_{0}}}{T} \int_{J} \| \epsilon(t) \|_{L^{2}}^{2} \lambda(t)^{-2} dt + \sup_{t \in J} \frac{|\xi(t)|^{2}}{\lambda(t)^{2}} + 2^{2k_{0}} T^{-10}.
\end{equation}
\end{theorem}
\begin{proof}
Again choose $t_{0} \in J$ such that $(\ref{4.1})$ holds. As in the two dimensional case, by $(\ref{5.17.0})$, the fact that $Q$ is a real valued function, $Q$ is smooth and all its derivatives are rapidly decreasing, and the Sobolev embedding theorem, taking $\epsilon = \epsilon_{1} + i \epsilon_{2}$,
\begin{equation}\label{4.34}
\aligned
E(P_{\leq k_{0} + 9} u) = \frac{1}{2} \| \nabla P_{\leq k_{0} + 9} u \|_{L^{2}}^{2} - \frac{d}{2(d + 2)} \| P_{\leq k_{0} + 9} u \|_{L^{\frac{2(d + 2)}{d}}}^{\frac{2(d + 2)}{d}} \\ = \frac{1}{2 \lambda(t)^{2}} \| \nabla Q(x) \|_{L^{2}}^{2} + \frac{1}{2} \frac{|\xi(t)|^{2}}{\lambda(t)^{2}} \| Q \|_{L^{2}}^{2} - \frac{d}{2(d + 2) \lambda(t)^{2}} \| Q \|_{L^{\frac{2(d + 2)}{d}}}^{\frac{2(d + 2)}{d}} \\
+ \frac{1}{\lambda(t)^{2}} (\nabla Q, \nabla \epsilon_{1})_{L^{2}} - \frac{\xi(t)}{\lambda(t)^{2}} \cdot (Q, \nabla \epsilon_{2}) + \frac{\xi(t)}{\lambda(t)^{2}} \cdot (\nabla Q, \epsilon_{2})_{L^{2}} + \frac{|\xi(t)|^{2}}{\lambda(t)^{2}} (Q, \epsilon_{1})_{L^{2}} - \frac{1}{\lambda(t)^{2}} (Q^{1 + \frac{4}{d}}, \epsilon_{1})_{L^{2}} \\
+ O(2^{2k_{0}} \| \epsilon \|_{L^{2}}^{2} ) + O(2^{2k_{0}} T^{-10}).
\endaligned
\end{equation}
As in the two dimensional case, $(\epsilon_{2}, \nabla Q)_{L^{2}} = (\nabla \epsilon_{2}, Q)_{L^{2}} = 0$, $(\nabla Q, \nabla \epsilon_{1})_{L^{2}} - (Q^{1 + \frac{4}{d}}, \epsilon_{1})_{L^{2}} = -(Q, \epsilon_{1})_{L^{2}} = \frac{1}{2} \| \epsilon \|_{L^{2}}^{2}$, and $E(Q) = 0$, so
\begin{equation}\label{4.35}
\aligned
E(P_{\leq k_{0} + 9} u) = \frac{1}{2} \frac{|\xi(t)|^{2}}{\lambda(t)^{2}} \| Q \|_{L^{2}}^{2} + \frac{1}{2 \lambda(t)^{2}} \| \epsilon \|_{L^{2}}^{2} - \frac{|\xi(t)|^{2}}{2 \lambda(t)^{2}} \| \epsilon \|_{L^{2}}^{2} + O(2^{2k_{0}} \| \epsilon \|_{L^{2}}^{2} ) + O(2^{2k_{0}} T^{-10}),
\endaligned
\end{equation}
so $(\ref{4.32})$ holds at $t_{0}$.

Next compute the change of energy. Following the computations in the two dimensional case,
\begin{equation}\label{4.36}
\aligned
\frac{d}{dt} E(P_{\leq k_{0} + 9} u)  = ( i \Delta P_{\leq k_{0} + 9} u + i P_{\leq k_{0} + 9} F(u), P_{\leq k_{0} + 9} F(u) - F(P_{\leq k_{0} + 9} u) )_{L^{2}}.
\endaligned
\end{equation}
Decompose
\begin{equation}\label{4.38}
P_{\leq k_{0} + 9} F(u) - F(P_{\leq k_{0} + 9} u) = F(u) - F(P_{\leq k_{0} + 9} u) - P_{\geq k_{0} + 9} F(u).
\end{equation}
From the proof of Theorem $\ref{t5.1}$, the fact that $Q$ is smooth and all its derivatives are rapidly decreasing, the Sobolev embedding theorem, and Fourier support properties,
\begin{equation}\label{4.39}
\aligned
 \int_{t_{0}}^{t'} ( i \Delta P_{\leq k_{0} + 9} u + i P_{\leq k_{0} + 9} F(u), P_{\geq k_{0} + 9} F(u))_{L^{2}} dt \lesssim 2^{2k_{0}} \| P_{k_{0} \leq \cdot \leq k_{0} + 9} u \|_{L_{t}^{2} L_{x}^{\frac{2d}{d - 2}}} \| P_{\geq k_{0}} F(u) \|_{L_{t}^{2} L_{x}^{\frac{2d}{d + 2}}} \\ + 2^{2k_{0}} \| P_{k_{0} \leq \cdot \leq k_{0} + 9} u \|_{L_{t}^{2} L_{x}^{\frac{2d}{d + 2}}}^{2} \lesssim \frac{2^{2k_{0}}}{T} \int_{J} \| \epsilon(t) \|_{L^{2}}^{2} \lambda(t)^{-2} dt + 2^{2k_{0}} T^{-10}.
\endaligned
\end{equation}

Next, by Taylor's formula,
\begin{equation}\label{4.41}
F(u) - F(P_{\leq k_{0} + 9} u) = (P_{\geq k_{0} + 9} u) \cdot \int_{0}^{1} F'(P_{\leq k_{0} + 9} u + s P_{\geq k_{0} + 9} u) ds.
\end{equation}
By $(\ref{5.6.1})$,
\begin{equation}\label{4.37}
\Delta u + F(u) = \Delta \tilde{Q} + \Delta \tilde{\epsilon} + F(\tilde{Q}) + F(u) - F(\tilde{Q}).
\end{equation}
Then by Theorem $\ref{t5.1}$, $(\ref{5.4})$, and $(\ref{5.5})$,
\begin{equation}\label{4.42}
\aligned
\int_{t_{0}}^{t'} (i P_{\leq k_{0} + 9} \Delta \tilde{\epsilon}, \int_{0}^{1} F'(P_{\leq k_{0} + 9} u + s P_{\geq k_{0} + 9} u) ds \cdot P_{\geq k_{0} + 9} u)_{L^{2}} dt \lesssim 2^{2k_{0}} (\frac{1}{T} \int_{J} \| \epsilon(t) \|_{L^{2}}^{2} \lambda(t)^{-2} dt) + 2^{2k_{0}} T^{-10}.
\endaligned
\end{equation}
Also, by Theorem $\ref{t5.1}$, the Sobolev embedding theorem, and interpolation,
\begin{equation}\label{4.40}
\aligned
\int_{t_{0}}^{t'} (i P_{\leq k_{0} + 9} (F(u) - F(\tilde{Q})), \int_{0}^{1} F'(P_{\leq k_{0} + 9} u + s P_{\geq k_{0} + 9} u) ds \cdot P_{\geq k_{0} + 9} u)_{L^{2}} dt \\ \lesssim  \| \tilde{\epsilon} \|_{L_{t}^{2} L_{x}^{\frac{2d}{d - 2}}} \| P_{\geq k_{0} + 9} u \|_{L_{t}^{2} L_{x}^{\frac{2d}{d - 2}}} \| \tilde{Q} \|_{L_{t,x}^{\infty}}^{4/d} \| u \|_{L_{t}^{\infty} L_{x}^{2}}^{4/d} + \| \Delta P_{\leq k_{0} + 9} \tilde{\epsilon} \|_{L_{t}^{2} L_{x}^{\frac{2d}{d - 2}}} \| \tilde{\epsilon} \|_{L_{t}^{\infty} L_{x}^{2}}^{4/d} \| P_{\geq k_{0} + 9} u \|_{L_{t}^{2} L_{x}^{\frac{2d}{d - 2}}} \| u \|_{L_{t}^{\infty} L_{x}^{2}}^{4/d} \\ + 2^{2k_{0}} \| P_{\geq k_{0} + 9} u \|_{L_{t}^{2} L_{x}^{\frac{2d}{d - 2}}} \| P_{\geq k_{0} + 9} \tilde{\epsilon} \|_{L_{t}^{2} L_{x}^{\frac{2d}{d - 2}}} \| \tilde{\epsilon} \|_{L_{t}^{\infty} L_{x}^{2}}^{4/d} \| u \|_{L_{t}^{\infty} L_{x}^{2}}^{4/d} \lesssim 2^{2k_{0}} (\frac{1}{T} \int_{J} \| \epsilon(t) \|_{L^{2}}^{2} \lambda(t)^{-2} dt) + 2^{2k_{0}} T^{-10}.
\endaligned
\end{equation}
Now expand,
\begin{equation}\label{4.43}
\aligned
\Delta \tilde{Q} + F(\tilde{Q}) = e^{-i \gamma(t)} e^{-ix \cdot \frac{\xi(t)}{\lambda(t)}} \lambda(t)^{-d/2 - 2} \Delta Q(\frac{x - x(t)}{\lambda(t)}) - 2i e^{-i \gamma(t)} e^{-ix \cdot \frac{\xi(t)}{\lambda(t)}} \lambda(t)^{-d/2 - 2} \xi(t) \cdot \nabla Q(\frac{x - x(t)}{\lambda(t)}) \\ -e^{-i \gamma(t)} e^{-ix \cdot \frac{\xi(t)}{\lambda(t)}} \lambda(t)^{-d/2 - 2} |\xi(t)|^{2} Q(\frac{x - x(t)}{\lambda(t)}) + e^{-i \gamma(t)} e^{-ix \cdot \frac{\xi(t)}{\lambda(t)}} \lambda(t)^{-d/2 - 2} Q^{1 + 4/d}(\frac{x - x(t)}{\lambda(t)}).
\endaligned
\end{equation}
By $(\ref{1.12})$,
\begin{equation}\label{4.44}
\aligned
e^{-i \gamma(t)} e^{-ix \cdot \frac{\xi(t)}{\lambda(t)}} \lambda(t)^{-d/2 - 2} \Delta Q(\frac{x - x(t)}{\lambda(t)}) + e^{-i \gamma(t)} e^{-ix \cdot \frac{\xi(t)}{\lambda(t)}} \lambda(t)^{-d/2 - 2} Q^{1 + 4/d}(\frac{x - x(t)}{\lambda(t)}) = \lambda(t)^{-2} \tilde{Q}.
\endaligned
\end{equation}
Plugging $(\ref{4.44})$ into $(\ref{4.40})$, using $(\ref{5.4})$ and $(\ref{5.5})$,
\begin{equation}\label{4.47}
\aligned
\int_{t_{0}}^{t'} \lambda(t)^{-2} \int_{0}^{1} (i \lambda(t)^{-2} \tilde{Q}, F'(P_{\leq k_{0} + 9} u + s P_{\geq k_{0} + 9} u)(P_{\geq k_{0} +9} u))_{L^{2}} ds dt \\ 
= \int_{t_{0}}^{t'} \lambda(t)^{-2} \int_{0}^{1} (i (P_{\leq k_{0} + 9} u + s P_{\geq k_{0} + 9} u), F'(P_{\leq k_{0} + 9} u + s P_{\geq k_{0} + 9} u)(P_{\geq k_{0} + 9} u))_{L^{2}} ds dt \\ 
+ \int_{t_{0}}^{t'} \lambda(t)^{-2} \int_{0}^{1} (i (P_{\geq k_{0} + 9} \tilde{Q} - P_{\leq k_{0} + 9} \tilde{\epsilon} - sP_{\geq k_{0} + 9} \tilde{\epsilon}), F'(P_{\leq k_{0} + 9} u + s P_{\geq k_{0} + 9} u)(P_{\geq k_{0} + 9} u))_{L^{2}} ds dt  \\
\lesssim (1 + \frac{4}{d}) \int_{t_{0}}^{t'} \lambda(t)^{-2} \int_{0}^{1} (i F(P_{\leq k_{0} + 9} u + s P_{\geq k_{0} + 9} u), (P_{\geq k_{0} + 9} u))_{L^{2}} ds dt \\ + \eta_{1}^{2} \| \tilde{\epsilon} \|_{L_{t}^{2} L_{x}^{\frac{2d}{d - 2}}} \| P_{\geq k_{0} + 9} u \|_{L_{t}^{2} L_{x}^{\frac{2d}{d - 2}}} \| u \|_{L_{t}^{\infty} L_{x}^{2}}^{4/d} +  \eta_{1}^{2} \| P_{\geq k_{0} + 9} \tilde{Q} \|_{L_{t}^{2} L_{x}^{\frac{2d}{d - 2}}} \| P_{\geq k_{0} + 9} u \|_{L_{t}^{2} L_{x}^{\frac{2d}{d - 2}}} \| u \|_{L_{t}^{\infty} L_{x}^{2}}^{4/d} \\
\lesssim (1 + \frac{4}{d}) \int_{t_{0}}^{t'} \lambda(t)^{-2} (i F(P_{\leq k_{0}} u), P_{\geq k_{0} + 9} u)_{L^{2}} dt + \eta_{1}^{2} \| P_{\geq k_{0}} u \|_{L_{t}^{2} L_{x}^{\frac{2d}{d - 2}}}^{2} \| u \|_{L_{t}^{\infty} L_{x}^{2}}^{4/d} \\ + \eta_{1}^{2} \| \tilde{\epsilon} \|_{L_{t}^{2} L_{x}^{\frac{2d}{d - 2}}} \| P_{\geq k_{0} + 9} u \|_{L_{t}^{2} L_{x}^{\frac{2d}{d - 2}}} \| u \|_{L_{t}^{\infty} L_{x}^{2}}^{4/d} +  \eta_{1}^{2} \| P_{\geq k_{0} + 9} \tilde{Q} \|_{L_{t}^{2} L_{x}^{\frac{2d}{d - 2}}} \| P_{\geq k_{0} + 9} u \|_{L_{t}^{2} L_{x}^{\frac{2d}{d - 2}}} \| u \|_{L_{t}^{\infty} L_{x}^{2}}^{4/d}.
\endaligned
\end{equation}
Then by Theorem $\ref{t5.1}$,
\begin{equation}\label{4.45}
\aligned
(1 + \frac{4}{d}) \int_{t_{0}}^{t'} \lambda(t)^{-2} (i F(P_{\leq k_{0}} u), P_{\geq k_{0} + 9} u)_{L^{2}} dt + \eta_{1}^{2} \| P_{\geq k_{0}} u \|_{L_{t}^{2} L_{x}^{\frac{2d}{d - 2}}}^{2} \| u \|_{L_{t}^{\infty} L_{x}^{2}}^{4/d} \\ + \eta_{1}^{2} \| \tilde{\epsilon} \|_{L_{t}^{2} L_{x}^{\frac{2d}{d - 2}}} \| P_{\geq k_{0} + 9} u \|_{L_{t}^{2} L_{x}^{\frac{2d}{d - 2}}} \| u \|_{L_{t}^{\infty} L_{x}^{2}}^{4/d} +  \eta_{1}^{2} \| P_{\geq k_{0} + 9} \tilde{Q} \|_{L_{t}^{2} L_{x}^{\frac{2d}{d - 2}}} \| P_{\geq k_{0} + 9} u \|_{L_{t}^{2} L_{x}^{\frac{2d}{d - 2}}} \| u \|_{L_{t}^{\infty} L_{x}^{2}}^{4/d} \\
\lesssim \eta_{1}^{2} \| P_{\geq k_{0}} F(u_{\leq k_{0}}) \|_{L_{t}^{2} L_{x}^{\frac{2d}{d + 2}}} \| P_{\geq k_{0} + 9} u \|_{L_{t}^{2} L_{x}^{\frac{2d}{d - 2}}} + \eta_{1}^{2} \| P_{\geq k_{0}} u \|_{L_{t}^{2} L_{x}^{\frac{2d}{d - 2}}}^{2} \| u \|_{L_{t}^{\infty} L_{x}^{2}}^{4/d} \\ + \eta_{1}^{2} \| \tilde{\epsilon} \|_{L_{t}^{2} L_{x}^{\frac{2d}{d - 2}}} \| P_{\geq k_{0} + 9} u \|_{L_{t}^{2} L_{x}^{\frac{2d}{d - 2}}} \| u \|_{L_{t}^{\infty} L_{x}^{2}}^{4/d} +  \eta_{1}^{2} \| P_{\geq k_{0} + 9} \tilde{Q} \|_{L_{t}^{2} L_{x}^{\frac{2d}{d - 2}}} \| P_{\geq k_{0} + 9} u \|_{L_{t}^{2} L_{x}^{\frac{2d}{d - 2}}} \| u \|_{L_{t}^{\infty} L_{x}^{2}}^{4/d} \\
\lesssim 2^{2k_{0}} (\frac{1}{T} \int_{J} \| \epsilon(t) \|_{L^{2}}^{2} \lambda(t)^{-2} dt) + 2^{2k_{0}} T^{-10}.
\endaligned
\end{equation}
By the product rule,
\begin{equation}\label{4.46}
\aligned
- 2i e^{-i \gamma(t)} e^{-ix \cdot \frac{\xi(t)}{\lambda(t)}} \lambda(t)^{-d/2 - 2} \xi(t) \cdot \nabla Q(\frac{x - x(t)}{\lambda(t)}) -e^{-i \gamma(t)} e^{-ix \cdot \frac{\xi(t)}{\lambda(t)}} \lambda(t)^{-d/2 - 2} |\xi(t)|^{2} Q(\frac{x - x(t)}{\lambda(t)}) \\ = -2i \frac{\xi(t)}{\lambda(t)^{2}} \cdot \nabla \tilde{Q} + \frac{|\xi(t)|^{2}}{\lambda(t)^{2}} \tilde{Q}.
\endaligned
\end{equation}
Plugging $\frac{|\xi(t)|^{2}}{\lambda(t)^{2}} \tilde{Q}$ into $(\ref{4.47})$ and using $(\ref{5.4})$ and $(\ref{5.5})$, the contribution of $\frac{|\xi(t)|^{2}}{\lambda(t)^{2}} \tilde{Q}$ can be handled in the same way as $\lambda(t)^{-2} \tilde{Q}$.

Finally, use the computations in the proof of Theorem $\ref{t5.1}$ to compute
\begin{equation}
\aligned
2 \int_{t_{0}}^{t'} (P_{\leq k_{0} + 9} \nabla \tilde{Q}, \int_{0}^{1} F'(P_{\leq k_{0} + 9} u + s P_{\geq k_{0} + 9} u) ds \cdot P_{\geq k_{0} + 9} u)_{L^{2}} dt \\ \lesssim 2^{2k_{0}} (\frac{1}{T} \int_{J} \| \epsilon(t) \|_{L^{2}}^{2} \lambda(t)^{-2} dt) + 2^{2k_{0}} T^{-10}.
\endaligned
\end{equation}

\end{proof}

The bound on $E(P_{\leq k_{0} + 9} u)$ gives good bounds on the $L^{2}$ and $\dot{H}^{1}$ norms of $\epsilon$ in higher dimensions as well.
\begin{theorem}\label{t4.4}
If $|x(t)| \leq T^{\frac{1}{2000d^{2}}}$ for all $t \in J$, and $(\ref{5.4})$ and $(\ref{5.5})$ hold, then
\begin{equation}\label{4.53}
\sup_{t \in J} \| P_{\leq k_{0} + 9} (\frac{e^{-i \gamma(t)} e^{-ix \cdot \frac{\xi(t)}{\lambda(t)}}}{\lambda(t)^{1/2}} \epsilon(t, \frac{x - x(t)}{\lambda(t)})) \|_{\dot{H}^{1}}^{2} \lesssim \frac{2^{2k_{0}}}{T} \int_{J} \| \epsilon(t) \|_{L^{2}}^{2} \lambda(t)^{-2} dt + \sup_{t \in J} \frac{|\xi(t)|^{2}}{\lambda(t)^{2}} + 2^{2k_{0}} T^{-10},
\end{equation}
and
\begin{equation}\label{4.54}
\sup_{t \in J} \| \epsilon(t) \|_{L^{2}}^{2} \lesssim \frac{2^{2k_{0}} T^{1/25d}}{\eta_{1}^{2} T} \int_{J} \| \epsilon(t) \|_{L^{2}}^{2} \lambda(t)^{-2} dt + \frac{T^{1/25d}}{\eta_{1}^{2}} \sup_{t \in J} \frac{|\xi(t)|^{2}}{\lambda(t)^{2}} + 2^{2k_{0}} \frac{T^{1/25d}}{\eta_{1}^{2}} T^{-10}.
\end{equation}
\end{theorem}
\begin{proof}
This theorem is a direct consequence of Theorem $\ref{t4.1}$ and an expansion of the energy. Recall from $(\ref{4.35})$ that
\begin{equation}\label{4.55}
\aligned
E(P_{\leq k_{0} + 9} u) = \frac{1}{2} \frac{|\xi(t)|^{2}}{\lambda(t)^{2}} \| Q \|_{L^{2}}^{2} + \frac{1}{2 \lambda(t)^{2}} \| \epsilon \|_{L^{2}}^{2} - \frac{|\xi(t)|^{2}}{2 \lambda(t)^{2}} \| \epsilon \|_{L^{2}}^{2} \\ + \mathcal P_{2} + O(|P_{\leq k_{0} + 9} \tilde{\epsilon}|^{3} |P_{\leq k_{0} + 9} \tilde{Q}|) + O(|P_{\leq k_{0} + 9} \tilde{\epsilon}|^{2 + 4/d}) + O(2^{2k_{0}} T^{-10}).
\endaligned
\end{equation}
Again let
\begin{equation}\label{4.56}
P_{\leq k_{0} + 9} (e^{-i \gamma(t)} e^{-ix \cdot \frac{\xi(t)}{\lambda(t)}} \lambda(t)^{-1} \epsilon(t, \frac{x - x(t)}{\lambda(t)})) = e^{-i \gamma(t)} e^{-ix \cdot \frac{\xi(t)}{\lambda(t)}} P_{\leq k_{0} + 9}(\lambda(t)^{-1} \epsilon(t, \frac{x - x(t)}{\lambda(t)})) + \mathcal R = \tilde{\tilde{\epsilon}}(t,x) + \mathcal R,
\end{equation}
where
\begin{equation}\label{4.57}
\| \mathcal R \|_{L^{2}} \lesssim 2^{-k} \frac{|\xi(t)|}{\lambda(t)} \| \epsilon \|_{L^{2}} \lesssim 2^{-k} \eta_{0} \| \epsilon \|_{L^{2}}^{2}, \qquad \text{and} \qquad \| \nabla \mathcal R \|_{L^{2}} \lesssim \frac{|\xi(t)|}{\lambda(t)} \| \epsilon \|_{L^{2}} \lesssim \eta_{0} \| \epsilon \|_{L^{2}}.
\end{equation}
As in dimension $d = 2$, since $Q$ is real valued, smooth, and has derivatives that are rapidly decreasing,
\begin{equation}\label{4.58}
\aligned
\mathcal P_{2} = \frac{1}{2} \| \nabla \tilde{\epsilon} \|_{L^{2}}^{2} - \frac{d + 2}{2d \lambda(t)^{d + 2}} \int Q(\frac{x - x(t)}{\lambda(t)})^{4/d} |P_{\leq k_{0} + 9} \epsilon(t,\frac{x - x(t)}{\lambda(t)})|^{2} dx \\ - \frac{1}{d \lambda(t)^{d + 2}} Re \int Q(\frac{x - x(t)}{\lambda(t)})^{4/d} (P_{\leq k_{0} + 9} \epsilon (t,\frac{x - x(t)}{\lambda(t)}))^{2} dx + O(2^{2k_{0}} T^{-10}) \\ + O(\eta_{0}^{2} \| \epsilon \|_{L^{2}}^{2}) + O(\eta_{0} \| \epsilon \|_{L^{2}} \| \nabla \tilde{\epsilon} \|_{L^{2}}).
\endaligned
\end{equation}
Also, as in dimension $d = 2$, by the product rule,
\begin{equation}\label{4.59}
\aligned
\frac{1}{2} \| \nabla \tilde{\epsilon} \|_{L^{2}}^{2} = \frac{|\xi(t)|^{2}}{2 \lambda(t)^{d + 2}} \| P_{\leq k_{0} + 9} \epsilon(t, \frac{x - x(t)}{\lambda(t)}) \|_{L^{2}}^{2} \\ + \frac{\xi(t)}{\lambda(t)^{d + 1}} \cdot (P_{\leq k_{0} + 9} \epsilon(t, \frac{x - x(t)}{\lambda(t)}), i P_{\leq k_{0} + 9} \nabla \epsilon(t, \frac{x - x(t)}{\lambda(t)}))_{L^{2}} + \frac{1}{2\lambda(t)^{d}} \| P_{\leq k_{0} + 9} \epsilon(t, \frac{x - x(t)}{\lambda(t)}) \|_{\dot{H}^{1}}^{2}.
\endaligned
\end{equation}
Rescaling, if $2^{n(t)} = \lambda(t)$,
\begin{equation}\label{4.60}
\aligned
\frac{1}{2 \lambda(t)^{d}} \| P_{\leq k_{0} + 9} \epsilon(t, \frac{x - x(t)}{\lambda(t)}) \|_{\dot{H}^{1}}^{2} - \frac{d + 2}{2d\lambda(t)^{d + 2}} \int Q(\frac{x - x(t)}{\lambda(t)})^{4/d} |P_{\leq k_{0} + 9} \epsilon(t,\frac{x - x(t)}{\lambda(t)})|^{2} dx \\ - \frac{1}{d \lambda(t)^{d + 2}} Re \int Q(\frac{x - x(t)}{\lambda(t)})^{4/d} (P_{\leq k_{0} + 9} \epsilon (t,\frac{x - x(t)}{\lambda(t)}))^{2} dx \\
= \frac{1}{2 \lambda(t)^{2}} \| P_{\leq k_{0} + 9 + n(t)} \epsilon(t, x) \|_{\dot{H}^{1}}^{2} - \frac{d + 2}{2d\lambda(t)^{2}} \int Q(x)^{4/d} |P_{\leq k_{0} + 9 + n(t)} \epsilon(t,x)|^{2} dx \\ - \frac{1}{d \lambda(t)^{2}} Re \int Q(x)^{4/d} (P_{\leq k_{0} + 9 + n(t)} \epsilon (t,x)^{2} dx.
\endaligned
\end{equation}

Then using the spectral theory of $\mathcal L$ in Theorem $\ref{t3.1}$, provided
\begin{equation}\label{4.61}
 g \perp span \{ \chi_{0}, i \chi_{0}, Q_{x_{j}}, i Q_{x_{j}} \},
\end{equation}
there exists a fixed constant $\lambda_{d} > 0$ such that
\begin{equation}\label{4.62}
\frac{1}{2} \| \nabla g \|_{L^{2}}^{2} + \frac{1}{2} \| g \|_{L^{2}}^{2} - \frac{d + 2}{2d} \int Q(x)^{2} |g(x)|^{2} - \frac{1}{d} Re \int Q(x)^{2} g(x)^{2} dx \geq \lambda_{d} \| g \|_{H^{1}}^{2}.
 \end{equation}
Again using the fact that $\chi_{0}$ and $Q_{x_{j}}$ are smooth with rapidly decreasing derivatives, and the bounds on $\lambda(t)$ in $(\ref{4.20})$,
\begin{equation}\label{4.63}
(P_{\leq k_{0} + 9 + n(t)} \epsilon, f)_{L^{2}} \lesssim T^{-10}.
\end{equation}
Therefore, for $\eta_{0} \ll 1$ sufficiently small, there exists some fixed $\lambda_{1} > 0$ such that
\begin{equation}\label{4.64}
\frac{1}{2 \lambda(t)^{2}} \| \epsilon \|_{L^{2}}^{2} + \mathcal P_{2} \geq \frac{\lambda_{d}}{\lambda(t)^{2}} \| \epsilon \|_{L^{2}}^{2} + \lambda_{d} \| \tilde{\epsilon} \|_{\dot{H}^{1}}^{2} - O(2^{2k_{0}} T^{-10}).
\end{equation}

 Next, by the Sobolev embedding theorem and $(\ref{4.57})$,
\begin{equation}\label{4.65}
\int |\tilde{\epsilon}(t,x)|^{2 + \frac{4}{d}} dx \lesssim \| \tilde{\tilde{\epsilon}} \|_{\dot{H}^{1}}^{2} \| \epsilon \|_{L^{2}}^{\frac{4}{d}} + \frac{|\xi(t)|^{2}}{\lambda(t)^{2}} \| \epsilon \|_{L^{2}}^{2 + 8/d} \lesssim \eta_{0}^{\frac{4}{d}} \| \tilde{\tilde{\epsilon}} \|_{\dot{H}^{1}}^{2} + \eta_{0}^{8/d} \frac{|\xi(t)|^{2}}{\lambda(t)^{2}} \| \epsilon \|_{L^{2}}^{2}.
\end{equation}
Therefore, by interpolation, for $\eta_{0} \ll 1$ sufficiently small,
\begin{equation}\label{4.66}
\aligned
E(P_{\leq k_{0} + 9} u) &\geq \frac{1}{2} \frac{|\xi(t)|^{2}}{\lambda(t)^{2}} \| Q \|_{L^{2}}^{2} + \frac{\lambda_{d}}{2 \lambda(t)^{2}} \| \epsilon \|_{L^{2}}^{2} + \frac{\lambda_{d}}{2} \| \tilde{\tilde{\epsilon}} \|_{\dot{H}^{1}}^{2} - \frac{|\xi(t)|^{2}}{2 \lambda(t)^{2}} \| \epsilon \|_{L^{2}}^{2} - O(2^{2k_{0}} T^{-10}) \\
&\geq \frac{1}{4} \frac{|\xi(t)|^{2}}{\lambda(t)^{2}} \| Q \|_{L^{2}}^{2} + \frac{\lambda_{d}}{2 \lambda(t)^{2}} \| \epsilon \|_{L^{2}}^{2} + \frac{\lambda_{d}}{2} \| \tilde{\tilde{\epsilon}} \|_{\dot{H}^{1}}^{2} - O(2^{2k_{0}} T^{-10})
\endaligned
\end{equation}
Plugging $(\ref{4.66})$ into $(\ref{4.32})$ proves Theorem $\ref{t4.4}$.
\end{proof}

\section{A frequency localized Morawetz estimate}
Proceeding to the frequency localized Morawetz estimates, again start with dimension $d = 2$.
\subsection{Two dimensions}
\begin{theorem}\label{t10.1}
Let $J = [a, b]$ be an interval on which
\begin{equation}\label{10.1}
\frac{|\xi(t)|}{\lambda(t)} \leq \eta_{0}, \qquad \frac{1}{\eta_{1}} \leq \lambda(t) \leq \frac{1}{\eta_{1}} T^{1/100}, \qquad \text{for all} \qquad t \in J, \qquad \int_{J} \lambda(t)^{-2} dt = T, \qquad \eta_{1}^{-2} T = 2^{3k}.
\end{equation}
Also suppose that $\epsilon = \epsilon_{1} + i \epsilon_{2}$ and suppose $\xi(a) = x(b) = 0$. Finally suppose there exists a uniform bound on $x(t)$,
\begin{equation}\label{10.1.1}
\sup_{t \in J} |x(t)| \leq R = T^{1/25}.
\end{equation}
Then for $T$ sufficiently large,
\begin{equation}\label{10.1.2}
\int_{a}^{b} \| \epsilon(t) \|_{L^{2}}^{2} \lambda(t)^{-2} dt \leq 3 (\epsilon_{2}(a), Q + x \cdot \nabla Q)_{L^{2}} - 3(\epsilon_{2}(b), Q + x \cdot \nabla Q)_{L^{2}} + \frac{T^{1/15}}{\eta_{1}^{2}} \sup_{t \in J} \frac{|\xi(t)|^{2}}{\lambda(t)^{2}} + O(T^{-8}).
\end{equation}
\end{theorem}
\begin{proof}
Define a Morawetz potential. Let $\chi(r) \in C^{\infty}([0, \infty))$ be a smooth, radial function, satisfying $\chi(r) = 1$ for $0 \leq r \leq 1$, and supported on $r \leq 2$. Then let
\begin{equation}\label{10.2}
\phi(r) = \int_{0}^{r} \chi^{2}(\frac{s}{2R}) ds = \int_{0}^{r} \psi(\frac{s}{2R}) ds,
\end{equation}
and let
\begin{equation}\label{10.3}
M(t) = \int \phi(r) Im[\overline{P_{\leq k + 9} u} \partial_{r} P_{\leq k + 9} u](t,x) dx = \int \phi(|x|) \frac{x}{|x|} \cdot Im[\overline{P_{\leq k + 9} u} \nabla P_{\leq k + 9} u](t,x) dx.
\end{equation}
Following $(\ref{6.34})$,
\begin{equation}\label{10.4}
i \partial_{t} P_{\leq k + 9} u + \Delta P_{\leq k + 9} u + F(P_{\leq k + 9} u) = F(P_{\leq k + 9} u) - P_{\leq k + 9} F(u) = -\mathcal N.
\end{equation}

Plugging in $(\ref{10.4})$ and integrating by parts,
\begin{equation}\label{10.5}
\aligned
\frac{d}{dt} M(t) = \int \phi(r) Re[-\Delta \overline{P_{\leq k + 9} u} \partial_{r} P_{\leq k + 9} u + \overline{P_{\leq k + 9} u} \Delta \partial_{r} P_{\leq k + 9} u] \\ + \int \phi(r) Re[-F(\overline{P_{\leq k + 9} u}) \partial_{r} P_{\leq k + 9} u + \overline{P_{\leq k + 9} u} \partial_{r} F(P_{\leq k + 9} u)] \\ + \int \phi(r) Re[\overline{P_{\leq k + 9} u} \partial_{r} \mathcal N](t,x) dx - \int \phi(r) Re[\bar{\mathcal N} \partial_{r} P_{\leq k + 9} u](t,x) dx \\
= 2 \int \chi^{2}(\frac{x}{R}) |\nabla P_{\leq k + 9} u|^{2} dx  - \frac{1}{2R^{2}} \int \psi''(\frac{x}{R}) |P_{\leq k + 9} u|^{2} dx - \int \chi^{2}(\frac{x}{R}) |P_{\leq k + 9} u|^{4} dx \\ + 2 \int [\frac{1}{|x|} \phi(x) - \chi^{2}(\frac{x}{R})](\delta_{jk} - \frac{x_{j} x_{k}}{|x|^{2}}) Re(\overline{\partial_{j} P_{\leq k + 9} u} \partial_{k} P_{\leq k + 9} u) dx \\
+ \int \phi(r) Re[\overline{P_{\leq k + 9} u} \partial_{r} \mathcal N](t,x) dx - \int \phi(r) Re[\bar{\mathcal N} \partial_{r} P_{\leq k + 9} u](t,x) dx.
\endaligned
\end{equation}

Next, following $(\ref{6.41})$,
\begin{equation}\label{10.6}
\aligned
\mathcal N = P_{\leq k + 9}(2 |u_{\leq k + 6}|^{2} u_{\geq k + 6} + (u_{\leq k + 6})^{2} \overline{u_{\geq k + 6}}) - (2 |u_{\leq k + 6}|^{2}  u_{k + 6 \leq \cdot \leq k + 9} + (u_{\leq k + 6})^{2} \overline{u_{k + 6 \leq \cdot \leq k + 9}}) \\
+ P_{\leq k + 9} O((u_{\geq k + 6})^{2} u) + O((u_{k + 6 \leq \cdot \leq k + 9})^{2} u) = \mathcal N^{(1)} + \mathcal N^{(2)}.
\endaligned
\end{equation}
As in $(\ref{6.42})$, by Theorems $\ref{t6.2}$ and $\ref{t6.3}$,
\begin{equation}\label{10.7}
\aligned
\int_{a}^{b} \int \phi(r) Re[\overline{P_{\leq k + 9} u} \partial_{r} \mathcal N^{(2)}] dx dt - \int_{0}^{T} \int \phi(r) Re[\overline{\mathcal N^{(2)}} \partial_{r} P_{\leq k + 9} u] dx dt \\
\lesssim 2^{k} R \| (u_{\geq k + 6})(u_{\leq k + 3}) \|_{L_{t,x}^{2}}^{2} + 2^{k} R \| u_{\geq k + 6} \|_{L_{t,x}^{4}}^{2} \| u_{\geq k + 3} \|_{L_{t,x}^{4}}^{2} \\
\lesssim 2^{k} R (\frac{1}{T} \int_{a}^{b} \| \epsilon(t) \|_{L^{2}}^{2} \lambda(t)^{-2} dt + 2^{2k} T^{-10}).
\endaligned
\end{equation}
Splitting $\mathcal N^{(1)}$,
\begin{equation}\label{10.8}
\mathcal N^{(1)} = O((u_{\geq k + 6})(u_{\geq k + 3}) u) + O((u_{\geq k + 6})(u_{\leq k + 3})^{2}) = \mathcal N^{(1,1)} + \mathcal N^{(1,2)}.
\end{equation}
Making calculations identical to the estimate $(\ref{10.7})$,
\begin{equation}\label{10.9}
\aligned
\int_{a}^{b} \int \phi(r) Re[\overline{P_{\leq k + 9} u} \partial_{r} \mathcal N^{(1,1)}] dx dt - \int_{a}^{b} \int \phi(r) Re[\overline{\mathcal N^{(1,1)}} \partial_{r} P_{\leq k + 9} u] dx dt \\
+ \int_{0}^{T} \int \phi(x) Re[\overline{P_{k + 3 \leq \cdot \leq k + 9} u} \partial_{r} \mathcal N^{(1,1)}] dx dt - \int_{a}^{b} \int \phi(x) Re[\overline{\mathcal N^{(1,1)}} \partial_{r} P_{\geq k + 3 \leq \cdot \leq k + 9} u] dx dt \\
\lesssim 2^{k} R \| (u_{\geq k + 6})(u_{\leq k}) \|_{L_{t,x}^{2}} \| (u_{\geq k + 3})(u_{\leq k}) \|_{L_{t,x}^{2}} + 2^{k} R \| u_{\geq k + 3} \|_{L_{t,x}^{4}}^{2} \| u_{\geq k} \|_{L_{t,x}^{4}}^{2} \\
\lesssim 2^{k} R (\frac{1}{T} \int_{a}^{b} \| \epsilon(t) \|_{L^{2}}^{2} \lambda(t)^{-2} dt  + 2^{2k} T^{-10}).
\endaligned
\end{equation}
Finally, using Bernstein's inequality, the fact that $\phi$ is smooth, rapidly decreasing, $R = T^{1/25}$, and $\| u \|_{L_{t,x}^{4}([a, b] \times \mathbb{R}^{2})} \lesssim T^{1/4}$,
\begin{equation}\label{10.10}
\aligned
 \int_{a}^{b} \int \phi(r) Re[\overline{u_{\leq k + 3}} \partial_{r} \mathcal N^{(1,2)}] dx dt - \int_{a}^{b} \int \phi(r) Re[\overline{\mathcal N^{(1,2)}} \partial_{r} u_{\leq k + 3}] dx dt \\
 \lesssim \| P_{\geq k + 3} \phi(x) \|_{L^{\infty}} \| (u_{\geq k + 6} u)(u_{\leq k + 3}) \|_{L_{t,x}^{2}} \| u_{\leq k + 3} \|_{L_{t,x}^{4}}^{2} \lesssim \frac{1}{T^{9}}.
 \endaligned
\end{equation}
Therefore, the error arising from frequency truncation is bounded by
\begin{equation}
2^{k} R (\frac{1}{T} \int_{a}^{b} \| \epsilon(t) \|_{L^{2}}^{2} \lambda(t)^{-2} dt + 2^{2k} T^{-10}).
\end{equation}

Using the fact that $(\frac{1}{r} \phi(r) - \chi^{2}(r)) (\delta_{jk} - \frac{x_{j} x_{k}}{|x|^{2}})$ is a positive definite matrix, by the fundamental theorem of calculus,
\begin{equation}\label{10.11}
\aligned
2 \int_{a}^{b} \int \chi^{2}(\frac{x}{R}) |\nabla P_{\leq k + 9} u|^{2} dx dt - \frac{1}{2R^{2}} \int_{a}^{b} \int \psi''(\frac{x}{R}) |P_{\leq k + 9} u|^{2} dx dt - \int_{a}^{b} \int \chi^{2}(\frac{x}{R}) |P_{\leq k + 9} u|^{4} dx dt \\ \lesssim M(b) - M(a) + 2^{k} R(\frac{1}{T} \int_{J} \| \epsilon(t) \|_{L^{2}}^{2} \lambda(t)^{-2} dt)) + O(T^{-8}).
\endaligned
\end{equation}

Next compute $M(b) - M(a)$ under the assumption that $\xi(a) = x(b) = 0$. Since $Q$ is smooth, real valued, rapidly decreasing, all its derivatives are rapidly decreasing, by $(\ref{10.1})$, $(\ref{10.1.1})$, $(\ref{10.2})$, and $\xi(b) = x(a) = 0$,
\begin{equation}\label{10.17}
\int \phi(|x|) \frac{x}{|x|} \cdot Im[\overline{P_{\leq k + 9} (e^{-i \gamma(t)} e^{-ix \cdot \frac{\xi(t)}{\lambda(t)}} \frac{1}{\lambda(t)} Q(\frac{x - x(t)}{\lambda(t)}))} \nabla (P_{\leq k + 9} (e^{-i \gamma(t)} e^{-ix \cdot \frac{\xi(t)}{\lambda(t)}} \frac{1}{\lambda(t)} Q(\frac{x - x(t)}{\lambda(t)})))] dx|_{a}^{b}
\end{equation}
\begin{equation}\label{10.16}
\aligned
= \frac{1}{\lambda(t)^{2}} \int \phi(|x|) \frac{x}{|x|} \cdot \frac{\xi(t)}{\lambda(t)} |Q(\frac{x - x(t)}{\lambda(t)})|^{2} dx + O(2^{k} R T^{-10}) \\ = \frac{x(t)}{\lambda(t)^{2}} \cdot \frac{\xi(t)}{\lambda(t)} \int Q(\frac{x}{\lambda(t)})^{2} dx|_{a}^{b} + O(2^{k} R T^{-10}) = O(2^{k} R T^{-10}).
\endaligned
\end{equation}

Next, by Theorem $\ref{t4.2}$, $(\ref{10.1})$, $(\ref{10.1.1})$, $(\ref{10.2})$, $\xi(b) = x(a) = 0$, $(\epsilon_{1}, \nabla Q)_{L^{2}} = (\epsilon_{2}, \nabla Q)_{L^{2}}$, the fact that $Q$ is smooth and rapidly decreasing, and that $\phi(|x|) \frac{x}{|x|}$ is smooth,
\begin{equation}\label{10.18}
\aligned
\int \phi(|x|) \frac{x}{|x|} \cdot Im[\overline{P_{\leq k + 9} (e^{-i \gamma(t)} e^{-ix \cdot \frac{\xi(t)}{\lambda(t)}} \frac{1}{\lambda(t)} \epsilon(\frac{x - x(t)}{\lambda(t)}))} \nabla (P_{\leq k + 9} (e^{-i \gamma(t)} e^{-ix \cdot \frac{\xi(t)}{\lambda(t)}} \frac{1}{\lambda(t)} Q(\frac{x - x(t)}{\lambda(t)})))] dx|_{a}^{b} \\
= \xi(t) \cdot (\epsilon_{1}, x Q)_{L^{2}} + \frac{x(t) \cdot \xi(t)}{\lambda(t)} (\epsilon_{1}, Q)_{L^{2}} - (\epsilon_{2}, x \cdot \nabla Q)_{L^{2}} - x(t)(\epsilon_{2}, \nabla Q)_{L^{2}} + O(T^{-10})|_{a}^{b}, \\
= \xi(t) \cdot (\epsilon_{1}, x Q)_{L^{2}} + (\epsilon_{2}, x \cdot \nabla Q) + O(T^{-10})|_{a}^{b} = O(\frac{|\xi(t)|}{\lambda(t)} \lambda(t) \| \epsilon \|_{L^{2}}) - (\epsilon_{2}, x \cdot \nabla Q)_{L^{2}}|_{a}^{b} + O(T^{-10}) \\
= \frac{T^{1/50}}{\eta_{1}^{2}} O(\frac{|\xi(t)|^{2}}{\lambda(t)^{2}}) + \| \epsilon \|_{L^{2}}^{2} - (\epsilon_{2}, x \cdot \nabla Q)_{L^{2}}|_{a}^{b} + O(T^{-10}) \\ = -(\epsilon_{2}, x \cdot \nabla Q)_{L^{2}}|_{a}^{b} +  O(\frac{2^{2k} T^{1/50}}{\eta_{1}^{2} T} \int_{J} \| \epsilon(t) \|_{L^{2}}^{2} \lambda(t)^{-2} dt + \frac{T^{1/50}}{\eta_{1}^{2}} \sup_{t \in J} \frac{|\xi(t)|^{2}}{\lambda(t)^{2}} + 2^{2k} \frac{T^{1/50}}{\eta_{1}^{2}} T^{-10}).
\endaligned
\end{equation}
Also, integrating by parts, since $Q$ and all its derivatives are rapidly decreasing, as well as $\phi(|x|) \frac{x}{|x|}$ is smooth,
\begin{equation}\label{10.19}
\aligned
\int \phi(|x|) \frac{x}{|x|} \cdot Im[\overline{P_{\leq k + 9} (e^{-i \gamma(t)} e^{-ix \cdot \frac{\xi(t)}{\lambda(t)}} \frac{1}{\lambda(t)} Q(\frac{x - x(t)}{\lambda(t)}))} \nabla (P_{\leq k + 9} (e^{-i \gamma(t)} e^{-ix \cdot \frac{\xi(t)}{\lambda(t)}} \frac{1}{\lambda(t)} \epsilon(\frac{x - x(t)}{\lambda(t)})))] dx|_{a}^{b}  \\ = (\ref{10.18}) - \int \nabla \cdot (\phi(|x|) \frac{x}{|x|}) \cdot Im[\overline{ \frac{1}{\lambda(t)} Q(\frac{x - x(t)}{\lambda(t)}))} ( \frac{1}{\lambda(t)} \epsilon(\frac{x - x(t)}{\lambda(t)})))] dx|_{a}^{b} + O(T^{-10})  \\ = (\ref{10.18}) - 2(\epsilon_{2}, Q)_{L^{2}}|_{a}^{b} + O(T^{-10}).
\endaligned
\end{equation}
Finally, by Theorem $\ref{t4.2}$, for any $t \in J$,
\begin{equation}\label{10.20}
\aligned
\int \phi(|x|) \frac{x}{|x|} \cdot Im[\overline{P_{\leq k + 9} (e^{-i \gamma(t)} e^{-ix \cdot \frac{\xi(t)}{\lambda(t)}} \frac{1}{\lambda(t)} \epsilon(\frac{x - x(t)}{\lambda(t)}))} \nabla (P_{\leq k + 9} (e^{-i \gamma(t)} e^{-ix \cdot \frac{\xi(t)}{\lambda(t)}} \frac{1}{\lambda(t)} \epsilon(\frac{x - x(t)}{\lambda(t)})))] dx \\
\lesssim R \| \epsilon \|_{L^{2}} \| P_{\leq k + 9} (e^{-i \gamma(t)} e^{-ix \cdot \frac{\xi(t)}{\lambda(t)}} \frac{1}{\lambda(t)} \epsilon(\frac{x - x(t)}{\lambda(t)})) \|_{\dot{H}^{1}} + R \frac{|\xi(t)|}{\lambda(t)} \| \epsilon \|_{L^{2}}^{2} \\
\lesssim R (\frac{2^{2k} T^{1/50}}{\eta_{1}^{2} T} \int_{J} \| \epsilon(t) \|_{L^{2}}^{2} \lambda(t)^{-2} dt + \frac{T^{1/50}}{\eta_{1}^{2}} \sup_{t \in J} \frac{|\xi(t)|^{2}}{\lambda(t)^{2}} + 2^{2k} \frac{T^{1/50}}{\eta_{1}^{2}} T^{-10}).
\endaligned
\end{equation}
Therefore, 
\begin{equation}\label{10.21}
\aligned
M(b) - M(a) = 2 (\epsilon_{2}(a), Q + x \cdot \nabla Q)_{L^{2}} - 2(\epsilon_{2}(b), Q + x \cdot \nabla Q)_{L^{2}} \\ + O(\frac{2^{2k} T^{1/15}}{\eta_{1}^{2} T} \int_{J} \| \epsilon(t) \|_{L^{2}}^{2} \lambda(t)^{-2} dt + \frac{T^{1/15}}{\eta_{1}^{2}} \sup_{t \in J} \frac{|\xi(t)|^{2}}{\lambda(t)^{2}} + 2^{2k} \frac{T^{1/15}}{\eta_{1}^{2}} T^{-10})).
\endaligned
\end{equation}

Therefore, to complete the proof of Theorem $\ref{t10.1}$, it only remains to obtain a lower bound for
\begin{equation}\label{10.12}
\aligned
2 \int_{a}^{b} \int \chi^{2}(\frac{x}{R}) |\nabla P_{\leq k + 9} u|^{2} dx dt - \frac{1}{2R^{2}} \int_{a}^{b} \int \psi''(\frac{x}{R}) |P_{\leq k + 9} u|^{2} dx dt - \int_{a}^{b} \int \chi^{2}(\frac{x}{R}) |P_{\leq k + 9} u|^{4} dx dt.
\endaligned
\end{equation}

Splitting
\begin{equation}\label{10.14}
|P_{\leq k + 9} u|^{2} \leq 2 \lambda(t)^{-2} |P_{\leq k + 9}(e^{-i \gamma(t)} e^{ix \cdot \frac{\xi(t)}{\lambda(t)}}  Q(\frac{x - x(t)}{\lambda(t)})|^{2} + 2 \lambda(t)^{-2} |P_{\leq k + 9}(e^{-i \gamma(t)} e^{ix \cdot \frac{\xi(t)}{\lambda(t)}}  \epsilon(t, \frac{x - x(t)}{\lambda(t)})|^{2}.
\end{equation}
By $(\ref{10.1})$, $(\ref{10.1.1})$, and $(\ref{10.2})$, the support of $\psi''(x)$, and the fact that $Q$ is smooth and all its derivatives are rapidly decreasing,
\begin{equation}\label{10.15.1}
\frac{1}{R^{2}} \lambda(t)^{-2} \int \psi''(\frac{x}{R}) |P_{\leq k + 9}(e^{-i \gamma(t)} e^{ix \cdot \frac{\xi(t)}{\lambda(t)}} Q(\frac{x - x(t)}{\lambda(t)})|^{2} dx \lesssim \lambda(t)^{-2} \frac{1}{T^{10}}.
\end{equation}
On the other hand, by $(\ref{10.1})$, $(\ref{10.1.1})$, and $(\ref{10.2})$, for $T$ sufficiently large,
\begin{equation}\label{10.16}
\frac{1}{R^{2}} \lambda(t)^{-2} \int \psi''(\frac{x}{R}) |P_{\leq k + 9}(e^{-i \gamma(t)} e^{ix \cdot \frac{\xi(t)}{\lambda(t)}} \epsilon(t, \frac{x - x(t)}{\lambda(t)})|^{2} dx \lesssim  \frac{1}{R} \lambda(t)^{-2} \| \epsilon \|_{L^{2}}^{2}.
\end{equation}

Since $Q$ is smooth and all its derivatives are rapidly decreasing, by $(\ref{10.1})$, $(\ref{10.1.1})$, and $(\ref{10.2})$,
\begin{equation}\label{10.17}
\aligned
\frac{1}{2} \int (1 - \chi^{2}(\frac{x}{2R})) |\nabla P_{\leq k + 9} (e^{-i \gamma(t)} e^{-ix \cdot \xi(t)} \frac{1}{\lambda(t)} Q(\frac{x - x(t)}{\lambda(t)})|^{2} dx \\ - \frac{1}{4} \int (1 - \chi^{2}(\frac{x}{2R})) |P_{\leq k + 9} (e^{-i \gamma(t)} e^{-ix \cdot \xi(t)} \frac{1}{\lambda(t)} Q(\frac{x - x(t)}{\lambda(t)})|^{4} dx = O(2^{2k} T^{-10}).
\endaligned
\end{equation}
Also,
\begin{equation}\label{10.19}
\aligned
\int (1 - \chi^{2}(\frac{x}{2R})) Re(\overline{\nabla P_{\leq k + 9} (e^{-i \gamma(t)} e^{-ix \cdot \xi(t)} \frac{1}{\lambda(t)} Q(\frac{x - x(t)}{\lambda(t)})} \cdot \nabla P_{\leq k + 9} (e^{-i \gamma(t)} e^{-ix \cdot \xi(t)} \frac{1}{\lambda(t)} \epsilon(t, \frac{x - x(t)}{\lambda(t)})) dx \\ -  \int (1 - \chi^{2}(\frac{x}{2R})) |P_{\leq k + 9} (e^{-i \gamma(t)} e^{-ix \cdot \xi(t)} \frac{1}{\lambda(t)} Q(\frac{x - x(t)}{\lambda(t)})|^{2} \\ \times Re(\overline{ P_{\leq k + 9} (e^{-i \gamma(t)} e^{-ix \cdot \xi(t)} \frac{1}{\lambda(t)} Q(\frac{x - x(t)}{\lambda(t)})} \cdot  P_{\leq k + 9} (e^{-i \gamma(t)} e^{-ix \cdot \xi(t)} \frac{1}{\lambda(t)} \epsilon(t, \frac{x - x(t)}{\lambda(t)})) dx = O(2^{2k} T^{-10}).
\endaligned
\end{equation}
Therefore, from $(\ref{4.24})$,
\begin{equation}\label{10.20}
\aligned
\frac{1}{2} \int \chi^{2}(\frac{x}{2R}) |\nabla P_{\leq k + 9} (e^{-i \gamma(t)} e^{-ix \cdot \xi(t)} \frac{1}{\lambda(t)} Q(\frac{x - x(t)}{\lambda(t)})|^{2} dx \\ - \frac{1}{4} \int  \chi^{2}(\frac{x}{2R}) |P_{\leq k + 9} (e^{-i \gamma(t)} e^{-ix \cdot \xi(t)} \frac{1}{\lambda(t)} Q(\frac{x - x(t)}{\lambda(t)})|^{4} dx \\
+ \int  \chi^{2}(\frac{x}{2R}) Re(\overline{\nabla P_{\leq k + 9} (e^{-i \gamma(t)} e^{-ix \cdot \xi(t)} \frac{1}{\lambda(t)} Q(\frac{x - x(t)}{\lambda(t)})} \cdot \nabla P_{\leq k + 9} (e^{-i \gamma(t)} e^{-ix \cdot \xi(t)} \frac{1}{\lambda(t)} \epsilon(t, \frac{x - x(t)}{\lambda(t)})) dx \\ -  \int \chi^{2}(\frac{x}{2R}) |P_{\leq k + 9} (e^{-i \gamma(t)} e^{-ix \cdot \xi(t)} \frac{1}{\lambda(t)} Q(\frac{x - x(t)}{\lambda(t)})|^{2} \\ \times Re(\overline{ P_{\leq k + 9} (e^{-i \gamma(t)} e^{-ix \cdot \xi(t)} \frac{1}{\lambda(t)} Q(\frac{x - x(t)}{\lambda(t)})} \cdot  P_{\leq k + 9} (e^{-i \gamma(t)} e^{-ix \cdot \xi(t)} \frac{1}{\lambda(t)} \epsilon(t, \frac{x - x(t)}{\lambda(t)})) dx \\
= \frac{1}{2} \frac{|\xi(t)|^{2}}{\lambda(t)^{2}} \| Q \|_{L^{2}}^{2} + \frac{1}{2 \lambda(t)^{2}} \| \epsilon \|_{L^{2}}^{2} - \frac{|\xi(t)|^{2}}{2 \lambda(t)^{2}} \| \epsilon \|_{L^{2}}^{2} + O(2^{2k} T^{-10}).
\endaligned
\end{equation}

Turning to the terms with two $\epsilon$'s, by the product rule,
\begin{equation}\label{10.22}
\aligned
\frac{1}{2} \int \chi^{2}(\frac{x}{2R}) |\nabla P_{\leq k + 9} (e^{-i \gamma(t)} e^{-ix \cdot \xi(t)} \frac{1}{\lambda(t)} \epsilon(t, \frac{x - x(t)}{\lambda(t)})|^{2} dx \\ = \frac{1}{2} \| \chi(\frac{x}{2R}) P_{\leq k + 9} (e^{-i \gamma(t)} e^{-ix \cdot \xi(t)} \frac{1}{\lambda(t)} \epsilon(t, \frac{x - x(t)}{\lambda(t)})) \|_{\dot{H}^{1}}^{2} + O(\frac{1}{R \lambda(t)^{2}} \| \epsilon \|_{L^{2}}^{2}).
\endaligned
\end{equation}
Then by $(\ref{10.1})$, $(\ref{10.1.1})$, $(\ref{10.2})$, $(\ref{4.32})$, and the fact that $Q_{x_{j}}$ and $\chi_{0}$ are rapidly decreasing,
\begin{equation}\label{10.23}
(\chi(\frac{x \lambda(t) + x(t)}{2R}) \epsilon, f)_{L^{2}} \lesssim T^{-10}.
\end{equation}
Therefore, following the analysis in $(\ref{4.25})$--$(\ref{4.32.1})$,
\begin{equation}\label{10.24}
\aligned
\frac{1}{2} \int \chi^{2}(\frac{x}{2R}) |\nabla P_{\leq k + 9} (e^{-i \gamma(t)} e^{-ix \cdot \xi(t)} \frac{1}{\lambda(t)} \epsilon(t, \frac{x - x(t)}{\lambda(t)})|^{2} dx \\ - \frac{1}{\lambda(t)^{4}} \int |P_{\leq k + 9}(e^{-i \gamma(t)} e^{-ix \cdot \frac{\xi(t)}{\lambda(t)}}Q(\frac{x - x(t)}{\lambda(t)}))|^{2} |P_{\leq k + 9} (e^{-i \gamma(t)} e^{-ix \cdot \xi(t)} \epsilon(t,\frac{x - x(t)}{\lambda(t)}))|^{2} dx \\ 
- \frac{1}{2 \lambda(t)^{4}} Re \int (\overline{P_{\leq k + 9}(e^{-i \gamma(t)} e^{-ix \cdot \frac{\xi(t)}{\lambda(t)}}Q(\frac{x - x(t)}{\lambda(t)}))})^{2} (P_{\leq k + 9} (e^{-i \gamma(t)} e^{-ix \cdot \xi(t)} \epsilon(t,\frac{x - x(t)}{\lambda(t)})))^{2} dx \\
\geq \frac{\lambda_{1}}{\lambda(t)^{2}} \| \epsilon \|_{L^{2}}^{2} + \lambda_{1} \int \chi(\frac{x}{2R})^{2} |\nabla P_{\leq k + 9} (e^{-i \gamma(t)} e^{-ix \cdot \xi(t)} \frac{1}{\lambda(t)} \epsilon(t, \frac{x - x(t)}{\lambda(t)})|^{2} dx  - O(2^{2k} T^{-10}).
\endaligned
\end{equation}

Following $(\ref{1.8})$,
\begin{equation}\label{10.25}
\aligned
\int \chi^{2}(\frac{x}{2R}) |g(x_{1}, x_{2})|^{4} dx_{1} dx_{2} \leq \int \chi(\frac{x}{2R}) |g(x_{1}, x_{2})|^{2} \cdot (\sup_{x_{1} \in \mathbb{R}} \chi(\frac{x}{2R}) |g(x_{1}, x_{2})|^{2}) dx_{1} dx_{2} \\
\leq \int \chi(\frac{x}{2R}) |g(x_{1}, x_{2})|^{2} \cdot (\int |\partial_{x_{1}}(\chi(\frac{x}{2R}) |g(s, x_{2})|^{2})| ds) dx_{1} dx_{2} \\ = \int (\int \chi(\frac{x}{2R}) |g(x_{1}, x_{2})|^{2} dx_{1}) \cdot (\int |\partial_{x_{1}}(\chi(\frac{x}{2R}) |g(x_{1}, x_{2})|^{2})| dx_{1}) dx_{2} \\
\leq \sup_{x_{2} \in \mathbb{R}} (\int \chi(\frac{x}{2R}) |g(x_{1}, x_{2})|^{2} dx_{1}) \cdot \int (\int |\partial_{x_{1}}(\chi(\frac{x}{2R}) |g(x_{1}, x_{2})|^{2})| dx_{1}) dx_{2} \\
\leq \int \int |\partial_{x_{2}}(\chi(\frac{x}{2R}) |g(x_{1}, x_{2})|^{2})| dx_{1} dx_{2} \cdot \int \int |\partial_{x_{1}}(\chi(\frac{x}{2R}) |g(x_{1}, x_{2})|^{2})| dx_{1} dx_{2} \\
\lesssim \| \chi(\frac{x}{2R}) \nabla g \|_{L^{2}}^{2} \| g \|_{L^{2}}^{2} + \frac{1}{R^{2}} \| g \|_{L^{2}}^{4}.
\endaligned
\end{equation}
Therefore,
\begin{equation}\label{10.26}
\aligned
\int \chi(\frac{x}{2R})^{2} |P_{\leq k + 9}(e^{-i \gamma(t)} e^{-ix \cdot \xi(t)} \lambda(t)^{-1} \epsilon(t, \frac{x - x(t)}{\lambda(t)}))|^{4} dx \\\lesssim \eta_{0}^{2} \| \chi(\frac{x}{2R}) \nabla P_{\leq k + 9}(e^{-i \gamma(t)} e^{-ix \cdot \xi(t)} \lambda(t)^{-1} \epsilon(t, \frac{x - x(t)}{\lambda(t)}) \|_{L^{2}}^{2} + \frac{\eta_{0}^{2}}{R^{2}} \| \epsilon \|_{L^{2}}^{2}.
\endaligned
\end{equation}
Finally, by interpolation,
\begin{equation}\label{10.27}
\aligned
\int \chi(\frac{x}{2R})^{2} |P_{\leq k + 9}(e^{-i \gamma(t)} e^{-ix \cdot \xi(t)} \lambda(t)^{-1} \epsilon(t, \frac{x - x(t)}{\lambda(t)}))|^{3} |P_{\leq k + 9} (e^{-i \gamma(t)} e^{-ix \cdot \frac{\xi(t)}{\lambda(t)}} \frac{1}{\lambda(t)} Q(\frac{x - x(t)}{\lambda(t)}))| dx \\\lesssim \eta_{0} \| \chi(\frac{x}{2R}) \nabla P_{\leq k + 9}(e^{-i \gamma(t)} e^{-ix \cdot \xi(t)} \lambda(t)^{-1} \epsilon(t, \frac{x - x(t)}{\lambda(t)}) \|_{L^{2}} \frac{\| \epsilon \|_{L^{2}}}{\lambda(t)} + \frac{\eta_{0}}{R \lambda(t)} \| \epsilon \|_{L^{2}}^{2}.
\endaligned
\end{equation}

Therefore, we have finally proved that there exists some $\lambda_{1} > 0$ such that for $\eta_{0}$ sufficiently small, 
\begin{equation}\label{10.28}
(\ref{10.12}) \gtrsim \lambda_{1} \int_{a}^{b} \| \epsilon \|_{L^{2}}^{2} \lambda(t)^{-2} dt - O(T^{-8}).
\end{equation}
Combining $(\ref{10.28})$, $(\ref{10.11})$, $(\ref{10.15.1})$, $(\ref{10.16})$, and $(\ref{10.21})$ proves Theorem $\ref{t10.1}$.
\end{proof}

\subsection{Dimensions $d \geq 3$}
The computations in dimensions $d \geq 3$ are quite similar.

\begin{theorem}\label{t10.2}
Let $J = [a, b]$ be an interval on which $(\ref{5.4})$ and $(\ref{5.5})$ hold, and $|x(t)| \leq T^{\frac{1}{2000d^{2}}}$.  Also suppose that $\epsilon = \epsilon_{1} + i \epsilon_{2}$, and that $\xi(a) = 0$ and $x(b) = 0$. Then for $T$ sufficiently large,
\begin{equation}\label{10.29}
\int_{a}^{b} \| \epsilon(t) \|_{L^{2}}^{2} \lambda(t)^{-2} dt \leq 3 (\epsilon_{2}(a), Q + x \cdot \nabla Q)_{L^{2}} - 3(\epsilon_{2}(b), Q + x \cdot \nabla Q)_{L^{2}} + \frac{T^{1/25d}}{\eta_{1}^{2}} \sup_{t \in J} \frac{|\xi(t)|^{2}}{\lambda(t)^{2}} + O(T^{-8}).
\end{equation}
\end{theorem}
\begin{proof}
Once again let $R = T^{1/25d}$, $\chi(r) \in C^{\infty}([0, \infty))$ be a smooth, radial function, satisfying $\chi(r) = 1$ for $0 \leq r \leq 1$, and supported on $r \leq 2$, and let
\begin{equation}\label{10.30}
\phi(r) = \int_{0}^{r} \chi^{2}(\frac{s}{2R}) ds = \int_{0}^{r} \psi(\frac{s}{2R}) ds,
\end{equation}
and
\begin{equation}\label{10.31}
M(t) = \int \phi(r) Im[\overline{P_{\leq k_{0} + 9} u} \partial_{r} P_{\leq k_{0} + 9} u](t,x) dx = \int \phi(|x|) \frac{x}{|x|} \cdot Im[\overline{P_{\leq k_{0} + 9} u} \nabla P_{\leq k_{0} + 9} u](t,x) dx.
\end{equation}
Once again,
\begin{equation}\label{10.32}
i \partial_{t} P_{\leq k_{0} + 9} u + \Delta P_{\leq k_{0} + 9} u + F(P_{\leq k_{0} + 9} u) = F(P_{\leq k_{0} + 9} u) - P_{\leq k_{0} + 9} F(u) = -\mathcal N.
\end{equation}

Plugging in $(\ref{10.32})$ and integrating by parts,
\begin{equation}\label{10.33}
\aligned
\frac{d}{dt} M(t) = \int \phi(r) Re[-\Delta \overline{P_{\leq k_{0} + 9} u} \partial_{r} P_{\leq k_{0} + 9} u + \overline{P_{\leq k_{0} + 9} u} \Delta \partial_{r} P_{\leq k_{0} + 9} u] \\ + \int \phi(r) Re[-F(\overline{P_{\leq k_{0} + 9} u}) \partial_{r} P_{\leq k_{0} + 9} u + \overline{P_{\leq k_{0} + 9} u} \partial_{r} F(P_{\leq k_{0} + 9} u)] \\ + \int \phi(r) Re[\overline{P_{\leq k_{0} + 9} u} \partial_{r} \mathcal N](t,x) - \int \phi(r) Re[\bar{\mathcal N} \partial_{r} P_{\leq k_{0} + 9} u](t,x) \\
= 2 \int \chi^{2}(\frac{x}{2R}) |\nabla P_{\leq k_{0} + 9} u|^{2} dx  - \frac{1}{2R^{2}} \int \psi''(\frac{x}{2R}) |P_{\leq k_{0} + 9} u|^{2} dx - \frac{2d}{d + 2} \int \chi^{2}(\frac{x}{2R}) |P_{\leq k_{0} + 9} u|^{4} dx \\ + 2 \int [\frac{1}{|x|} \phi(x) - \chi^{2}(\frac{x}{2R})](\delta_{jk} - \frac{x_{j} x_{k}}{|x|^{2}}) Re(\overline{\partial_{j} P_{\leq k_{0} + 9} u} \partial_{k} P_{\leq k_{0} + 9} u) dx 
\\ + \int \phi(r) Re[\overline{P_{\leq k_{0} + 9} u} \partial_{r} \mathcal N](t,x) - \int \phi(r) Re[\bar{\mathcal N} \partial_{r} P_{\leq k_{0} + 9} u](t,x).
\endaligned
\end{equation}

Decompose
\begin{equation}\label{10.34}
\mathcal N = P_{\leq k_{0} + 9} F(u) - F(P_{\leq k_{0} + 9} u) = F(u) - F(P_{\leq k_{0} + 9} u) - P_{\geq k_{0} + 9} F(u).
\end{equation}
Then by Theorem $\ref{t5.1}$, Bernstein's inequality, Fourier support properties of $P_{\geq k_{0} + 9}$ and $P_{\leq k_{0}}$, and the fact that $\phi$ is smooth and all its derivatives are rapidly decreasing,
\begin{equation}\label{10.35}
\aligned
\int_{J} \int \phi(r) Re[\overline{P_{\geq k_{0} + 9} F(u)} \partial_{r} P_{\leq k_{0} + 9} u](t,x) dx dt \lesssim 2^{k_{0}} R \| P_{\geq k_{0}} F(u) \|_{L_{t}^{2} L_{x}^{\frac{2d}{d + 2}}} \| P_{\geq k_{0}} u \|_{L_{t}^{2} L_{x}^{\frac{2d}{d - 2}}} \\
+ \int_{J} \int \phi(r) Re[\overline{P_{\geq k_{0} + 9} F(u)} \partial_{r} P_{\leq k_{0}} u](t,x) dx dt,
\endaligned
\end{equation}
and
\begin{equation}\label{10.36}
\aligned
 \int_{J} \int \phi(r) Re[\overline{P_{\geq k_{0} + 9} F(u)} \partial_{r} P_{\leq k_{0}} u](t,x) dx dt \lesssim 2^{k_{0}} R \| P_{\geq k_{0}} F(u) \|_{L_{t}^{2} L_{x}^{\frac{2d}{d + 2}}} T^{-5} \\
 \lesssim 2^{k_{0}} R (\frac{1}{T} \int_{J} \| \epsilon(t) \|_{L^{2}}^{2} \lambda(t)^{-2} dt) + T^{-9}.
 \endaligned
\end{equation}
Next, by Taylor's formula,
\begin{equation}\label{10.37}
\aligned
\int \phi(r) Im[\overline{F(u) - F(P_{\leq k_{0} + 9} u)} \partial_{r} (P_{\leq k_{0} + 9} u)] dx dt \\ = \int \int_{0}^{1}  \phi(r) Im[\overline{F'(P_{\leq k_{0} + 9} u + s P_{\geq k_{0} + 9} u) (P_{\geq k_{0} + 9} u)} \partial_{r} (P_{\leq k_{0} + 9} u)] ds dx dt.
\endaligned
\end{equation}
By the chain rule and the computations in $(\ref{10.35})$--$(\ref{10.36})$,
\begin{equation}\label{10.38}
\aligned
\int  \phi(r) Im[\overline{F'(P_{\leq k_{0} + 9} u) (P_{\geq k_{0} + 9} u)} \partial_{r} (P_{\leq k_{0} + 9} u)] dx dt = \int \phi(r) Im[\overline{P_{\geq k_{0} + 9} u}] \partial_{r} F(P_{\leq k_{0} + 9} u) dx dt \\ \lesssim 2^{k_{0}} R (\frac{1}{T} \int_{J} \| \epsilon(t) \|_{L^{2}}^{2} \lambda(t)^{-2} dt) + T^{-9}.
\endaligned
\end{equation}
Next, by Theorem $\ref{t5.1}$, if $\frac{1}{p} = \frac{1}{2} \frac{1}{1 + 4/d}$, and $\frac{1}{q} = \frac{d + 2}{2d} \frac{1}{1 + 4/d}$,
\begin{equation}\label{10.40}
\aligned
\int \int_{0}^{1}  \phi(r) Im[\overline{[F'(P_{\leq k_{0} + 9} u + s P_{\geq k_{0} + 9} u) - F'(P_{\leq k_{0} + 9} u)] (P_{\geq k_{0} + 9} u)} \partial_{r} (P_{\leq k_{0} + 9} \tilde{\epsilon})] ds dx dt \\
\lesssim R \| P_{\geq k_{0} + 9} u \|_{L_{t}^{2} L_{x}^{\frac{2d}{d - 2}}} \| P_{\geq k_{0} + 9} u \|_{L_{t}^{p} L_{x}^{q}}^{4/d} \| \nabla P_{\leq k_{0} + 9} \tilde{\epsilon} \|_{L_{t}^{p} L_{x}^{q}} \lesssim 2^{k_{0}} R (\frac{1}{T} \int_{J} \| \epsilon(t) \|_{L^{2}}^{2} \lambda(t)^{-2} dt) + T^{-9}.
\endaligned
\end{equation}
Then using the fact that $Q$ is smooth with all derivatives rapidly decreasing,
\begin{equation}\label{10.39}
\aligned
\int \int_{0}^{1}  \phi(r) Im[\overline{[F'(P_{\leq k_{0} + 9} u + s P_{\geq k_{0} + 9} u) - F'(P_{\leq k_{0} + 9} \tilde{Q} + P_{\leq k_{0} + 9} \tilde{\epsilon})] (P_{\geq k_{0} + 9} u)} \partial_{r} (P_{\leq k_{0} + 9} \tilde{Q})] ds dx dt \\
\lesssim \int \int_{0}^{1}  \phi(r) Im[\overline{[F'(\tilde{Q} + P_{\leq k_{0} + 9} \tilde{\epsilon} + s P_{\geq k_{0} + 9} \tilde{\epsilon}) - F'(\tilde{Q} )] (P_{\geq k_{0} + 9} u)} \partial_{r} (\tilde{Q})] ds dx dt \\ +  \int \int_{0}^{1}  \phi(r) Im[\overline{[F'(\tilde{Q} + P_{\leq k_{0} + 9} \tilde{\epsilon}) - F'(\tilde{Q} )] (P_{\geq k_{0} + 9} u)} \partial_{r} (\tilde{Q})] ds dx dt+ T^{-9}.
\endaligned
\end{equation}
Using the computations in Theorem $\ref{t5.1}$,
\begin{equation}\label{10.40}
\aligned
\int \int_{0}^{1}  \phi(r) Im[\overline{[F'(\tilde{Q} + P_{\leq k_{0} + 9} \tilde{\epsilon} + s P_{\geq k_{0} + 9} \tilde{\epsilon}) - F'(\tilde{Q})] (P_{\geq k_{0} + 9} u)} \partial_{r} (\tilde{Q})] ds dx dt \\ \lesssim 2^{k_{0}} R (\frac{1}{T} \int_{J} \| \epsilon(t) \|_{L^{2}}^{2} \lambda(t)^{-2} dt) + T^{-9},
\endaligned
\end{equation}
and
\begin{equation}\label{10.40.1}
\aligned
\int \int_{0}^{1}  \phi(r) Im[\overline{[F'(\tilde{Q} + P_{\leq k_{0} + 9} \tilde{\epsilon}) - F'(\tilde{Q})] (P_{\geq k_{0} + 9} u)} \partial_{r} (\tilde{Q})] ds dx dt \\ \lesssim 2^{k_{0}} R (\frac{1}{T} \int_{J} \| \epsilon(t) \|_{L^{2}}^{2} \lambda(t)^{-2} dt) + T^{-9},
\endaligned
\end{equation}

Next, integrating by parts,
\begin{equation}\label{10.41}
\aligned
\int \int \phi(x) \frac{x}{|x|} \cdot Im[(\overline{P_{\leq k_{0} + 9} u}) \nabla \mathcal N] dx dt = -\int \int \phi(x) \frac{x}{|x|} \cdot Im[(
\nabla \overline{P_{\leq k_{0} + 9} u}) \mathcal N] dx dt \\ - \int \int \nabla \cdot (\phi(x) \frac{x}{|x|}) Im[(
\overline{P_{\leq k_{0} + 9} u}) \mathcal N] dx dt.
\endaligned
\end{equation}
The first term in $(\ref{10.41})$ is handled by the computations in $(\ref{10.35})$--$(\ref{10.40})$. Again using the smoothness of $\phi(x) \frac{x}{|x|}$, Theorem $\ref{t5.1}$, and Fourier support arguments,
\begin{equation}\label{10.42}
\int \int \nabla \cdot (\phi(x) \frac{x}{|x|}) Im[(\overline{P_{\leq k_{0} + 9} u})P_{\geq k_{0} + 9} F(u)] dx dt \lesssim \frac{1}{T} \int_{J} \| \epsilon(t) \|_{L^{2}}^{2} \lambda(t)^{-2} dt + T^{-9}.
\end{equation}
Also, since $\nabla \cdot (\phi(x) \frac{x}{|x|}) \lesssim \inf\{ 1, \frac{R}{|x|} \}$, using the analysis in $(\ref{10.37})$--$(\ref{10.40})$,
\begin{equation}\label{10.43}
\aligned
\int \int \nabla \cdot (\phi(x) \frac{x}{|x|}) Im[(\overline{P_{\leq k_{0} + 9} u})(F(u) - F(P_{\leq k_{0} + 9} u))] dx dt \\ = \int \int \int_{0}^{1} \nabla \cdot (\phi(x) \frac{x}{|x|}) Im[(\overline{P_{\leq k_{0} + 9} u})F'(P_{\leq k_{0} + 9} u + s P_{\geq k_{0} + 9} u)(P_{\geq k_{0} + 9} u)] ds dx dt \\
\lesssim (1 + \frac{4}{d}) \int \int \int_{0}^{1} \nabla \cdot (\phi(x) \frac{x}{|x|}) Im[\overline{F(P_{\leq k_{0} + 9} u + s P_{\geq k_{0} + 9} u)}(P_{\geq k_{0} + 9} u)] ds dx dt + \| P_{\geq k_{0} + 9} u \|_{L_{t}^{2} L_{x}^{\frac{2d}{d - 2}}}^{2}  \\
\lesssim (1 + \frac{4}{d}) \int \int \nabla \cdot (\phi(x) \frac{x}{|x|}) Im[\overline{F(P_{\leq k_{0}} u)}(P_{\geq k_{0} + 9} u)] dx dt + \| P_{\geq k_{0} + 9} u \|_{L_{t}^{2} L_{x}^{\frac{2d}{d - 2}}}^{2}. 
\endaligned
\end{equation}
By Theorem $\ref{t5.1}$,
\begin{equation}
\| P_{\geq k_{0} + 9} u \|_{L_{t}^{2} L_{x}^{\frac{2d}{d - 2}}}^{2} + \| P_{\geq k_{0}} F(P_{\leq k_{0}} u) \|_{L_{t}^{2} L_{x}^{\frac{2d}{d + 2}}} \| P_{\geq k_{0} + 9} u \|_{L_{t}^{2} L_{x}^{\frac{2d}{d - 2}}} \lesssim \frac{1}{T} \int_{J} \| \epsilon(t) \|_{L^{2}}^{2} \lambda(t)^{-2} dt + T^{-9}.
\end{equation}
Meanwhile, since $\nabla \cdot (\frac{x}{|x|} \phi(x) )$ is smooth,
\begin{equation}
\aligned
(1 + \frac{4}{d}) \int \int \nabla \cdot (\phi(x) \frac{x}{|x|}) Im[\overline{P_{\leq k_{0}} F(P_{\leq k_{0}} u)}(P_{\geq k_{0} + 9} u)] dx dt \lesssim T^{-5} \| P_{\geq k_{0} + 9} u \|_{L_{t}^{2} L_{x}^{\frac{2d}{d - 2}}} \\ \lesssim \frac{1}{T} \int_{J} \| \epsilon(t) \|_{L^{2}}^{2} \lambda(t)^{-2} dt + T^{-9}.
\endaligned
\end{equation}
Therefore, the error arising from frequency truncation is bounded by
\begin{equation}\label{10.44}
2^{k_{0}} R (\frac{1}{T} \int_{J} \| \epsilon(t) \|_{L^{2}}^{2} \lambda(t)^{-2} dt) + T^{-9}.
\end{equation}

Using the fact that $(\frac{1}{r} \phi(r) - \chi^{2}(r)) (\delta_{jk} - \frac{x_{j} x_{k}}{|x|^{2}})$ is a positive definite matrix, we have proved
\begin{equation}\label{10.45}
\aligned
2 \int_{a}^{b} \int \chi^{2}(\frac{x}{R}) |\nabla P_{\leq k_{0} + 9} u|^{2} dx dt - \frac{1}{2R^{2}} \int_{a}^{b} \int \psi''(\frac{x}{R}) |P_{\leq k_{0} + 9} u|^{2} dx dt \\ - \frac{2d}{d + 2} \int_{a}^{b} \int \chi^{2}(\frac{x}{R}) |P_{\leq k_{0} + 9} u|^{\frac{2(d + 2)}{d}} dx dt \lesssim M(b) - M(a) + \frac{1}{\eta_{1}^{2} T^{9}} + 2^{k_{0}} R(\frac{1}{T} \int_{0}^{T} \| \epsilon(t) \|_{L^{2}}^{2} \lambda(t)^{-2} dt)).
\endaligned
\end{equation}

The estimate of $M(b) - M(a)$ under the assumption that $\xi(a) = x(b) = 0$ is the same as in dimension $d = 2$. Indeed, since $Q$ is smooth, real valued, rapidly decreasing, all its derivatives are rapidly decreasing, by $(\ref{5.4})$, $(\ref{5.5})$, and $(\ref{10.30})$, $\xi(b) = x(a) = 0$, as in dimension $d = 2$,
\begin{equation}\label{10.47}
\aligned
\int \phi(|x|) \frac{x}{|x|} \cdot Im[\overline{P_{\leq k_{0} + 9} (e^{-i \gamma(t)} e^{-ix \cdot \frac{\xi(t)}{\lambda(t)}} \frac{1}{\lambda(t)^{d/2}} Q(\frac{x - x(t)}{\lambda(t)}))} \\ \times \nabla (P_{\leq k_{0} + 9} (e^{-i \gamma(t)} e^{-ix \cdot \frac{\xi(t)}{\lambda(t)}} \frac{1}{\lambda(t)^{d/2}} Q(\frac{x - x(t)}{\lambda(t)})))] dx|_{a}^{b} = O(2^{k_{0}} R T^{-10}).
\endaligned
\end{equation}

Also,
\begin{equation}\label{10.49}
\aligned
\int \phi(|x|) \frac{x}{|x|} \cdot Im[\overline{P_{\leq k_{0} + 9} (e^{-i \gamma(t)} e^{-ix \cdot \frac{\xi(t)}{\lambda(t)}} \frac{1}{\lambda(t)^{d/2}} \epsilon(\frac{x - x(t)}{\lambda(t)}))} \\ \times \nabla (P_{\leq k_{0} + 9} (e^{-i \gamma(t)} e^{-ix \cdot \frac{\xi(t)}{\lambda(t)}} \frac{1}{\lambda(t)^{d/2}} Q(\frac{x - x(t)}{\lambda(t)})))] dx|_{a}^{b} \\
= -(\epsilon_{2}, x \cdot \nabla Q)_{L^{2}}|_{a}^{b} +  O(\frac{2^{2k_{0}} T^{1/25d}}{\eta_{1}^{2} T} \int_{J} \| \epsilon(t) \|_{L^{2}}^{2} \lambda(t)^{-2} dt + \frac{T^{1/25d}}{\eta_{1}^{2}} \sup_{t \in J} \frac{|\xi(t)|^{2}}{\lambda(t)^{2}} + 2^{2k_{0}} \frac{T^{1/25d}}{\eta_{1}^{2}} T^{-10}).
\endaligned
\end{equation}
Also, integrating by parts, since $Q$ and all its derivatives are rapidly decreasing, as well as that $\phi(|x|) \frac{x}{|x|}$ is smooth,
\begin{equation}\label{10.50}
\aligned
\int \phi(|x|) \frac{x}{|x|} \cdot Im[\overline{P_{\leq k_{0} + 9} (e^{-i \gamma(t)} e^{-ix \cdot \frac{\xi(t)}{\lambda(t)}} \frac{1}{\lambda(t)^{d/2}} Q(\frac{x - x(t)}{\lambda(t)}))} \\ \times \nabla (P_{\leq k_{0} + 9} (e^{-i \gamma(t)} e^{-ix \cdot \frac{\xi(t)}{\lambda(t)}} \frac{1}{\lambda(t)^{d/2}} \epsilon(\frac{x - x(t)}{\lambda(t)})))] dx  \\ = (\ref{10.49}) - \int \nabla \cdot (\phi(|x|) \frac{x}{|x|}) \cdot Im[\frac{1}{\lambda(t)^{d/2}} \overline{ Q(\frac{x - x(t)}{\lambda(t)}))} ( \frac{1}{\lambda(t)^{d/2}} \epsilon(\frac{x - x(t)}{\lambda(t)})))] dx + O(T^{-10})  \\ = (\ref{10.49}) - d (\epsilon_{2}, Q)_{L^{2}} + O(T^{-10}).
\endaligned
\end{equation}
Finally, by Theorem $\ref{t4.4}$,
\begin{equation}\label{10.51}
\aligned
\int \phi(|x|) \frac{x}{|x|} \cdot Im[\overline{P_{\leq k_{0} + 9} (e^{-i \gamma(t)} e^{-ix \cdot \frac{\xi(t)}{\lambda(t)}} \frac{1}{\lambda(t)^{d/2}} \epsilon(\frac{x - x(t)}{\lambda(t)}))} \\ \times  \nabla (P_{\leq k_{0} + 9} (e^{-i \gamma(t)} e^{-ix \cdot \frac{\xi(t)}{\lambda(t)}} \frac{1}{\lambda(t)^{d/2}} \epsilon(\frac{x - x(t)}{\lambda(t)})))] dx \\
\lesssim R \| \epsilon \|_{L^{2}} \| P_{\leq k_{0} + 9} (e^{-i \gamma(t)} e^{-ix \cdot \frac{\xi(t)}{\lambda(t)}} \frac{1}{\lambda(t)^{d/2}} \epsilon(\frac{x - x(t)}{\lambda(t)})) \|_{\dot{H}^{1}} + R \frac{|\xi(t)|}{\lambda(t)} \| \epsilon \|_{L^{2}}^{2} \\
\lesssim R (\frac{2^{2k_{0}} T^{1/25d}}{\eta_{1}^{2} T} \int_{J} \| \epsilon(t) \|_{L^{2}}^{2} \lambda(t)^{-2} dt + \frac{T^{1/25d}}{\eta_{1}^{2}} \sup_{t \in J} \frac{|\xi(t)|^{2}}{\lambda(t)^{2}} + 2^{2k_{0}} \frac{T^{1/25d}}{\eta_{1}^{2}} T^{-10}).
\endaligned
\end{equation}
Therefore, letting $\Lambda$ denote the operator
\begin{equation}\label{10.52}
\Lambda f = x \cdot \nabla f + \frac{d}{2} f,
\end{equation}
\begin{equation}\label{10.53}
\aligned
M(b) - M(a) = 2 (\epsilon_{2}(a), \Lambda Q)_{L^{2}} - 2(\epsilon_{2}(b), \Lambda Q)_{L^{2}} \\ + O(\frac{2^{2k_{0}} T^{2/15d}}{\eta_{1}^{2} T} \int_{J} \| \epsilon(t) \|_{L^{2}}^{2} \lambda(t)^{-2} dt + \frac{T^{2/15d}}{\eta_{1}^{2}} \sup_{t \in J} \frac{|\xi(t)|^{2}}{\lambda(t)^{2}} + 2^{2k_{0}} \frac{T^{2/15d}}{\eta_{1}^{2}} T^{-10})).
\endaligned
\end{equation}

Therefore, to complete the proof of Theorem $\ref{t10.2}$, it only remains to obtain a lower bound for
\begin{equation}\label{10.54}
\aligned
2 \int_{a}^{b} \int \chi^{2}(\frac{x}{2R}) |\nabla P_{\leq k_{0} + 9} u|^{2} dx dt - \frac{1}{2R^{2}} \int_{a}^{b} \int \psi''(\frac{x}{2R}) |P_{\leq k_{0} + 9} u|^{2} dx dt \\ - \frac{2d}{d + 2} \int_{a}^{b} \int \chi^{2}(\frac{x}{2R}) |P_{\leq k_{0} + 9} u|^{\frac{2(d + 2)}{d}} dx dt.
\endaligned
\end{equation}

Again splitting
\begin{equation}\label{10.55}
|P_{\leq k_{0} + 9} u|^{2} \leq 2 \lambda(t)^{-d} |P_{\leq k_{0} + 9}(e^{-i \gamma(t)} e^{ix \cdot \frac{\xi(t)}{\lambda(t)}} Q(\frac{x - x(t)}{\lambda(t)})|^{2} + 2 \lambda(t)^{-d} |P_{\leq k_{0} + 9}(e^{-i \gamma(t)} e^{ix \cdot \frac{\xi(t)}{\lambda(t)}} \epsilon(t, \frac{x - x(t)}{\lambda(t)})|^{2}.
\end{equation}
By $(\ref{5.4})$, $(\ref{5.5})$, and $(\ref{10.30})$, the support of $\psi''(x)$, and the fact that $Q$ is smooth and all its derivatives are rapidly decreasing,
\begin{equation}\label{10.56}
\frac{1}{R^{2}} \lambda(t)^{-d} \int \psi''(\frac{x}{R}) |P_{\leq k_{0} + 9}(e^{-i \gamma(t)} e^{ix \cdot \frac{\xi(t)}{\lambda(t)}}  Q(\frac{x - x(t)}{\lambda(t)})|^{2} dx \lesssim \lambda(t)^{-2} \frac{1}{T^{10}}.
\end{equation}
On the other hand, by $(\ref{5.4})$, $(\ref{5.5})$, and $(\ref{10.30})$, for $T$ sufficiently large,
\begin{equation}\label{10.57}
\frac{1}{R^{2}} \lambda(t)^{-d} \int \psi''(\frac{x}{R}) |P_{\leq k + 9}(e^{-i \gamma(t)} e^{ix \cdot \frac{\xi(t)}{\lambda(t)}} \epsilon(t, \frac{x - x(t)}{\lambda(t)})|^{2} dx \lesssim  \frac{1}{R} \lambda(t)^{-2} \| \epsilon \|_{L^{2}}^{2}.
\end{equation}

As in two dimensions, since $Q$ is smooth and all its derivatives are rapidly decreasing, by $(\ref{5.4})$, $(\ref{5.5})$, and $(\ref{10.30})$,
\begin{equation}\label{10.58}
\aligned
\frac{1}{2} \int (1 - \chi^{2}(\frac{x}{2R})) |\nabla P_{\leq k_{0} + 9} (e^{-i \gamma(t)} e^{-ix \cdot \xi(t)} \frac{1}{\lambda(t)^{d/2}} Q(\frac{x - x(t)}{\lambda(t)})|^{2} dx \\ - \frac{d}{2(d + 2)} \int (1 - \chi^{2}(\frac{x}{2R})) |P_{\leq k_{0} + 9} (e^{-i \gamma(t)} e^{-ix \cdot \xi(t)} \frac{1}{\lambda(t)} Q(\frac{x - x(t)}{\lambda(t)})|^{\frac{2(d + 2)}{d}} dx = O(T^{-10}).
\endaligned
\end{equation}
Also,
\begin{equation}\label{10.59}
\aligned
\int (1 - \chi^{2}(\frac{x}{2R})) Re(\overline{\nabla P_{\leq k_{0} + 9} (e^{-i \gamma(t)} e^{-ix \cdot \xi(t)} \frac{1}{\lambda(t)^{d/2}} Q(\frac{x - x(t)}{\lambda(t)})} \\ \cdot \nabla P_{\leq k_{0} + 9} (e^{-i \gamma(t)} e^{-ix \cdot \xi(t)} \frac{1}{\lambda(t)^{d/2}} \epsilon(t, \frac{x - x(t)}{\lambda(t)})) dx \\ -  \int (1 - \chi^{2}(\frac{x}{2R})) |P_{\leq k_{0} + 9} (e^{-i \gamma(t)} e^{-ix \cdot \xi(t)} \frac{1}{\lambda(t)^{d/2}} Q(\frac{x - x(t)}{\lambda(t)})|^{\frac{4}{d}} \\ \times Re(\overline{ P_{\leq k_{0} + 9} (e^{-i \gamma(t)} e^{-ix \cdot \xi(t)} \frac{1}{\lambda(t)^{d/2}} Q(\frac{x - x(t)}{\lambda(t)})} \\ \cdot  P_{\leq k_{0} + 9} (e^{-i \gamma(t)} e^{-ix \cdot \xi(t)} \frac{1}{\lambda(t)^{d/2}} \epsilon(t, \frac{x - x(t)}{\lambda(t)})) dx = O(2^{2k_{0}} T^{-10}).
\endaligned
\end{equation}
Therefore, from $(\ref{4.55})$,
\begin{equation}\label{10.60}
\aligned
\frac{1}{2} \int \chi^{2}(\frac{x}{2R}) |\nabla P_{\leq k_{0} + 9} (e^{-i \gamma(t)} e^{-ix \cdot \xi(t)} \frac{1}{\lambda(t)^{d/2}} Q(\frac{x - x(t)}{\lambda(t)}))|^{2} dx \\ - \frac{d}{2(d + 2)} \int  \chi^{2}(\frac{x}{2R}) |P_{\leq k_{0} + 9} (e^{-i \gamma(t)} e^{-ix \cdot \xi(t)} \frac{1}{\lambda(t)^{d/2}} Q(\frac{x - x(t)}{\lambda(t)}))|^{\frac{2(d + 2)}{d}} dx \\
+ \int  \chi^{2}(\frac{x}{2R}) Re(\overline{\nabla P_{\leq k_{0} + 9} (e^{-i \gamma(t)} e^{-ix \cdot \xi(t)} \frac{1}{\lambda(t)} Q(\frac{x - x(t)}{\lambda(t)})}) \cdot \nabla P_{\leq k_{0} + 9} (e^{-i \gamma(t)} e^{-ix \cdot \xi(t)} \frac{1}{\lambda(t)} \epsilon(t, \frac{x - x(t)}{\lambda(t)})) dx \\ -  \int \chi^{2}(\frac{x}{2R}) |P_{\leq k_{0} + 9} (e^{-i \gamma(t)} e^{-ix \cdot \xi(t)} \frac{1}{\lambda(t)^{d/2}} Q(\frac{x - x(t)}{\lambda(t)})|^{2} \\ \times Re(\overline{ P_{\leq k_{0} + 9} (e^{-i \gamma(t)} e^{-ix \cdot \xi(t)} \frac{1}{\lambda(t)^{d/2}} Q(\frac{x - x(t)}{\lambda(t)}))} \cdot  P_{\leq k_{0} + 9} (e^{-i \gamma(t)} e^{-ix \cdot \xi(t)} \frac{1}{\lambda(t)} \epsilon(t, \frac{x - x(t)}{\lambda(t)})) dx \\
= \frac{1}{2} \frac{|\xi(t)|^{2}}{\lambda(t)^{2}} \| Q \|_{L^{2}}^{2} + \frac{1}{2 \lambda(t)^{2}} \| \epsilon \|_{L^{2}}^{2} - \frac{|\xi(t)|^{2}}{2 \lambda(t)^{2}} \| \epsilon \|_{L^{2}}^{2} + O(2^{2k_{0}} T^{-10}).
\endaligned
\end{equation}

Turning to the terms with two $\epsilon$'s, by the product rule,
\begin{equation}\label{10.61}
\aligned
\frac{1}{2} \int \chi^{2}(\frac{x}{2R}) |\nabla P_{\leq k_{0} + 9} (e^{-i \gamma(t)} e^{-ix \cdot \xi(t)} \frac{1}{\lambda(t)^{d/2}} \epsilon(t, \frac{x - x(t)}{\lambda(t)})|^{2} dx \\ = \frac{1}{2} \| \chi(\frac{x}{2R}) P_{\leq k_{0} + 9} (e^{-i \gamma(t)} e^{-ix \cdot \xi(t)} \frac{1}{\lambda(t)^{d/2}} \epsilon(t, \frac{x - x(t)}{\lambda(t)})) \|_{\dot{H}^{1}}^{2} + O(\frac{1}{R \lambda(t)^{2}} \| \epsilon \|_{L^{2}}^{2}).
\endaligned
\end{equation}
Then by $(\ref{5.4})$, $(\ref{5.5})$, $(\ref{10.30})$, $(\ref{4.32})$, and the fact that $Q_{x_{j}}$ and $\chi_{0}$ are rapidly decreasing,
\begin{equation}\label{10.62}
(\chi(\frac{x \lambda(t) + x(t)}{2R}) \epsilon, f)_{L^{2}} \lesssim T^{-10}.
\end{equation}
Therefore, following the analysis in $(\ref{4.56})$--$(\ref{4.64})$,
\begin{equation}\label{10.63}
\aligned
\frac{1}{2} \int \chi^{2}(\frac{x}{2R}) |\nabla P_{\leq k_{0} + 9} (e^{-i \gamma(t)} e^{-ix \cdot \xi(t)} \frac{1}{\lambda(t)^{d/2}} \epsilon(t, \frac{x - x(t)}{\lambda(t)})|^{2} dx \\ - \frac{d + 2}{2d\lambda(t)^{d + 2}} \int |P_{\leq k_{0} + 9}(e^{-i \gamma(t)} e^{-ix \cdot \frac{\xi(t)}{\lambda(t)}}Q(\frac{x - x(t)}{\lambda(t)}))|^{2} |P_{\leq k_{0} + 9} (e^{-i \gamma(t)} e^{-ix \cdot \xi(t)} \epsilon(t,\frac{x - x(t)}{\lambda(t)}))|^{2} dx \\ 
- \frac{1}{d \lambda(t)^{d + 2}} Re \int (\overline{P_{\leq k + 9}(e^{-i \gamma(t)} e^{-ix \cdot \frac{\xi(t)}{\lambda(t)}}Q(\frac{x - x(t)}{\lambda(t)}))})^{2} (P_{\leq k + 9} (e^{-i \gamma(t)} e^{-ix \cdot \xi(t)} \epsilon(t,\frac{x - x(t)}{\lambda(t)})))^{2} dx \\
\geq \frac{\lambda_{d}}{\lambda(t)^{2}} \| \epsilon \|_{L^{2}}^{2} + \lambda_{d} \int \chi(\frac{x}{2R})^{2} |\nabla P_{\leq k_{0} + 9} (e^{-i \gamma(t)} e^{-ix \cdot \xi(t)} \frac{1}{\lambda(t)^{d/2}} \epsilon(t, \frac{x - x(t)}{\lambda(t)})|^{2} dx  - O(2^{2k_{0}} T^{-10}).
\endaligned
\end{equation}

By the Sobolev embedding theorem, for $g \in H^{1}$,
\begin{equation}\label{10.64}
\aligned
\int \chi^{2}(\frac{x}{2R}) |g(x)|^{\frac{2(d + 2)}{d}} dx \lesssim \| \chi(\frac{x}{2R}) g \|_{\dot{H}^{1}}^{2} \| g \|_{L^{2}}^{4/d}.
\endaligned
\end{equation}
Therefore, by the product rule,
\begin{equation}\label{10.65}
\aligned
\int \chi(\frac{x}{2R})^{2} |P_{\leq k_{0} + 9}(e^{-i \gamma(t)} e^{-ix \cdot \xi(t)} \lambda(t)^{-d/2} \epsilon(t, \frac{x - x(t)}{\lambda(t)}))|^{\frac{2(d + 2)}{d}} dx \\\lesssim \eta_{0}^{4/d} \| \chi(\frac{x}{2R}) \nabla P_{\leq k_{0} + 9}(e^{-i \gamma(t)} e^{-ix \cdot \xi(t)} \lambda(t)^{-d/2} \epsilon(t, \frac{x - x(t)}{\lambda(t)}) \|_{L^{2}}^{2} + \frac{\eta_{0}^{4/d}}{R^{2}} \| \epsilon \|_{L^{2}}^{2}.
\endaligned
\end{equation}
Therefore, we have proved that for any $d \geq 3$, there exists some $\lambda_{d} > 0$ such that for $\eta_{0}$ sufficiently small, \begin{equation}\label{10.66}
(\ref{10.45}) \gtrsim \lambda_{d} \int_{a}^{b} \| \epsilon \|_{L^{2}}^{2} \lambda(t)^{-2} dt - O(T^{-8}).
\end{equation}
Combining $(\ref{10.66})$, $(\ref{10.45})$, $(\ref{10.56})$, $(\ref{10.57})$, and $(\ref{10.45})$ proves Theorem $\ref{t10.2}$.
\end{proof}

\section{An $L_{s}^{p}$ bound on $\| \epsilon(s) \|_{L^{2}}$ when $p > 1$}
As in one dimension, Theorems $\ref{t10.1}$ and $\ref{t10.2}$ imply that $\| \epsilon(s) \|_{L^{2}}$ lies in $L_{s}^{p}$ for any $p > 1$.
\begin{theorem}\label{t11.1}
Let $u$ be a solution to $(\ref{1.1})$ that satisfies $\| u \|_{L^{2}} = \| Q \|_{L^{2}}$, and suppose
\begin{equation}\label{11.1}
\sup_{s \in [0, \infty)} \| \epsilon(s) \|_{L^{2}} \leq \eta_{\ast},
\end{equation}
and $\| \epsilon(0) \|_{L^{2}} = \eta_{\ast}$. Then
\begin{equation}\label{11.2}
\int_{0}^{\infty} \| \epsilon(s) \|_{L^{2}}^{2} ds \lesssim \eta_{\ast},
\end{equation}
with implicit constant independent of $\eta_{\ast}$ when $\eta_{\ast} \ll 1$ is sufficiently small.

Furthermore, for any $j \in \mathbb{Z}_{\geq 0}$, let
\begin{equation}\label{11.3}
s_{j} = \inf \{ s \in [0, \infty) : \| \epsilon(s) \|_{L^{2}} = 2^{-j} \eta_{\ast} \}.
\end{equation}
By definition, $s_{0} = 0$, and the continuity of $\| \epsilon(s) \|_{L^{2}}$ combined with Theorem $\ref{t2.4}$ implies that such an $s_{j}$ exists for any $j > 0$. Then,
\begin{equation}\label{11.4}
\int_{s_{j}}^{\infty} \| \epsilon(s) \|_{L^{2}}^{2} ds \lesssim 2^{-j} \eta_{\ast},
\end{equation}
for each $j \geq 0$, with implicit constant independent of $\eta_{\ast}$.
\end{theorem}

\begin{proof}
Set $T_{\ast} = \frac{1}{\eta_{\ast}}$ and suppose that $T_{\ast}$ is sufficiently large such that Theorems $\ref{t10.1}$ and $\ref{t10.2}$ hold. Then by $(\ref{3.23})$ and $(\ref{11.1})$, for any $s' \geq 0$,
\begin{equation}\label{11.5}
|\sup_{s \in [s', s' + T_{\ast}]} \ln(\lambda(s)) - \inf_{s \in [s', s' + T_{\ast}]} \ln(\lambda(s))| \lesssim 1,
\end{equation}
with implicit constant independent of $s' \geq 0$. Let $J$ be the largest dyadic integer that satisfies
\begin{equation}\label{11.6}
J = 2^{j_{\ast}} \leq -\ln(\eta_{\ast})^{1/4}.
\end{equation}
By $(\ref{11.5})$ and the triangle inequality,
\begin{equation}\label{11.7}
\aligned
|\sup_{s \in [s', s' + J T_{\ast}]} \ln(\lambda(s)) - \inf_{s \in [s', s' + J T_{\ast}]} \ln(\lambda(s))| \lesssim J,
\endaligned
\end{equation}
and therefore,
\begin{equation}\label{11.8}
\frac{\sup_{s \in [s', s' + 3 J T_{\ast}]} \lambda(s)}{\inf_{s \in [s', s' + 3JT^{\ast}]} \lambda(s)} \lesssim T_{\ast}^{\frac{1}{5000d^{2}}}.
\end{equation}
Rescale so that
\begin{equation}\label{11.9}
\frac{1}{\eta_{1}} \leq \lambda(s) \leq \frac{1}{\eta_{1}} T_{\ast}^{\frac{1}{5000d^{2}}} \qquad \text{for any} \qquad s \in [s', s' + 3J T_{\ast}].
\end{equation}
Now take a subset $[a, b] \subset [s', s' + 3 J T_{\ast}]$, make a Galilean transformation so that $\xi(a) = 0$, and a translation in space so that $x(b) = 0$. By $(\ref{3.22})$ and $(\ref{11.9})$,
\begin{equation}\label{11.10}
\sup_{s \in [s', s' + 3J T_{\ast}]} |\frac{\xi(s)}{\lambda(s)}| \lesssim \eta_{1} J \eta_{\ast} \ll \eta_{0}.
\end{equation}
Also by $(\ref{3.24})$,
\begin{equation}\label{11.11}
|x(s)| \lesssim \frac{1}{\eta_{1}^{2}} T_{\ast}^{\frac{1}{5000d^{2}}} \eta_{1} J + \frac{1}{\eta_{1}} T_{\ast}^{\frac{1}{5000d^{2}}} J \ll T_{\ast}^{\frac{1}{2000d^{2}}}.
\end{equation}
Therefore, Theorems $\ref{t10.1}$ and $\ref{t10.2}$ may be utilized on $[s', s' + J T_{\ast}]$, proving that for any $s' \geq 0$,
\begin{equation}\label{11.12}
\int_{s'}^{s' + J T_{\ast}} \| \epsilon(s) \|_{L^{2}}^{2} ds \lesssim \| \epsilon(s') \|_{L^{2}} + \| \epsilon(s' + J T_{\ast}) \|_{L^{2}} + \eta_{1}^{2} J^{2} \eta_{\ast}^{2} + O(\frac{1}{J^{9} T_{\ast}^{9}}).
\end{equation}
Note that the left hand side of $(\ref{11.12})$ is scale invariant.

Moreover, for any $s' > J T_{\ast}$,
\begin{equation}\label{11.13}
\int_{s'}^{s' + J T_{\ast}} \| \epsilon(s) \|_{L^{2}}^{2} ds \lesssim \inf_{s \in [s' - J T_{\ast}, s']} \| \epsilon(s) \|_{L^{2}} + \inf_{s \in [s' + J T_{\ast}, s' + 2J T_{\ast}]} \| \epsilon(s) \|_{L^{2}} + \eta_{1}^{2} J^{2} \eta_{\ast}^{2} + O(\frac{1}{J^{9} T_{\ast}^{9}}).
\end{equation}
In particular, for a fixed $s' \geq 0$,
\begin{equation}\label{11.14}
\sup_{a > 0} \int_{s' + a J T_{\ast}}^{s' + (a + 1) J T_{\ast}} \| \epsilon(s) \|_{L^{2}}^{2} \lesssim \frac{1}{J^{1/2} T_{\ast}^{1/2}} (\sup_{a \geq 0} \int_{s' + a J T_{\ast}}^{s' + (a + 1) J T_{\ast}} \| \epsilon(s) \|_{L^{2}}^{2} ds)^{1/2} + \eta_{1}^{2} J^{2} \eta_{\ast}^{2} + O(\frac{1}{J^{9} T_{\ast}^{9}}).
\end{equation}
Meanwhile, when $a = 0$,
\begin{equation}\label{11.15}
 \int_{s'}^{s' + J T_{\ast}} \| \epsilon(s) \|_{L^{2}}^{2} \lesssim \| \epsilon(s') \|_{L^{2}} + \frac{1}{J^{1/2} T_{\ast}^{1/2}} (\sup_{a \geq 0} \int_{s' + a J T_{\ast}}^{s' + (a + 1) J T_{\ast}} \| \epsilon(s) \|_{L^{2}}^{2} ds)^{1/2} + \eta_{1}^{2} J^{2} \eta_{\ast}^{2} + O(\frac{1}{J^{9} T_{\ast}^{9}}).
\end{equation}
Therefore, taking $s' = s_{j_{\ast}}$,
\begin{equation}\label{11.16}
\sup_{a \geq 0} \int_{s_{j_{\ast}} + a J T_{\ast}}^{s_{j_{\ast}} + (a + 1) J T_{\ast}} \| \epsilon(s) \|_{L^{2}}^{2} ds \lesssim 2^{-j_{\ast}} \eta_{\ast} + \eta_{1}^{2} J^{2} \eta_{\ast}^{2} + O(2^{-9 j_{\ast}} \eta_{\ast}^{9}).
\end{equation}
Then by the triangle inequality,
\begin{equation}\label{11.17}
\sup_{s' \geq s_{j_{\ast}}} \int_{s'}^{s' + J T_{\ast}} \| \epsilon(s) \|_{L^{2}}^{2} ds \lesssim 2^{-j_{\ast}} \eta_{\ast},
\end{equation}
and by H{\"o}lder's inequality,
\begin{equation}\label{11.18}
\sup_{s' \geq s_{j_{\ast}}} \int_{s'}^{s' + J T_{\ast}} \| \epsilon(s) \|_{L^{2}} ds \lesssim 1.
\end{equation}

Repeating this argument, Theorem $\ref{t11.1}$ can be proved by induction. Indeed, fix a constant $C < \infty$ and suppose that there exists a positive integer $n_{0}$ such that for all integers $0 \leq n \leq n_{0}$,
\begin{equation}\label{11.19}
\sup_{s' \geq s_{nj_{\ast}}} \int_{s'}^{s' + J^{n} T_{\ast}} \| \epsilon(s) \|_{L^{2}} ds \leq C, \qquad \sup_{s' \geq s_{nj_{\ast}}} \int_{s'}^{s' + J^{n} T_{\ast}} \| \epsilon(s) \|_{L^{2}}^{2} ds \leq C J^{-n} \eta_{\ast}.
\end{equation}
Then by $(\ref{11.5})$, for $s' \geq s_{n j_{\ast}}$,
\begin{equation}\label{11.20}
\frac{\sup_{s \in [s', s' + 3 J^{n + 1} T_{\ast}]} \lambda(s)}{\inf_{s \in [s', s' + 3 J^{n + 1} T_{\ast}]} \lambda(s)} \lesssim T_{\ast}^{\frac{1}{5000d^{2}}}.
\end{equation}
\begin{remark}
The $C$ in $(\ref{11.19})$ will ultimately be given by the implicit constants in Theorems $\ref{t10.1}$ and $\ref{t10.2}$, so for $T_{\ast}$ sufficiently large, $(\ref{11.20})$ will hold.
\end{remark}
Rescale so that $(\ref{11.9})$ holds. Also, for $[a, b] \subset [s', s' + 3 J^{n + 1} T_{\ast}]$, setting $\xi(a) = 0$ and $x(b) = 0$, by $(\ref{3.22})$ and $(\ref{11.19})$,
\begin{equation}\label{11.21}
\sup_{s \in [s', s' + 3 J^{n + 1} T_{\ast}]} \frac{|\xi(s)|}{\lambda(s)} \lesssim C J^{-n} \eta_{1} \eta_{\ast},
\end{equation}
and by $(\ref{3.24})$ and $(\ref{11.19})$,
\begin{equation}\label{11.22}
\aligned
|x(s)| \lesssim \sup \lambda(s) \int_{s'}^{s' + 3 J^{n + 1} T_{\ast}} \| \epsilon(s) \|_{L^{2}} ds + \sup \lambda(s)^{2} \cdot \sup \frac{|\xi(s)|}{\lambda(s)} \int_{s'}^{s' + 3 J^{n + 1} T_{\ast}} 1 ds \\
\lesssim \frac{1}{\eta_{1}} T_{\ast}^{\frac{1}{5000d^{2}}} CJ + \frac{1}{\eta_{1}^{2}} T_{\ast}^{1/25d} C J^{-n} \eta_{1} \eta_{\ast} J^{n + 1} T_{\ast} \lesssim \frac{1}{\eta_{1}} T_{\ast}^{\frac{1}{5000d^{2}}} CJ + \frac{1}{\eta_{1}} T_{\ast}^{\frac{1}{2500d^{2}}} CJ \ll T_{\ast}^{\frac{1}{2000d^{2}}}.
\endaligned
\end{equation}
Then by Theorems $\ref{t10.1}$ and $\ref{t10.2}$,
\begin{equation}\label{11.23}
\sup_{s' \geq s_{(n + 1) j_{\ast}}} \int_{s'}^{s' + J^{n + 1} T_{\ast}} \| \epsilon(s) \|_{L^{2}}^{2} ds \leq C J^{-(n + 1)} T_{\ast}^{-1},
\end{equation}
and by H{\"o}lder's inequality,
\begin{equation}\label{11.24}
\sup_{s' \geq s_{(n + 1) j_{\ast}}} \int_{s'}^{s' + J^{n + 1} T_{\ast}} \| \epsilon(s) \|_{L^{2}} ds \leq C.
\end{equation}
Therefore, $(\ref{11.19})$ holds for any integer $n > 0$.


Now take any $j \in \mathbb{Z}$ and suppose $n j_{\ast} < j \leq (n + 1) j_{\ast}$. Then $(\ref{11.20})$--$(\ref{11.22})$ hold on $[s_{j} + a J^{n + 1} T_{\ast}, s_{j} + (a + 1) J^{n + 1} T_{\ast}]$ for any $a \geq 0$, so
by Theorems $\ref{t10.1}$ and $\ref{t10.2}$,
\begin{equation}\label{11.26}
\sup_{a \geq 0} \int_{s_{j} + a J^{n + 1} T_{\ast}}^{s_{j} + (a + 1) J^{n + 1} T_{\ast}} \| \epsilon(s) \|_{L^{2}}^{2} ds \lesssim 2^{-j} \eta_{\ast},
\end{equation}
and therefore by H{\"o}lder's inequality, for any $s' \geq s_{j}$,
\begin{equation}\label{11.27}
\sup_{s' \geq s_{j}} \int_{s'}^{s' + 2^{j} T_{\ast}} \| \epsilon(s) \|_{L^{2}} ds \lesssim 1,
\end{equation}
with bound independent of $j$. Inequalities $(\ref{11.26})$ and $(\ref{11.27})$ imply that $(\ref{11.20})$--$(\ref{11.22})$ hold on $[s', s' + 3 \cdot 2^{j} J T_{\ast}]$ for any $s' \geq s_{j}$, so
\begin{equation}\label{11.28}
\int_{s_{j}}^{s_{j} + 2^{j} J T_{\ast}} \| \epsilon(s) \|_{L^{2}}^{2} \lesssim 2^{-j} \eta_{\ast},
\end{equation}
and therefore, by the mean value theorem,
\begin{equation}\label{11.29}
\inf_{s \in [s_{j}, s_{j} + 2^{j} J T_{\ast}]} \| \epsilon(s) \|_{L^{2}} \lesssim 2^{-j} \eta_{\ast} J^{-1/2},
\end{equation}
which implies
\begin{equation}\label{11.30}
s_{j + 1} \in [s_{j}, s_{j} + 2^{j} J T_{\ast}].
\end{equation}
Therefore, by $(\ref{11.28})$ and H{\"o}lder's inequality,
\begin{equation}\label{11.31}
\int_{s_{j}}^{s_{j + 1}} \| \epsilon(s) \|_{L^{2}}^{2} ds \lesssim 2^{-j} \eta_{\ast}, \qquad \text{and} \qquad \int_{s_{j}}^{s_{j + 1}} \| \epsilon(s) \|_{L^{2}} ds \lesssim 1,
\end{equation}
with constant independent of $j$. Summing in $j$ gives $(\ref{11.2})$ and $(\ref{11.4})$.
\end{proof}

Now then, $(\ref{3.25})$ and $(\ref{11.2})$ imply
\begin{equation}\label{11.33}
\lim_{s \rightarrow \infty} \| \epsilon(s) \|_{L^{2}} = 0.
\end{equation}
Next, by definition of $s_{j}$, $(\ref{11.4})$ implies
\begin{equation}\label{11.34}
\int_{s_{j}}^{s_{j + 1}} \| \epsilon(s) \|_{L^{2}} ds \lesssim 1,
\end{equation}
and for any $1 < p < \infty$,
\begin{equation}\label{11.35}
(\int_{s_{j}}^{s_{j + 1}} \| \epsilon(s) \|_{L^{2}}^{p} ds) \lesssim \eta_{\ast}^{p - 1} 2^{-j(p - 1)}, 
\end{equation}
which implies that $\| \epsilon(s) \|_{L^{2}}$ belongs to $L_{s}^{p}$ for any $p > 1$, but not $L_{s}^{1}$.

Comparing $(\ref{11.35})$ to the pseudoconformal transformation of the soliton, $(\ref{1.15})$, for $0 < t < 1$,
\begin{equation}\label{11.36}
\lambda(t) \sim t, \qquad \text{and} \qquad \| \epsilon(t) \|_{L^{2}} \sim t,
\end{equation}
so
\begin{equation}\label{11.36}
\int_{0}^{1} \| \epsilon(t) \|_{L^{2}} \lambda(t)^{-2} dt = \infty,
\end{equation}
but for any $p > 1$,
\begin{equation}\label{11.38}
\int_{0}^{1} \| \epsilon(t) \|_{L^{2}}^{p} \lambda(t)^{-2} dt < \infty.
\end{equation}
For the soliton, $\epsilon(s) \equiv 0$ for any $s \in \mathbb{R}$, so obviously, $\| \epsilon(s) \|_{L^{2}} \in L_{s}^{p}$ for $1 \leq p \leq \infty$.

\section{Monotonicity of $\lambda$}
As in the one dimensional case, it is possible to use the virial identity from \cite{merle2005blow}, to show that $\lambda(s)$ is an approximately monotone decreasing function.
\begin{theorem}\label{t8.2}
For any $s \geq 0$, let
\begin{equation}\label{14.0}
\tilde{\lambda}(s) = \inf_{\tau \in [0, s]} \lambda(\tau).
\end{equation}
Then for any $s \geq 0$,
\begin{equation}\label{8.0}
1 \leq \frac{\lambda(s)}{\tilde{\lambda}(s)} \leq 3.
\end{equation}

\end{theorem}

\begin{proof}
Suppose there exist $0 \leq s_{-} \leq s_{+} < \infty$ satisfying
\begin{equation}\label{8.19}
\frac{\lambda(s_{+})}{\lambda(s_{-})} = e.
\end{equation}
Then we can show that $u$ is a soliton solution to $(\ref{1.1})$, which is a contradiction, since $\lambda(s)$ is constant in that case.

The proof that $(\ref{8.19})$ implies that $u$ is a soliton uses a virial identity from \cite{merle2005blow}, combined with the $L_{s}^{p}$ bounds on $\| \epsilon(s) \|_{L^{2}}$ obtained in Theorem $\ref{t11.1}$. Using $(\ref{3.19})$, compute
\begin{equation}\label{8.41}
\aligned
\frac{d}{ds} (\epsilon, |x|^{2} Q) + \frac{\lambda_{s}}{\lambda} \| x Q \|_{L^{2}}^{2} + 4 (\frac{d}{2} Q + x \cdot \nabla Q, \epsilon_{2})_{L^{2}} \\ = O(|\gamma_{s} + 1 - \frac{x_{s}}{\lambda} \cdot \xi(s) - |\xi(s)|^{2} | \| \epsilon \|_{L^{2}}) + O(|\xi_{s} - \frac{\lambda_{s}}{\lambda} \xi(s)| \| \epsilon \|_{L^{2}}) + O(|\frac{\lambda_{s}}{\lambda}| \| \epsilon \|_{L^{2}}) + O(|\frac{x_{s}}{\lambda} + 2 \xi| \| \epsilon \|_{L^{2}}) \\
 + O(\| \epsilon \|_{L^{2}}^{2}) + O(\| \epsilon \|_{L^{2}} \| \epsilon \|_{L^{2(1 + \frac{4}{d}}}^{1 + \frac{4}{d}}), \qquad \text{if} \qquad 2 \leq d \leq 4, \qquad + O(\| \epsilon \|_{L^{2}}^{1 + 4/d}), \qquad \text{if} \qquad d > 5.
\endaligned
\end{equation}
Indeed, by direct computation,
\begin{equation}
(\nabla Q, x^{2} Q)_{L^{2}} = (iQ, x^{2} Q)_{L^{2}} = (i \nabla Q, x^{2} Q) = 0, \qquad \text{and} \qquad -\mathcal L_{-}(y^{2} Q) = -2d Q - 4x \cdot \nabla Q.
\end{equation}
Also,
\begin{equation}
\aligned
(\frac{d}{2} Q + x \cdot \nabla Q, |x|^{2} Q)_{L^{2}} = \frac{d}{2} \| xQ \|_{L^{2}}^{2} + \frac{1}{2} (|x|^{2} x, \nabla Q^{2})_{L^{2}} = \frac{d}{2} \| x Q \|_{L^{2}}^{2} + \frac{1}{8} (\nabla |x|^{4}, \nabla Q)_{L^{2}} \\
= \frac{d}{2} \| x Q \|_{L^{2}}^{2} - \frac{1}{8} (\Delta |x|^{4}, Q^{2})_{L^{2}} = \frac{d}{2} \| xQ \|_{L^{2}}^{2} - \frac{d + 2}{2} \| xQ \|_{L^{2}}^{2} = -\| xQ \|_{L^{2}}^{2}.
\endaligned
\end{equation}
Then by Theorem $\ref{t11.1}$, the fundamental theorem of calculus, and $(\ref{3.21})$--$(\ref{3.24})$,
\begin{equation}\label{8.42}
\| x Q \|_{L^{2}}^{2} + 4 \int_{s_{-}}^{s_{+}} (\epsilon_{2}, \frac{Q}{2} + x Q_{x})_{L^{2}} = O(\eta_{\ast}).
\end{equation}
Therefore, there exists $s' \in [s_{-}, s_{+}]$ such that
\begin{equation}\label{8.43}
(\epsilon_{2}, \frac{d}{2} Q + x \cdot \nabla Q)_{L^{2}} < 0.
\end{equation}
Rescale so that $\lambda(s') = \frac{1}{\eta_{1}}$. 

Since $s' \geq 0$, there exists some $j \geq 0$ such that $s_{j} \leq s' + T_{\ast} < s_{j + 1}$. Using the proof of Theorem $\ref{t11.1}$, in particular $(\ref{11.20})$--$(\ref{11.22})$, setting $\xi(s') = 0$ and $x(s_{j + 1 + J}) = 0$,
\begin{equation}\label{8.44}
\int_{s'}^{s_{j + 1 + J}} |\frac{\lambda_{s}}{\lambda}| ds \lesssim J.
\end{equation}
Then by Theorems $\ref{t10.1}$ and $\ref{t10.2}$, $(\ref{8.43})$ implies
\begin{equation}\label{8.45}
\int_{s'}^{s_{j + 1 + J}} \| \epsilon(s) \|_{L^{2}}^{2} ds \lesssim 2^{-(j + 1 + J)} \eta_{\ast},
\end{equation}
and therefore by definition of $s_{j + 1 + J}$,
\begin{equation}\label{8.46}
\int_{s'}^{s_{j + 1 + J}} \| \epsilon(s) \|_{L^{2}} ds \lesssim 1.
\end{equation}

Arguing by induction, suppose that for some $1 \leq k \leq k_{0}$,
\begin{equation}\label{8.47}
\int_{s'}^{s_{j + k}} \| \epsilon(s) \|_{L^{2}}^{2} ds \lesssim 2^{-j - k} \eta_{\ast},
\end{equation}
and
\begin{equation}\label{8.48}
\int_{s'}^{s_{j + k}} \| \epsilon(s) \|_{L^{2}} ds \lesssim 1,
\end{equation}
with implicit constant independent of $k$. By Theorem $\ref{t11.1}$,
\begin{equation}\label{8.47}
\int_{s'}^{s_{j + k + J}} \| \epsilon(s) \|_{L^{2}}^{2} ds \lesssim 2^{-j - k} \eta_{\ast},
\end{equation}
and
\begin{equation}\label{8.48}
\int_{s'}^{s_{j + k + J}} \| \epsilon(s) \|_{L^{2}} ds \lesssim J.
\end{equation}
Then by $(\ref{11.20})$--$(\ref{11.22})$ and setting $\xi(s') = x(s_{j + k + J}) = 0$,
\begin{equation}\label{8.28.1}
\sup_{s \in [s', s_{j + k + J}]} |x(s)| \lesssim T_{\ast}^{\frac{1}{2000d^{2}}}, \qquad \sup_{s \in [s', s_{j + k + J}]} |\xi(s)| \leq \eta_{0}, \qquad \frac{\sup_{s \in [s', s_{j + k + J}]} \lambda(s)}{\inf_{s \in [s', s_{j + k + J}]} \lambda(s)} \leq T_{\ast}^{\frac{1}{5000d^{2}}},
\end{equation}
so by Theorems $\ref{t10.1}$ and $\ref{t10.2}$ and $(\ref{8.43})$,
\begin{equation}\label{8.47.1}
\int_{s'}^{s_{j + k}} \| \epsilon(s) \|_{L^{2}}^{2} ds \lesssim 2^{-j - k} \eta_{\ast},
\end{equation}
and
\begin{equation}\label{8.48.1}
\int_{s'}^{s_{j + k}} \| \epsilon(s) \|_{L^{2}} ds \lesssim 1,
\end{equation}
for $1 \leq k \leq k_{0} + J$. Therefore, $(\ref{8.47.1})$ and $(\ref{8.48.1})$ hold for any $k$, with implicit constant independent of $k$.

Taking $k \rightarrow \infty$,
\begin{equation}\label{8.49}
\int_{s'}^{\infty} \| \epsilon(s) \|_{L^{2}}^{2} ds = 0,
\end{equation}
which implies that $\epsilon(s) = 0$ for all $s \geq s'$. Therefore,
\begin{equation}\label{8.50}
u_{0}(x) = e^{i \gamma} e^{ix \cdot \xi} \lambda^{d/2} Q(\lambda x + x_{0}),
\end{equation}
for some $\gamma \in \mathbb{R}$, $\lambda > 0$, $\xi \in \mathbb{R}^{d}$, $x_{0} \in \mathbb{R}^{d}$, which proves that $u$ is a soliton solution.
\end{proof}

\section{Proof of Theorem $\ref{t2.2}$}
The key difference between the soliton solution $(\ref{1.3.1})$ and the pseudoconformal transformation of the soliton $(\ref{1.5})$ is that the soliton is global, while the pseudoconformal transformation of the soliton blows up in finite time. To prove Theorem $\ref{t2.2}$, it suffices to show that a solution to $(\ref{1.1})$ satisfying the conditions of Theorem $\ref{t2.2}$ that blows up in infinite time must be a soliton. A pseudoconformal transformation of a finite time blowup solution will then show that it must be a pseudoconformal transformation of a soliton.

\begin{theorem}\label{t14.0}
In dimensions $2 \leq d \leq 15$, if $u$ is a solution to $(\ref{1.1})$ that satisfies the conditions of Theorem $\ref{t2.2}$, blows up forward in time, and
\begin{equation}\label{14.1}
\sup(I) = \infty,
\end{equation}
then $u$ is equal to a soliton solution.
\end{theorem}
\begin{remark}
This theorem is the only place where the proof of Theorem $\ref{t2.2}$ does not work in dimensions $d \geq 16$.
\end{remark}
\begin{proof}
For any integer $k \geq 0$, let
\begin{equation}\label{14.4}
I(k) = \{ s \geq 0 : 2^{-k + 2} \leq \tilde{\lambda}(s) \leq 2^{-k + 3} \}.
\end{equation}
Then by $(\ref{8.0})$,
\begin{equation}\label{14.5}
2^{-k} \leq \lambda(s) \leq 2^{-k + 3},
\end{equation}
for all $s \in I(k)$. By $(\ref{3.18})$, the fact that $\sup(I) = \infty$ implies that
\begin{equation}\label{14.7}
\sum 2^{-2k} |I(k)| = \infty.
\end{equation}
If $\lambda(s) \rightarrow 0$ as $s \rightarrow \infty$, then there exists a sequence $k_{n} \nearrow \infty$ such that
\begin{equation}\label{14.8}
|I(k_{n})| 2^{-2k_{n}} \geq \frac{1}{k_{n}^{2}}.
\end{equation}
Now let $I(k_{n}) = [a_{n}, b_{n}]$. Then by $(\ref{11.3})$ and $(\ref{11.31})$, for any $j \geq 0$,
\begin{equation}\label{14.9}
|s_{j + 1} - s_{j}| \lesssim 2^{j} \eta_{\ast}^{-1}, \qquad \| \epsilon(s_{j + 1}) \|_{L^{2}} = 2^{-(j + 1)} \eta_{\ast},
\end{equation}
and therefore there exists $s' \in [0, b_{n}]$ such that
\begin{equation}\label{14.9.1}
\| \epsilon(s_{n}') \|_{L^{2}} \lesssim k_{n}^{2} 2^{-2k_{n}}.
\end{equation}

When $d = 2$, the proof of Theorem $\ref{t14.0}$ is much simpler, so we will start with that. Make a Galilean transformation so that $\xi(s') = 0$. By $(\ref{3.22})$ and Theorem $\ref{t10.1}$, for any $0 \leq s \leq s_{n}'$,
\begin{equation}\label{14.10.3}
\frac{|\xi(s)|}{\lambda(s)} \lesssim \int_{0}^{s_{n}'} \frac{1}{\lambda(s)} \| \epsilon(s) \|_{L^{2}}^{2} ds \lesssim \eta_{1} \eta_{\ast}.
\end{equation}

Using $(\ref{1.5.1})$, rescale so that
\begin{equation}\label{14.10.2}
\frac{1}{\eta_{1}} \leq \lambda(s) \leq \frac{1}{\eta_{1}} 2^{k_{n}}, \qquad \text{for any} \qquad 0 \leq s \leq s_{n}'.
\end{equation}
Setting $t_{n}' = s^{-1}(s_{n}')$, using $(\ref{14.10.3})$--$(\ref{14.10.2})$, Corollary $\ref{c6.4}$ implies that for $r_{n} = \frac{2 k_{n}}{3}$,
\begin{equation}\label{14.14}
\| P_{\geq r_{n}} u \|_{U_{\Delta}^{2}([0, t_{n}'] \times \mathbb{R}^{2})} \lesssim \eta_{\ast}.
\end{equation}
Furthermore, arguing by induction on frequency, and using $(\ref{6.61})$ and the preceding computations,
\begin{equation}\label{14.15}
\| P_{\geq r_{n} + \frac{k_{n}}{4}} u \|_{U_{\Delta}^{2}([0, t_{n}'] \times \mathbb{R}^{2})} \lesssim k_{n}^{2} 2^{-2k_{n}}.
\end{equation}
Then using the computations in $(\ref{4.5})$,
\begin{equation}\label{14.16}
E(P_{\leq r_{n} + \frac{k_{n}}{4}} u(t_{n}')) \lesssim (k_{n}^{2} 2^{-2k_{n}} 2^{r_{n} + \frac{k_{n}}{4}})^{2} \sim (k_{n}^{2} 2^{-\frac{13k_{n}}{12}})^{2}.
\end{equation}
Next, following the computations in the proof of Theorem $\ref{t4.1}$, and using $(\ref{14.15})$,
\begin{equation}\label{14.17}
\sup_{t \in [0, t_{n}']} E(P_{\leq r_{n} + \frac{k_{n}}{4}} u(t)) \lesssim (k_{n}^{2} 2^{-\frac{13k_{n}}{12}})^{2}.
\end{equation}
Therefore, by $(\ref{4.31})$ and $(\ref{14.10.2})$,
\begin{equation}\label{14.12}
\| \epsilon(0) \|_{L^{2}}^{2} \lesssim (k_{n}^{2} 2^{-\frac{k_{n}}{12}} \eta_{1}^{-1})^{2}.
\end{equation}
Since $k_{n} \rightarrow \infty$ as $n \rightarrow \infty$, $(\ref{14.12})$ implies that $\epsilon(0) = 0$, or that $u$ is a soliton solution to $(\ref{1.1})$.

If $\lambda(s) \geq \delta > 0$ for some $\delta > 0$, then rescale so that
\begin{equation}\label{14.12.1}
\frac{1}{\eta_{1}} \leq \lambda(s) \leq \frac{1}{\eta_{1}} 2^{k_{0}},
\end{equation}
for some $k_{0} \in \mathbb{Z}_{\geq 0}$. Then for $k_{n} = n$, $|s_{n}| \lesssim 2^{n} \eta_{\ast}^{-1}$, by $(\ref{14.10.3})$, $(\ref{14.12.1})$, Corollary $\ref{c6.4}$, and $(\ref{14.14})$--$(\ref{14.17})$,
\begin{equation}\label{14.12.2}
\sup_{t \in [0, t_{n}']} E(P_{\leq r_{n} + \frac{k_{n}}{4}} u(t)) \lesssim 2^{-\frac{13k_{n}}{6}}, \qquad \text{and} \qquad \sup_{t \in [0, t_{n}']} \| \epsilon(t) \|_{L^{2}} \lesssim \eta_{1}^{-1} 2^{-\frac{13 k_{n}}{12}} 2^{k_{0}}.
\end{equation}
In this case as well, since $k_{n} \rightarrow \infty$, $\epsilon(0) = 0$.

In dimensions $d \geq 3$, the proof is complicated by two factors. The first is that the long time Strichartz estimates in Theorem $\ref{t5.1}$ depend on a bound on $|x(t)|$, which was not needed in two dimensions. The second is that in dimensions $d \geq 3$, $F(x) = |x|^{\frac{4}{d}} x$ is not a smooth function of $x$. 

First suppose $\lambda(s) \searrow 0$ as $s \rightarrow \infty$. Again rescale so that
\begin{equation}\label{14.11}
\frac{1}{\eta_{1}} \leq \lambda(s) \leq \frac{1}{\eta_{1}} 2^{k_{n}}, \qquad \text{for any} \qquad 0 \leq s \leq s_{n}'.
\end{equation}
Suppose $s_{n}' \in I(\tilde{k}_{n}) = [a_{n}, b_{n}]$, for $\tilde{k}_{n} \leq k_{n}$. By $(\ref{14.10.3})$,
\begin{equation}\label{14.10}
\frac{|\xi(s)|}{\lambda(s)} \lesssim \eta_{1} \eta_{\ast}, \qquad \text{for all} \qquad 0 \leq s \leq s_{n}'.
\end{equation}
Also, for any $s_{j} \in I(\tilde{k}_{n})$, $s_{j} \in [a_{n}, b_{n}]$, setting $x(s_{j}) = \xi(s_{j}) = 0$, by $(\ref{14.10.3})$,
\begin{equation}\label{14.10.1}
\sup_{s \in [s_{j}, s_{j + 1}]} |x(s)| \lesssim \frac{1}{\eta_{1}} + \frac{1}{\eta_{1}^{2}} \eta_{1} \eta_{\ast} 2^{-j} \frac{2^{j}}{\eta_{\ast}} \lesssim \frac{1}{\eta_{1}}.
\end{equation}
Since $|s_{n}' - a_{n}| \lesssim \frac{1}{k_{n}^{2}} 2^{2k_{n}}$, setting $T \sim \frac{1}{k_{n}^{2}} 2^{2k_{n}}$ and $\eta_{1}^{-2} T = 2^{\alpha_{d} k}$, as in $(\ref{5.5})$, Theorem $\ref{t5.1}$ implies that for any $i \geq \frac{2k_{n}}{10d}$, letting $a_{n}' = s^{-1}(a_{n})$,
\begin{equation}\label{14.10.2.1}
\aligned
\| P_{\geq i} u \|_{L_{t}^{2} L_{x}^{\frac{2d}{d - 2}}([a_{n}', t_{n}'] \times \mathbb{R}^{d})} + \| P_{\geq i} u \|_{U_{\Delta}^{p}([a_{n}', t_{n}'] \times \mathbb{R}^{d})} \\ \lesssim k_{n} 2^{\frac{\alpha_{d}}{2}((1 + \frac{1}{10d})k - i)} (\frac{1}{T} \int_{a_{n}'}^{t_{n}'} \| \epsilon(t) \|_{L^{2}}^{2} \lambda(t)^{-2} dt)^{1/2} + k_{n}^{2} 2^{-2k_{n}} + T^{-10} \\
= \eta_{1}^{-1 + \frac{1}{10d}} k_{n}^{1 - \frac{1}{5d}} 2^{\frac{k_{n}}{5d}} 2^{-\frac{\alpha_{d}}{2} i} (\int_{a_{n}'}^{t_{n}'} \| \epsilon(t) \|_{L^{2}}^{2} \lambda(t)^{-2} dt)^{1/2} + k_{n}^{2} 2^{-2k_{n}} + T^{-10}.
\endaligned
\end{equation}

In fact, revisiting the proof of Theorem $\ref{t5.1}$, by $(\ref{14.9.1})$, the right hand side of $(\ref{5.8.1})$ can be replaced by $k_{n}^{2} 2^{-2k_{n}} + T^{-10}$, and $(\ref{5.14.2})$ can be replaced by
\begin{equation}
\aligned
\| P_{\geq i} u \|_{U_{\Delta}^{p}(J \times \mathbb{R}^{d})} + \| P_{\geq i} u \|_{L_{t}^{2} L_{x}^{\frac{2d}{d - 2}}(J \times \mathbb{R}^{d})} \lesssim \| \int_{t_{0}}^{t} e^{i(t - \tau) \Delta} P_{\geq i} F(u_{\leq i}) d\tau \|_{U_{\Delta}^{p} \cap L_{t}^{2} L_{x}^{\frac{2d}{d - 2}}} + k_{n}^{2} 2^{-2k_{n}} + T^{-10}.
\endaligned
\end{equation}
Next, the contributions of $(\ref{5.9})$, $(\ref{5.25.4})$--$(\ref{5.25.5})$, and $(\ref{5.25.13})$ can be replaced by
\begin{equation}\label{14.12.0}
 \eta_{1}^{-1 + \frac{1}{10d}} k_{n}^{1 - \frac{1}{5d}} 2^{\frac{k_{n}}{5d}} 2^{-\frac{\alpha_{d}}{2} i} (\int_{a_{n}'}^{t_{n}'} \| \epsilon(t) \|_{L^{2}}^{2} \lambda(t)^{-2} dt)^{1/2} \| \epsilon \|_{L_{t}^{\infty} L_{x}^{2}([a_{n}', t_{n}'] \times \mathbb{R}^{d})}^{2/d} + k_{n}^{2} 2^{-2k_{n}} \| \epsilon \|_{L_{t}^{\infty} L_{x}^{2}([a_{n}', t_{n}'] \times \mathbb{R}^{d})}^{2/d} + T^{-10}.
\end{equation}
Furthermore, by the support properties of $\psi^{2}(x)$, where $\psi$ is as defined in the proof of Theorem $\ref{t5.1}$, the contribution of $(\ref{5.25.1})$ and $(\ref{5.25.3})$ may be controlled by the right hand side of $(\ref{14.12.0})$ plus
\begin{equation}\label{14.12.1}
\aligned
2^{-i} \int  (v, (\nabla (u_{\leq i} - \tilde{Q}) \tilde{Q}^{4/d}))_{L^{2}} dt + 2^{-i} (v, (\nabla \tilde{Q} \frac{(u_{\leq i} - \tilde{Q})}{\tilde{Q}^{1 - 4/d}}))_{L^{2}} dt,
\endaligned
\end{equation}
where $\| v \|_{V_{\Delta}^{2}} = 1$ and $\hat{v}$ is supported on $|\xi| \geq 2^{i}$. Theorem $\ref{t5.0}$ implies that $Q^{4/d}$ is differentiable, and the gradient is smooth and rapidly decreasing. Theorem $\ref{t5.0}$ and $(\ref{1.12})$ also imply that since $Q$ is radially symmetric,
\begin{equation}
\aligned
\nabla \frac{\nabla Q}{Q^{1 - 4/d}} = \nabla_{k} \frac{Q_{r}}{Q^{1 - 4/d}} \frac{x_{j}}{|x|} = (\frac{\delta_{jk}}{|x|} + \frac{x_{j} x_{k}}{|x|^{3}}) \frac{Q_{r}}{Q^{1 - 4/d}} + \frac{Q_{rr}}{Q^{1 - 4/d}} \frac{x_{j} x_{k}}{|x|^{2}} - (1 - \frac{4}{d}) \frac{Q_{r}^{2}}{Q^{2 - 4/d}} \frac{x_{j} x_{k}}{|x|^{2}} \\
\lesssim \frac{Q_{r}}{|x| Q^{1 - 4/d}} + Q^{4/d} + \frac{Q_{r}^{2}}{Q^{2 - 4/d}} \in L_{x}^{d/2}.
\endaligned
\end{equation}
Therefore, by Bernstein's inequality, arguing as in $(\ref{5.25.1})$ and $(\ref{5.25.3})$,
\begin{equation}
(\ref{14.12.1}) \lesssim \eta_{1}^{1/p} 2^{-\frac{i}{p} - 2i} T^{\frac{1}{2000d^{2}}}\| \langle \nabla \rangle^{2} (P_{\leq i} u - \tilde{Q}) \|_{L_{t}^{2} L_{x}^{\frac{2d}{d - 2}}}.
\end{equation}
Therefore, plugging in $(\ref{14.10.2.1})$, we have proved
\begin{equation}\label{14.12.2}
\aligned
\| P_{\geq i} u \|_{U_{\Delta}^{p}([a_{n}', t_{n}'] \times \mathbb{R}^{d})} + \| P_{\geq i} u \|_{L_{t}^{2} L_{x}^{\frac{2d}{d - 2}}([a_{n}', t_{n}'] \times \mathbb{R}^{d})} \lesssim \eta_{1}^{-1 + \frac{1}{10d}} k_{n}^{1 - \frac{1}{5d}} 2^{\frac{k_{n}}{5d}} 2^{-(\alpha_{d} - 1) i} (\int_{a_{n}'}^{t_{n}'} \| \epsilon(t) \|_{L^{2}}^{2} \lambda(t)^{-2} dt)^{1/2} \\ +  \eta_{1}^{-1 + \frac{1}{10d}} k_{n}^{1 - \frac{1}{5d}} 2^{\frac{k_{n}}{5d}} 2^{-\frac{\alpha_{d}}{2} i} (\int_{a_{n}'}^{t_{n}'} \| \epsilon(t) \|_{L^{2}}^{2} \lambda(t)^{-2} dt)^{1/2} \| \epsilon \|_{L_{t}^{\infty} L_{x}^{2}([a_{n}', t_{n}] \times \mathbb{R}^{d})}^{2/d} + k_{n}^{2} 2^{-2k_{n}} + T^{-10},
\endaligned
\end{equation}
and revisiting $(\ref{5.13})$ and $(\ref{5.13.0})$,
\begin{equation}\label{14.12.3}
\aligned
\| P_{\geq i} F(u)\|_{U_{\Delta}^{p'} \cap L_{t}^{2} L_{x}^{\frac{2d}{d + 2}}([a_{n}', t_{n}'] \times \mathbb{R}^{d})} \lesssim \eta_{1}^{-1 + \frac{1}{10d}} k_{n}^{1 - \frac{1}{5d}} 2^{\frac{k_{n}}{5d}} 2^{-(\alpha_{d} - 1)i} (\int_{a_{n}'}^{t_{n}'} \| \epsilon(t) \|_{L^{2}}^{2} \lambda(t)^{-2} dt)^{1/2} \\ +  \eta_{1}^{-1 + \frac{1}{10d}} k_{n}^{1 - \frac{1}{5d}} 2^{\frac{k_{n}}{5d}} 2^{-\frac{\alpha_{d}}{2} i} (\int_{a_{n}'}^{t_{n}'} \| \epsilon(t) \|_{L^{2}}^{2} \lambda(t)^{-2} dt)^{1/2} \| \epsilon \|_{L_{t}^{\infty} L_{x}^{2}([a_{n}', t_{n}] \times \mathbb{R}^{d})}^{2/d} \\ + k_{n}^{2} 2^{-2k_{n}} \| \epsilon \|_{L_{t}^{\infty} L_{x}^{2}}^{4/d} + k_{n}^{2} 2^{-2k_{n}} 2^{-\frac{k_{n}}{10d}}+ T^{-10}.
\endaligned
\end{equation}
Combining $(\ref{14.12.2})$ and $(\ref{14.12.3})$ with the proof of Theorem $\ref{t4.3}$, for $t \in [a_{n}', t_{n}']$, and Theorem $\ref{t11.1}$,
\begin{equation}\label{14.12.3.1}
\aligned
E(P_{\leq k_{n}(1 - \frac{1}{10d})} u(t)) \lesssim 2^{-4k_{n}} 2^{2k_{n}(1 - \frac{1}{10d})} k_{n}^{4} + 2^{(4 - 2 \alpha_{d})k_{n}(1 - \frac{1}{10d})} 2^{\frac{2k_{n}}{5d}} k_{n}^{2} \eta_{1}^{-2} \| \epsilon \|_{L_{t}^{\infty} L_{x}^{2}([a_{n}', t_{n}'] \times \mathbb{R}^{d})} \\
+ \eta_{1}^{-2} k_{n}^{2} 2^{\frac{2k_{n}}{5d}} 2^{-(\alpha_{d} - 2) k_{n}(1 - \frac{1}{10d})} \| \epsilon \|_{L_{t}^{\infty} L_{x}^{2}([a_{n}', t_{n}'] \times \mathbb{R}^{d})}^{1 + 4/d} + T^{-10}.
\endaligned
\end{equation}
Now then, when $d = 3, 4$, by $(\ref{4.66})$, $\| \epsilon(t) \|_{L^{2}}^{2} \lesssim \lambda(t)^{2} E(u(t)) + T^{-10}$, so
\begin{equation}\label{14.12.3.2}
\aligned
E(P_{\leq k_{n}(1 - \frac{1}{10d})} u(t)) \lesssim 2^{-4k_{n}} 2^{2k_{n}(1 - \frac{1}{10d})} k_{n}^{4} + 2^{(8 - 4 \alpha_{d}) k_{n}(1 - \frac{1}{10d})} 2^{\frac{4k_{n}}{5d}} k_{n}^{4} \eta_{1}^{-8} + T^{-10},
\endaligned
\end{equation}
so for $k_{n}$ sufficiently large,
\begin{equation}\label{14.12.3.3}
E(P_{\leq k_{n}(1 - \frac{1}{10d})} u(t)) \lesssim 2^{-4k_{n}} 2^{2k_{n}(1 - \frac{1}{10d})} k_{n}^{4} + T^{-10},
\end{equation}
for any $t \in [a_{n}', t_{n}']$. Furthermore, for any $j \leq \tilde{k}_{n}$, suppose $I(j) = [a_{j}, b_{j}]$. Then, $\lambda(s) \sim \frac{1}{\eta_{1}} 2^{(\tilde{k}_{n} - j)}$. Rescaling so that $\lambda(s) \sim \frac{1}{\eta_{1}}$ on this interval, repeating the calculations obtaining $(\ref{14.12.2})$ and $(\ref{14.12.3})$, and then rescaling back, if $a_{j}' = s^{-1}(a_{j})$ and $b_{j}' = s^{-1}(b_{j})$,
\begin{equation}\label{14.12.4}
\aligned
\| P_{\geq i} u \|_{U_{\Delta}^{p}([a_{j}', b_{j}'] \times \mathbb{R}^{d})} + \| P_{\geq i} u \|_{L_{t}^{2} L_{x}^{\frac{2d}{d - 2}}([a_{j}', b_{j}'] \times \mathbb{R}^{d})} \lesssim \eta_{1}^{-1 + \frac{1}{10d}} k_{n}^{1 - \frac{1}{5d}} 2^{\frac{k_{n}}{5d}} 2^{-(\alpha_{d} - 1) (i + \tilde{k}_{n} - j)} \| \epsilon \|_{L_{t}^{\infty} L_{x}^{2}([a_{j}', b_{j}'] \times \mathbb{R}^{d})}^{1/2} \\ +  \eta_{1}^{-1 + \frac{1}{10d}} k_{n}^{1 - \frac{1}{5d}} 2^{\frac{k_{n}}{5d}} 2^{-\frac{\alpha_{d}}{2} (i + \tilde{k}_{n} - j)} \| \epsilon \|_{L_{t}^{\infty} L_{x}^{2}([a_{j}', b_{j}'] \times \mathbb{R}^{d})}^{1/2+2/d} + \sup \{ 2^{-i}, 2^{-k_{n}(1 - \frac{1}{10d})} \} E(P_{\leq k_{n}(1 - \frac{1}{10d})} u(b_{j}'))^{1/2} + T^{-10},
\endaligned
\end{equation}
and letting $Y$ denote the dual space to $U_{\Delta}^{p} \cap L_{t}^{2} L_{x}^{\frac{2d}{d - 2}}$,
\begin{equation}\label{14.12.5}
\aligned
\| P_{\geq i} F(u)\|_{Y(J \times \mathbb{R}^{d})} \lesssim \eta_{1}^{-1 + \frac{1}{10d}} k_{n}^{1 - \frac{1}{5d}} 2^{\frac{k_{n}}{5d}} 2^{-(\alpha_{d} - 1) (i + \tilde{k}_{n} - j)} \| \epsilon \|_{L_{t}^{\infty} L_{x}^{2}([a_{j}', b_{j}'] \times \mathbb{R}^{d})}^{1/2} \\ +  \eta_{1}^{-1 + \frac{1}{10d}} k_{n}^{1 - \frac{1}{5d}} 2^{\frac{k_{n}}{5d}} 2^{-\frac{\alpha_{d}}{2} (i + \tilde{k}_{n} - j)} \| \epsilon \|_{L_{t}^{\infty} L_{x}^{2}([a_{j}', b_{j}'] \times \mathbb{R}^{d})}^{1/2+2/d} \\ + \sup \{ 2^{-i}, 2^{-k_{n}(1 - \frac{1}{10d})} \} E(P_{\leq k_{n}(1 - \frac{1}{10d})} u(b_{j}'))^{1/2} (\| \epsilon \|_{L_{t}^{\infty} L_{x}^{2}}^{4/d} + 2^{-\frac{k_{n}}{10d}})+ T^{-10}.
\endaligned
\end{equation}

Using $(\ref{4.66})$, for $0 \leq j \leq \tilde{k}_{n}$,
\begin{equation}
\eta_{1} 2^{j - k_{n}} \| \epsilon(t) \|_{L^{2}} \lesssim E(u(t))^{1/2},
\end{equation}
so arguing by induction on $j$, $0 \leq j \leq \tilde{k}_{n}$, and following $(\ref{14.12.3.1})$--$(\ref{14.12.3.3})$,
\begin{equation}\label{14.12.6}
\sup_{t \in [0, t_{n}']} E(u(t)) \lesssim 2^{-4k_{n}} 2^{2k_{n}(1 - \frac{1}{10d})} k_{n}^{4} + T^{-10}.
\end{equation}
Again by $(\ref{4.66})$ and $(\ref{14.11})$, $(\ref{14.12.6})$ implies
\begin{equation}
\| \epsilon(0) \|_{L^{2}} \lesssim \frac{1}{\eta_{1}^{2}} 2^{-2k_{n}} 2^{2k_{n}(1 - \frac{1}{10d})} k_{n}^{4} + \frac{1}{\eta_{1}^{2}} 2^{2k_{n}} T^{-10}.
\end{equation}
Taking $k_{n} \rightarrow \infty$ implies $\epsilon(0) = 0$, so $u$ is a soliton. The case when $\lambda(s)$ has a positive lower bound is easier, as in the two dimensional case.

In dimensions $5 \leq d \leq 8$,
\begin{equation}\label{14.12.7}
\aligned
E(P_{\leq k_{n}(1 - \frac{1}{10d})} u(t)) \lesssim 2^{-4k_{n}} 2^{2k_{n}(1 - \frac{1}{10d})} k_{n}^{4} + 2^{(4 - 2 \alpha_{d})k_{n}(1 - \frac{1}{10d})} 2^{\frac{2k_{n}}{5d}} k_{n}^{2} \eta_{1}^{-2} \| \epsilon \|_{L_{t}^{\infty} L_{x}^{2}([a_{n}', t_{n}'] \times \mathbb{R}^{d})} \\
+ \eta_{1}^{-2} k_{n}^{2} 2^{\frac{2k_{n}}{5d}} 2^{-(\alpha_{d} - 2) k_{n}(1 - \frac{1}{10d})} \| \epsilon \|_{L_{t}^{\infty} L_{x}^{2}([a_{n}', t_{n}'] \times \mathbb{R}^{d})}^{1 + 4/d} + T^{-10},
\endaligned
\end{equation}
for $t \in [a_{n}', t_{n}']$ implies
\begin{equation}
\aligned
E(P_{\leq k_{n}(1 - \frac{1}{10d})} u(t)) \lesssim 2^{-4k_{n}} 2^{2k_{n}(1 - \frac{1}{10d})} k_{n}^{4} + 2^{(8 - 4 \alpha_{d})k_{n}(1 - \frac{1}{10d})} 2^{\frac{4k_{n}}{5d}} k_{n}^{4} \eta_{1}^{-4}  \\
+ (\eta_{1}^{-2} k_{n}^{2} 2^{\frac{2k_{n}}{5d}} 2^{-(\alpha_{d} - 2) k_{n}(1 - \frac{1}{10d})})^{\frac{2d}{d - 4}} + T^{-10}.
\endaligned
\end{equation}
Doing some algebra, since $\alpha_{d} = 3 - \frac{1}{5d}$,
\begin{equation}
\frac{2 k_{n}}{5d} \cdot \frac{2d}{d - 4} - (1 - \frac{3}{10d} + \frac{1}{50d^{2}}) \frac{d}{d - 4} = 2 + \frac{8}{d - 4} - \frac{3}{10(d - 4)} - \frac{1}{50d(d - 4)} - \frac{4}{5d(d - 4)} \geq 2 + \frac{4}{d - 4}.
\end{equation}
Therefore, for $t \in [a_{n}', t_{n}']$,
\begin{equation}
E(P_{\leq k_{n}(1 - \frac{1}{10d})} u(t)) \lesssim 2^{-2k_{n}} 2^{-\frac{4k_{n}}{d - 4}}.
\end{equation}
Then as in $(\ref{14.12.4})$ and $(\ref{14.12.5})$,
\begin{equation}
\sup_{t \in [0, t_{n}']} E(P_{\leq k_{n}(1 - \frac{1}{10d})} u(t)) \lesssim 2^{-2k_{n}} 2^{-\frac{4k_{n}}{d - 4}}.
\end{equation}
Then by by $(\ref{4.66})$ and $(\ref{14.10.2})$, taking $k_{n} \rightarrow \infty$, $\epsilon(0) = 0$.

When $d \geq 9$, recall that $\alpha_{d} = 2 + \frac{8}{d} - \frac{1}{5d}$. However, examining the proof of Theorem $\ref{t5.1}$, it is only the contribution of $(\ref{5.25.13})$ that needs $\alpha_{d} = 2 + \frac{8}{d} - \frac{1}{5d}$, the other terms have the regularity $\frac{3}{2} - \frac{1}{10d}$. Therefore, for $t \in [a_{n}', t_{n}']$, replace $(\ref{14.12.7})$ by
\begin{equation}\label{14.18.0}
\aligned
E(P_{\leq k_{n}(1 - \frac{1}{10d})} u(t)) \lesssim 2^{-4k_{n}} 2^{2k_{n}(1 - \frac{1}{10d})} k_{n}^{4} + 2^{(4 - 2 \alpha_{d})k_{n}(1 - \frac{1}{10d})} 2^{\frac{2k_{n}}{5d}} k_{n}^{2} \eta_{1}^{-2} \| \epsilon \|_{L_{t}^{\infty} L_{x}^{2}([a_{n}', t_{n}'] \times \mathbb{R}^{d})} \\
+ \eta_{1}^{-2} k_{n}^{2} 2^{\frac{2k_{n}}{5d}} 2^{-(1 - \frac{1}{5d}) k_{n}(1 - \frac{1}{10d})} \| \epsilon \|_{L_{t}^{\infty} L_{x}^{2}([a_{n}', t_{n}'] \times \mathbb{R}^{d})}^{1 + 4/d} + \eta_{1}^{-2} k_{n}^{2} 2^{\frac{2k_{n}}{5d}} 2^{-(\frac{8}{d} - \frac{1}{5d}) k_{n}(1 - \frac{1}{10d})} \| \epsilon \|_{L_{t}^{\infty} L_{x}^{2}}^{1 + 8/d} + T^{-10}.
\endaligned
\end{equation}
Then
\begin{equation}
\eta_{1}^{-2} k_{n}^{2} 2^{\frac{2k_{n}}{5d}} 2^{-(\frac{8}{d} - \frac{1}{5d}) k_{n}(1 - \frac{1}{10d})} \| \epsilon \|_{L_{t}^{\infty} L_{x}^{2}}^{1 + 8/d} \lesssim \| \epsilon \|_{L_{t}^{\infty} L_{x}^{2}}^{2} + (\eta_{1}^{-2} k_{n}^{2} 2^{\frac{2k_{n}}{5d}} 2^{-(\frac{8}{d} - \frac{1}{5d}) k_{n}(1 - \frac{1}{10d})})^{\frac{2d}{d - 8}}.
\end{equation}
Doing some algebra, for $8 \leq d \leq 15$,
\begin{equation}
(\frac{8}{d} - \frac{3}{5d} - \frac{7}{10d^{2}}) \cdot \frac{2d}{d - 8} \geq 2 + \frac{3}{70}.
\end{equation}
Therefore, for $8 \leq d \leq 15$, for $k_{n}$ sufficiently large,
\begin{equation}
 (\eta_{1}^{-2} k_{n}^{2} 2^{\frac{2k_{n}}{5d}} 2^{-(\frac{8}{d} - \frac{1}{5d}) k_{n}(1 - \frac{1}{10d})})^{\frac{2d}{d - 8}} \lesssim 2^{-2k_{n}} 2^{-\frac{3}{70} k_{n}}.
\end{equation}
Once again, taking $k_{n} \rightarrow \infty$ proves $\epsilon(0) = 0$.

Dimensions $d \geq 16$ remain unresolved.
\end{proof}

Now turn to a finite time blowup solution. As in dimension one, $\sup(I) < \infty$ implies that $u$ is a pseudoconformal transformation of the soliton. Suppose without loss of generality that $\sup(I) = 0$, and
\begin{equation}\label{14.18}
\sup_{-1 < t < 0} \| \epsilon(t) \|_{L^{2}} \leq \eta_{\ast}.
\end{equation}
Then decomposing $u$,
\begin{equation}\label{14.19}
u(t,x) = \frac{e^{-i \gamma(t)} e^{-ix \cdot \frac{\xi(t)}{\lambda(t)}}}{\lambda(t)^{d/2}} Q(\frac{x - x(t)}{\lambda(t)}) + \frac{e^{-i \gamma(t)} e^{-i x \cdot \frac{\xi(t)}{\lambda(t)}}}{\lambda(t)^{d/2}} \epsilon(t,\frac{x - x(t)}{\lambda(t)}).
\end{equation}
Then apply the pseudoconformal transformation to $u(t,x)$. For $-\infty < t < -1$, let
\begin{equation}\label{14.20}
\aligned
v(t,x) = \frac{1}{t^{d/2}} \overline{u(\frac{1}{t}, \frac{x}{t})} e^{i |x|^{2}/4t} = \frac{1}{t^{d/2}} \frac{e^{i \gamma(1/t)} e^{i \frac{x}{t} \cdot \frac{\xi(\frac{1}{t})}{\lambda(\frac{1}{t})}}}{\lambda(1/t)^{d/2}} Q(\frac{x - t x(\frac{1}{t})}{t \lambda(1/t)}) e^{i |x|^{2}/4t} \\ + \frac{1}{t^{d/2}} \frac{e^{i \gamma(1/t)} e^{i \frac{x}{t} \cdot \frac{\xi(\frac{1}{t})}{\lambda(\frac{1}{t})}}}{\lambda(1/t)^{d/2}} \overline{\epsilon(\frac{1}{t}, \frac{x - t x(\frac{1}{t})}{t \lambda(1/t)})} e^{i |x|^{2}/4t}.
\endaligned
\end{equation}
Since the $L^{2}$ norm is preserved by the pseudoconformal transformation,
\begin{equation}\label{14.21}
\aligned
\lim_{t \searrow -\infty} \| \frac{1}{t^{d/2}} \frac{e^{i \gamma(1/t)} e^{i \frac{x}{t} \cdot \frac{\xi(\frac{1}{t})}{\lambda(\frac{1}{t})}}}{\lambda(1/t)^{d/2}} \overline{\epsilon(\frac{1}{t}, \frac{x - t x(\frac{1}{t})}{t \lambda(1/t)})} e^{i |x|^{2}/4t} \|_{L^{2}} = 0, \qquad \text{and} \\ 
\qquad \sup_{-\infty < t < -1}  \| \frac{1}{t^{d/2}} \frac{e^{i \gamma(1/t)} e^{i \frac{x}{t} \cdot \frac{\xi(\frac{1}{t})}{\lambda(\frac{1}{t})}}}{\lambda(1/t)^{d/2}} \overline{\epsilon(\frac{1}{t}, \frac{x - t x(\frac{1}{t})}{t \lambda(1/t)})} e^{i x^{2}/4t} \|_{L^{2}} \leq \eta_{\ast}.
\endaligned
\end{equation}

Decompose
\begin{equation}
\frac{|x|^{2}}{4t} = \frac{|x - t x(\frac{1}{t})|^{2}}{4t} + \frac{x \cdot x(\frac{1}{t})}{2} - \frac{t}{4} |x(\frac{1}{t})|^{2}.
\end{equation}
Since
\begin{equation}\label{14.22}
 \frac{1}{t^{d/2}} \frac{e^{i \gamma(1/t)} e^{i \frac{x}{t} \cdot \frac{\xi(\frac{1}{t})}{\lambda(\frac{1}{t})}} e^{i x \cdot \frac{x(\frac{1}{t})}{2}} e^{i \frac{t}{4} |x(\frac{1}{t})|^{2}}}{\lambda(1/t)^{d/2}} Q(\frac{x - t x(\frac{1}{t})}{t \lambda(1/t)})
\end{equation}
is in the form of $\frac{e^{i \tilde{\gamma}(t)} e^{ix \cdot \frac{\tilde{\xi(t)}}{\tilde{\lambda(t)}}}}{\tilde{\lambda}(t)^{d/2}} Q(\frac{x - \tilde{x}(t)}{\tilde{\lambda}(t)})$, it only remains to estimate
\begin{equation}\label{14.23}
 \| \frac{1}{t^{d/2}} \frac{e^{i \gamma(1/t)} e^{i\frac{x}{t} \cdot \frac{\xi(t)}{\lambda(t)}}}{\lambda(1/t)^{d/2}} Q(\frac{x - t x(\frac{1}{t})}{t \lambda(1/t)}) (e^{i |x - t x(\frac{1}{t})|^{2}/4t} - 1) \|_{L^{2}}.
\end{equation}

As in $(\ref{14.5})$, for any $k \geq 0$, $\lambda(s) \sim 2^{-k}$ for all $s \in I(k)$. Furthermore, by  $(\ref{3.23})$, $\| \epsilon(t) \|_{L^{2}} \rightarrow 0$ as $t \nearrow 0$ implies that there exists a sequence $c_{k} \nearrow \infty$ such that
\begin{equation}\label{14.24}
|I(k)| \geq c_{k}, \qquad \text{for all} \qquad k \geq 0.
\end{equation}
Then by $(\ref{3.18})$, there exists $r(t) \searrow 0$ as $t \nearrow 0$ such that
\begin{equation}\label{14.25}
\lambda(t) \leq t^{1/2} r(t), \qquad \text{so} \qquad \lambda(1/t) \leq t^{-1/2} r(1/t).
\end{equation}
Therefore, since $Q$ is rapidly decreasing,
\begin{equation}\label{14.26}
\lim_{t \searrow -\infty} \| \frac{1}{t^{d/2} \lambda(1/t)^{d/2}} Q(\frac{x - t x(\frac{1}{t})}{t \lambda(1/t)}) \frac{|x - t x(\frac{1}{t})|^{2}}{4t} \|_{L^{2}} = 0,
\end{equation}
as well as
\begin{equation}\label{14.27}
\lim_{t \searrow -\infty} \| \frac{1}{t^{d/2} \lambda(1/t)^{d/2}} Q(\frac{x - t x(\frac{1}{t})}{t \lambda(1/t)}) (e^{i |x - t x(\frac{1}{t})|^{2}/4t} - 1)\|_{L^{2}} = 0,
\end{equation}
Therefore, $v$ is a solution that blows up backward in time at $\inf(I) = -\infty$ and $v$ satisfies the conditions of Theorem $\ref{t3.1}$ on $(-\infty, t_{0}]$ for some $t_{0} \in \mathbb{R}$. Therefore, by time reversal symmetry and Theorem $\ref{t14.0}$, $v$ must be a soliton. Therefore, $u$ is the pseudoconformal transformation of a soliton.

\section{A Liouville result}
Recall the Liouville theorem for the generalized KdV equation from \cite{martel2000liouville}.
\begin{theorem}\label{t15.1}
Let $u_{0} \in H^{1}(\mathbb{R})$ and let $\alpha = \| u_{0} - Q \|_{H^{1}}$. Suppose that the solution to the focusing, mass-critical generalized KdV problem,
\begin{equation}
u_{t} + \partial_{x}(u_{xx} + u^{5}) = 0, \qquad u(0,x) = u_{0},
\end{equation}
 is global in time, and for all $t \in \mathbb{R}$, and assume that for some $c_{1}, c_{2} > 0$,
\begin{equation}
c_{1} \leq \| u(t) \|_{H^{1}} \leq c_{2}.
\end{equation}
Also suppose that there exists $x(t)$ such that
\begin{equation}
v(t, x) = u(t, x + x(t)),
\end{equation}
satisfies
\begin{equation}
\forall \epsilon_{0}, \qquad \exists R_{0}(\epsilon_{0}) > 0, \qquad \forall t \in \mathbb{R}, \qquad \int_{|x| > R_{0}} v(t,x)^{2} dx \leq \epsilon_{0}.
\end{equation}
There exists $\alpha_{0} > 0$ such that if $\alpha < \alpha_{0}$, there exists $\lambda_{0}$, $x_{0}$ such that
\begin{equation}
u(t, x) = \lambda_{0}^{1/2} Q(\lambda_{0}(x - x_{0}) - \lambda_{0}^{3} t).
\end{equation}
\end{theorem}

Using Theorems $\ref{t1.2}$ and $\ref{t1.3}$, we can prove such a result for the nonlinear Schr{\"o}dinger equation, without requiring the initial data to lie close to the soliton.
\begin{theorem}\label{t15.2}
Let $u_{0} \in H^{1}(\mathbb{R}^{d})$ and suppose $\| u_{0} \|_{L^{2}} \leq \| Q \|_{L^{2}} + \alpha$, for some $0 < \alpha < \alpha_{0} \ll 1$. Suppose that a solution $u(t)$ to $(\ref{1.1})$ is defined for all $t \in \mathbb{R}$ and for some $c_{1}, c_{2} > 0$,
\begin{equation}\label{15.1}
c_{1} \leq \| u(t) \|_{H^{1}} \leq c_{2}, \qquad \text{for all} \qquad t \in \mathbb{R}.
\end{equation}
Also suppose that for all $t \in \mathbb{R}$ there exists $x(t) \in \mathbb{R}^{d}$ such that
\begin{equation}\label{15.2}
v(t, x) = u(t, x + x(t)),
\end{equation}
satisfies
\begin{equation}\label{15.3}
\forall \epsilon_{0} > 0, \qquad \exists R_{0} > 0, \qquad \forall t \in \mathbb{R}, \qquad \int_{|x| > R_{0}} |v(t, x)|^{2} dx \leq \epsilon_{0}.
\end{equation}
Then there exists $\alpha_{0} > 0$ sufficiently small such that if $\alpha < \alpha_{0}$, $u$ should be in the form $(\ref{1.2})$.
\end{theorem}
\begin{proof}
By Theorem $\ref{t1.3}$ and scattering for $\| u_{0} \|_{L^{2}} < \| Q \|_{L^{2}}$, it suffices to check
\begin{equation}\label{15.4}
\| Q \|_{L^{2}} < \| u \|_{L^{2}} + \| Q \|_{L^{2}} + \alpha.
\end{equation}

By \cite{fan20182}, \cite{dodson20202}, and \cite{dodson20212}, there exists a sequence $t_{n} \rightarrow +\infty$ and a sequence $\gamma_{\ast, n} \in \mathbb{R}$, $\xi_{\ast, n} \in \mathbb{R}^{d}$, $\lambda_{\ast, n} \in (0, \infty)$, $x_{\ast, n} \in \mathbb{R}^{d}$, such that
\begin{equation}\label{15.5}
e^{i \gamma_{\ast, n}} e^{ix \cdot \xi_{\ast, n}} \lambda_{\ast, n}^{d/2} u(t_{n}, \lambda_{\ast, n} x + x_{\ast, n}) \rightharpoonup Q, \qquad \text{weakly in} \qquad L^{2}.
\end{equation}
By $(\ref{15.1})$ and $(\ref{15.3})$, this can be upgraded to convergence in $L^{2}$, which implies $\| u \|_{L^{2}} = \| Q \|_{L^{2}}$, which proves the theorem.
\end{proof}

In fact, it is possible to say more. Suppose $u_{0}$ does not lie in $H^{1}$, but only in $L^{2}$, but we have uniform bounds on the length of the intervals for which local well-posedness of $(\ref{1.1})$ holds. The Liouville theorem still holds.
\begin{theorem}\label{t15.3}
Let $u_{0} \in L^{2}(\mathbb{R}^{d})$ and suppose $\| u_{0} \|_{L^{2}} \leq \| Q \|_{L^{2}} + \alpha$, for some $0 < \alpha < \alpha_{0} \ll 1$. Suppose that a solution $u(t)$ to $(\ref{1.1})$ is defined for all $t \in \mathbb{R}$ and for some $c_{1}, c_{2} > 0$,
\begin{equation}\label{15.6}
\aligned
\sup_{t_{0} \in \mathbb{R}} \| u \|_{L_{t,x}^{\frac{2(d + 2)}{d}}([t_{0}, t_{0} + 1] \times \mathbb{R}^{d})} \leq c_{2}, \\
\inf_{t_{0} \in \mathbb{R}} \| u \|_{L_{t,x}^{\frac{2(d + 2)}{d}}([t_{0}, t_{0} + 1] \times \mathbb{R}^{d})} \geq c_{1}.
\endaligned
\end{equation}
Also suppose that for all $t \in \mathbb{R}$ there exists $x(t) \in \mathbb{R}^{d}$ such that
\begin{equation}\label{15.7}
v(t, x) = u(t, x + x(t)),
\end{equation}
satisfies
\begin{equation}\label{15.8}
\forall \epsilon_{0} > 0, \qquad \exists R_{0} > 0, \qquad \forall t \in \mathbb{R}, \qquad \int_{|x| > R_{0}} |v(t, x)|^{2} dx \leq \epsilon_{0}.
\end{equation}
Then there exists $\alpha_{0} > 0$ sufficiently small such that if $\alpha < \alpha_{0}$, $u$ should be in the form $(\ref{1.2})$.
\end{theorem}
\begin{remark}
Note that $(\ref{15.1})$ and $(\ref{15.3})$ imply $(\ref{15.6})$.
\end{remark}
\begin{proof}
Once again, it suffices to only consider initial data that satisfy $(\ref{15.4})$. Once again, $(\ref{15.5})$ holds for a sequence $t_{n} \nearrow \infty$.

We claim that for any $N$,
\begin{equation}\label{15.9}
\| P_{\leq N} (e^{i \gamma_{\ast, n}} e^{ix \cdot \xi_{\ast, n}} \lambda_{\ast, n}^{d/2} u(t_{n}, \lambda_{\ast, n} x + x_{\ast, n}) - Q(x)) \|_{L^{2}} \rightarrow 0.
\end{equation}
Otherwise, by $(\ref{15.8})$,
\begin{equation}\label{15.10}
P_{\leq N} (e^{i \gamma_{\ast, n}} e^{ix \cdot \xi_{\ast, n}} \lambda_{\ast, n}^{d/2} u(t_{n}, \lambda_{\ast, n} x + x_{\ast, n}) - Q(x)) \rightharpoonup f \neq 0, \qquad \text{weakly in} \qquad L^{2},
\end{equation}
which would contradict $(\ref{15.5})$. Therefore, there exists a sequence $N_{n} \nearrow \infty$ such that
\begin{equation}\label{15.11}
\| P_{\leq N_{n}} (e^{i \gamma_{\ast, n}} e^{ix \cdot \xi_{\ast, n}} \lambda_{\ast, n}^{d/2} u(t_{n}, \lambda_{\ast, n} x + x_{\ast, n}) - Q(x)) \|_{L^{2}} \rightarrow 0.
\end{equation}

Next, since $(\ref{15.6})$ implies that $u$ blows up in both time directions, for $\alpha_{0}$ sufficiently small, Theorem $\ref{t1.3}$ combined with standard perturbative arguments implies that $u$ is close to a member of the soliton family $(\ref{1.3.1})$ for all $t \in \mathbb{R}$. Furthermore, $(\ref{15.6})$ implies that $\lambda(t) \sim 1$ for all $t \in \mathbb{R}$. Therefore, let $u_{n}(t)$ be the solution to $(\ref{1.1})$ with initial data of the form
\begin{equation}\label{15.12}
u_{n}(0) = e^{i \gamma_{\ast, n}} e^{ix \cdot \xi_{\ast, n}} \lambda_{\ast, n}^{d/2} u(t_{n}, \lambda_{\ast, n} x + x_{\ast, n}).
\end{equation}
Then by $(\ref{15.11})$ and standard perturbative arguments, there exists a sequence $T_{n} \nearrow \infty$ such that
\begin{equation}\label{15.13}
u_{n}(t) = e^{it} Q + v_{n}(t) + r_{n}(t), \qquad \text{where} \qquad \| v_{n} \|_{L_{t,x}^{\frac{2(d + 2)}{d}}([0, T_{n}] \times \mathbb{R}^{d})} \lesssim 1, \qquad \| r_{n} \|_{L_{t,x}^{\frac{2(d + 2)}{d}}([0, T_{n}] \times \mathbb{R}^{d})} \searrow 0.
\end{equation}
However, by H{\"o}lder's inequality and $(\ref{15.8})$, for $n$ sufficiently large,
\begin{equation}\label{15.14}
\| v_{n}(t) \|_{L^{\frac{2(d + 2)}{d}}(\mathbb{R}^{d})} \gtrsim (\| u_{0} \|_{L^{2}} - \| Q \|_{L^{2}}),
\end{equation}
for any $t \in [0, T_{n}]$, with lower bound independent of $n$. This gives a contradiction for $n$ sufficiently large, when $\| u_{0} \|_{L^{2}} > \| Q \|_{L^{2}}$. When $\| u_{0} \|_{L^{2}} = \| Q \|_{L^{2}}$, apply Theorem $\ref{t1.3}$.
\end{proof}

\section*{Acknowledgement}
The author was partially supported by NSF Grant DMS--$1764358$. The author was also greatly helped by many stimulating discussions with Frank Merle at Cergy-Pontoise and University of Chicago, as well as his constant encouragement to pursue this problem. The author would also like to recognize the many helpful discussions that he had with Svetlana Roudenko and Anudeep Kumar Arora, both at George Washington University and Florida International University.

The author was also greatly assisted by discussions of the nonlinear Schr{\"o}dinger equation with his PhD student, Dr. Zehua Zhao. The author also gratefully acknowledges discussions with Chenjie Fan and Jason Murphy at Oberwolfach, and Robin Neumayer at the Institute for Advanced Study.

The author would also like to thank Jonas L{\"u}hrmann and Cristian Gavrus for many helpful discussions on the related mass-critical generalized KdV equation.

\bibliography{biblio}
\bibliographystyle{alpha}

\end{document}